\documentclass[reqno,11pt]{amsart}
\usepackage{a4wide}
\usepackage{amsmath,amssymb,amsthm,mathtools}
\usepackage{microtype}
\usepackage{enumitem}
\usepackage[colorlinks,linkcolor=blue,anchorcolor=red,citecolor=blue]{hyperref}
\usepackage{cite}
\usepackage{cleveref}
\usepackage{cite}
\allowdisplaybreaks
\usepackage{marginnote}

\numberwithin{equation}{section}
\usepackage{color}

\newtheorem{theorem}{Theorem}[section]
\newtheorem{proposition}[theorem]{Proposition}
\newtheorem{lemma}[theorem]{Lemma}

\theoremstyle{definition}

\newtheorem{remark}[theorem]{Remark}

\usepackage{mathrsfs,needspace,xcolor}
\newcommand{\R}{\mathbb R}

\newcommand{\dd}{\,\mathrm d}

\newcommand{\Peff}{P_{\mathrm{eff}}}
\DeclareMathOperator{\spanop}{span}
\newcommand{\la}{\lambda_D}\newcommand{\cv}{c_v}
\newcommand{\intR}{\int_{\R}}\newcommand{\dt}{\partial_t-\sigma\partial_\xi}
\newcommand{\tu}{\widetilde u}
\newcommand{\tv}{\widetilde v}\newcommand{\tz}{\widetilde\theta}
\newcommand{\tp}{\widetilde\phi}\newcommand{\pp}{P}
\newcommand{\E}{\mathcal E}\newcommand{\Gs}{\mathcal G^S}
\newcommand{\D}{\mathbf D}\newcommand{\G}{\mathbf G}
\newcommand{\B}{\mathbf B}\newcommand{\Y}{\mathbf Y}
\newcommand{\Pp}{\mathbf P}\newcommand{\Rp}{\mathcal R_\phi}
\newcommand{\norm}[1]{\left\|#1\right\|}

\begin{document}
	
\title[Shock profiles for the NSFP system under the Boltzmann relation]{Shock profiles for the Navier--Stokes--Fourier--Poisson system under the Boltzmann relation}
	
\author[R.-J. Duan]{Renjun Duan}
\address[R.-J. Duan]{Department of Mathematics, The Chinese University of Hong Kong, Shatin, Hong Kong}
\email{rjduan@math.cuhk.edu.hk}
	
\author[R.~Li]{Rui Li}
\address[R.~Li]{Department of Mathematics, The Chinese University of Hong Kong, Shatin, Hong Kong, China.}
\email{ruili001@cuhk.edu.hk}

\begin{abstract}
The one-dimensional compressible Navier--Stokes--Fourier--Poisson system under the Boltzmann relation is studied for the existence of small-amplitude viscous shock profiles connecting the Rankine--Hugoniot states of the quasineutral Euler system with effective pressure $P_{\mathrm{eff}}=\rho(\theta+1)$. This profile is unique up to translation, of Lax type, and orbitally stable in time without a zero-mass assumption on the perturbation. The proof is based on a center-manifold reduction of the traveling-wave ODE and on the method of $a$-contraction with shifts, together with a relative-entropy functional augmented by electric energy and a sharp estimate of the viscous flux associated with $P_{\mathrm{eff}}$.
\end{abstract}
	
	
\date{\today}
\maketitle
\thispagestyle{empty}
	
\tableofcontents
	
\section{Introduction}
	In the paper, we are concerned with the full compressible Navier--Stokes--Fourier--Poisson (NSFP) system for ion dynamics in plasmas under the Boltzmann relation in the one-dimensional whole line. For such system in Eulerian coordinates, the unknowns are the ion density $\rho=\rho(t,x)>0$, velocity $u=u(t,x)$, temperature $\theta=\theta(t,x)>0$ and electric potential $\phi=\phi(t,x)$ with $t\geq 0$ and $x\in \R$, that are governed by
	\begin{equation}
		\label{eq:NSFP}
		\begin{cases}
			\rho_t+(\rho u)_x=0,\\
			(\rho u)_t+(\rho u^2+p)_x-\rho\phi_x=(\mu u_x)_x,\\
			\bigl(\rho(\tfrac12 u^2+e)\bigr)_t
			+\bigl(u(\rho(\tfrac12 u^2+e)+p)\bigr)_x
			-\rho u\phi_x
			=(\mu u u_x+\kappa\theta_x)_x,\\
			\la^2\phi_{xx}=\rho-e^{-\phi}.
		\end{cases}
	\end{equation}
	Here $p=\rho\theta$ and $e=c_v\theta$ are the pressure and specific internal energy for the ideal ionic fluid, with $c_v=1/(\gamma-1)$ and $\gamma>1$. 
	The viscosity $\mu$, heat conductivity $\kappa$ and Debye length $\la$ are taken as positive constants.  
	
	A traveling shock profile is a smooth solution of the NSFP system \eqref{eq:NSFP} of the form $(\rho,u,\theta,\phi)(\zeta)$ with $\zeta=x-st$ and shock speed $s$, which connects two distinct far-field constant states
	\begin{align}
		\label{eq:far}
		\lim_{\zeta\to\pm\infty}(\rho,u,\theta,\phi)(\zeta)=(\rho_\pm,u_\pm,\theta_\pm,\phi_\pm),\qquad
		\lim_{\zeta\to\pm\infty}\phi'(\zeta)=0,
	\end{align}
	with $\rho_\pm>0$, $\theta_\pm>0$, under the quasineutral far-field assumption $\rho_\pm=e^{-\phi_\pm}$.  A Galilean translation and a density scaling reduce the left state to
	\begin{align}
		\label{eq:leftE}
		(\rho_-,u_-,\theta_-,\phi_-)=(1,0,\theta_0,0),\qquad \theta_0>0.
	\end{align}
	Dropping those second-order terms with constants \(\mu,\,\kappa,\,\la^2\) in \eqref{eq:NSFP} and imposing the quasineutral condition \(\rho=e^{-\phi}\) (hence \(\phi=-\log\rho\) and \(-\rho\phi_x=\rho_x\)) yields the quasineutral Euler system
	\begin{equation}
		\label{eq:QNEuler}
		\begin{cases}
			\rho_t+(\rho u)_x=0,\\
			(\rho u)_t+\bigl(\rho u^2+\rho(\theta+1)\bigr)_x=0,\\
			\theta_t+u\theta_x+(\gamma-1)\theta u_x=0.
		\end{cases}
	\end{equation}
	Equivalently, in primitive variables \(U=(\rho,u,\theta)\),
	\[
	U_t+\mathcal{A}(U)U_x=0,\qquad
	\mathcal{A}(U)=\begin{pmatrix}
		u & \rho & 0\\
		(\theta+1)/\rho & u & 1\\
		0 & (\gamma-1)\theta & u
	\end{pmatrix},
	\]
	with eigenvalues \(\lambda_1=u-c\), \(\lambda_2=u\), \(\lambda_3=u+c\) and ion-acoustic speed \(c=\sqrt{\gamma\theta+1}\). At the left state \eqref{eq:leftE} one has
	\begin{equation}\notag
		c_0:=\sqrt{\gamma\theta_0+1}.
	\end{equation}
	Throughout the paper, we consider the compressive $3$-shock of small strength, namely $0<s<c_0$ with $\rho_+<1$ and $\rho_+$ sufficiently close to $1$.
	
	The propagation of ion-acoustic shocks in a weakly collisional plasma is a classical problem in kinetic and fluid modeling. When the electron inertia is negligible, the electron density is determined by the electrostatic potential through the Boltzmann relation $n_e=e^{-\phi}$, and the ions may be described either kinetically by a Vlasov--Poisson--Boltzmann system or, in the fluid regime, by a Navier--Stokes--Poisson system.  In the isothermal setting, following the approach in \cite{MP}, Duan, Liu and Zhang \cite{DLZ} established the existence of small-amplitude smooth traveling waves by reducing the profile equations to a three-dimensional ODE and applying the center-manifold theorem, and proved nonlinear stability for strictly positive ion temperature under a zero-mass condition.  Kang, Kwon and Shim \cite{KKS} later obtained asymptotic stability of the same isothermal profiles without zero mass, by the method of $a$-contraction with time-dependent shifts \cite{Vasseur}. More recently, Shim \cite{ShimComposite} established the asymptotic
	stability of composite waves consisting of a shock profile and a
	rarefaction wave for the isothermal Navier--Stokes--Poisson system.
	
	Viscous plasma shocks belong to the broader study of how dissipative fluid equations resolve and stabilize shocks of hyperbolic conservation laws. The classical framework goes back to Riemann's analysis of nonlinear gas waves \cite{Riemann} and Lax's theory of hyperbolic systems \cite{Lax}. For systems regularized by artificial viscosity, Bianchini and Bressan \cite{BB} established the existence of the vanishing-viscosity limit for initial data of small total variation. The monograph of Liu \cite{LiuBook} and the survey of Matsumura \cite{MatsumuraSurvey} describe the shock-wave theory and its relation to the large-time behavior of compressible viscous fluids.
	
	Regarding the stability of viscous shocks, Il'in and Oleinik \cite{IO} obtained early large-time results for scalar equations, while Sattinger \cite{Sattinger} developed a stability theory for traveling waves of nonlinear parabolic systems. For systems of conservation laws, Matsumura and Nishihara \cite{MN} treated compressible viscous gas, and Goodman \cite{Go} treated conservation laws with artificial diffusion under the zero mass condition. The subsequent works of Liu \cite{Liu85}, Szepessy and Xin \cite{SX}, and Liu and Zeng \cite{LZ} eliminated the need for the zero-mass condition by introducing diffusion waves and constructing approximate Green's functions to derive pointwise estimates.  Mascia and Zumbrun \cite{MZ} established the nonlinear orbital stability
	of small-amplitude shock profiles for a class of dissipative symmetric hyperbolic--parabolic systems, including the compressible
	Navier--Stokes equations.
	
	
	A complementary approach uses relative entropy to compare solutions directly, following the stability framework of Dafermos \cite{Dafermos79,Dafermos96} and DiPerna \cite{DiPerna}. For inviscid scalar shocks, L\'eger \cite{Leger} obtained an $L^2$ stability estimate after allowing a translation of the reference shock. 
	Serre and Vasseur \cite{SVcontraction} identified the restrictions on unweighted $L^2$-type contraction for systems. Kang and Vasseur \cite{KVcriteria} developed criteria for contraction with different weights on the two sides of an entropic discontinuity; the application to extremal shocks is also discussed by Vasseur \cite{Vasseur}. 
	A broader account of the relative-entropy method is given by Serre and Vasseur \cite{SVsurvey}.
	
	For viscous equations, Kang and Vasseur \cite{KVscalar} established contraction with shifts for scalar shock waves, and Kang \cite{KangScalar} treated strictly convex scalar fluxes. Kang, Vasseur and Wang \cite{KVWscalar} obtained contraction for multidimensional scalar planar shocks. The same approach also applies to the viscous traveling waves of a hyperbolic-parabolic chemotaxis system, as shown by Choi, Kang, Kwon and Vasseur \cite{CKKV}. For the barotropic Navier--Stokes equations, Kang and Vasseur \cite{KV-JEMS} proved contraction for large perturbations of small shocks using the effective-velocity formulation associated with the Bresch--Desjardins entropy \cite{BD}. Estimates uniform in viscosity then led to uniqueness and stability of entropy shocks of the isentropic Euler system within a class of physical inviscid limits \cite{KVinviscid}, and to well-posedness of the two-shock Riemann problem in that class \cite{KVtwo}.
	
	For large-time asymptotics, an advantage of contraction with shifts is that it can be combined with direct energy estimates for other wave patterns. Kang, Vasseur and Wang \cite{KVWcomposite} proved stability of a composite wave consisting of a viscous shock and a rarefaction wave for the barotropic Navier--Stokes equations. 
	For the full Navier--Stokes--Fourier equations, Kang, Vasseur and Wang \cite{KVW} treated the generic combination of a viscous shock, a viscous contact wave and a rarefaction. These developments provide the fluid stability background for the additional Poisson coupling to be considered in the present paper.
	
Our goal of this paper is to treat the full one-dimensional compressible NSFP system \eqref{eq:NSFP}. We expect to prove the existence of small-amplitude shock profiles, and then their nonlinear dynamical stability in Lagrangian coordinates, without a zero-mass assumption on the initial perturbation. First of all, to deal with the existence of shock profile, a key observation is that using the Poisson equation \eqref{eq:NSFP}$_4$, one can rewrite the NSF fluid part of \eqref{eq:NSFP} as the conservative form
	\begin{equation}
		\label{eq:NSFPcons}
		\begin{cases}
			\rho_t+(\rho u)_x=0,\\[6pt]
			(\rho u)_t
			+\bigl(\rho u^2+p-\dfrac12\lambda_D^2\phi_x^2+e^{-\phi}\bigr)_x
			=(\mu u_x)_x,\\[6pt]
			\mathscr E_t+
			\left[
			u\left(\mathscr E+\rho\theta+e^{-\phi}
			+\frac{\la^2}{2}\phi_x^2\right)
			+\la^2\phi_x\phi_t
			\right]_x
			=(\mu uu_x+\kappa\theta_x)_x.
		\end{cases}
	\end{equation}
	where
	\begin{equation*}
		\mathscr E
		=\rho\left(\frac{u^2}{2}+c_v\theta-\phi\right)
		+\rho-e^{-\phi}-\frac{\la^2}{2}\phi_x^2,
	\end{equation*}
see the derivation in Sec.~\ref{sec:RH} later.
	This provides a possibility of constructing the shock profile for \eqref{eq:NSFP}. Indeed, as mentioned before, the shock profile is of the form $(\rho,u,\theta,\phi)(\zeta)$ with $\zeta=x-st$ and substituting it into \eqref{eq:NSFP} gives 
	\begin{equation}\label{eq:profile-second-order.intro}
		\left\{\begin{aligned}
			&-s\rho'+(\rho u)'=0,\\
			&-s(\rho u)'+(\rho u^2+\rho\theta)'-\rho W=\mu u'',\\
			&-s\bigl[\rho(\tfrac12u^2+c_v\theta)\bigr]'
			+\bigl[u\{\rho(\tfrac12u^2+c_v\theta)+\rho\theta\}\bigr]'
			-\rho uW=(\mu uu'+\kappa\theta')',\\
			&\phi':=W,\qquad \la^2W'=\rho-e^{-\phi},
		\end{aligned}\right.
	\end{equation}
	with 
	$$
	\rho W=\bigl(\tfrac12\la^2 W^2-e^{-\phi}\bigr)'
	$$
	$$\rho u W=\bigl(-s\phi+s(\tfrac12\la^2 W^2-e^{-\phi})\bigr)'
	=\rho(u-s)W+s\rho W=-s\phi'+s\bigl(\tfrac12\la^2 W^2-e^{-\phi}\bigr)',
	$$	
	where from \eqref{eq:profile-second-order.intro}$_1$ together with \eqref{eq:leftE} we have used the fact that
	\begin{equation}\label{eqn.ru}
		\rho(u-s)=-s.
	\end{equation}
	Note that \eqref{eqn.ru} above gives $u=s(1-\frac1\rho)$. Therefore, to solve \eqref{eq:profile-second-order.intro} it suffices to solve $Y=(\rho,\theta,\phi,W)$ which satisfies
	\begin{equation}\label{eq:vector-field.intro}
		Y'=\mathscr F(Y;s):=\begin{pmatrix}
			\dfrac{\rho^2}{\mu s}F_1(Y;s)\\[2mm]
			\dfrac1\kappa\bigl(G(Y;s)-uF_1(Y;s)\bigr)\\[2mm]
			W\\ \la^{-2}(\rho-e^{-\phi})
		\end{pmatrix},
	\end{equation}
	where we have denoted
	\begin{align*}
		F_1(\rho,\theta,\phi,W;s)
		=&-\frac{s^2}{\rho}(\rho-1)+\rho\theta-\theta_0
		+e^{-\phi}-1-\frac{\la^2}{2}W^2,\\
		G(\rho,\theta,\phi,W;s)
		=&-s\left[\rho(\tfrac12u^2+c_v\theta)-c_v\theta_0\right]
		+u\left[\rho(\tfrac12u^2+c_v\theta)+\rho\theta\right]\notag\\
		&+s\phi-s\left(\frac{\la^2}{2}W^2-e^{-\phi}+1\right),
		\qquad u=s(1-\rho^{-1}).
	\end{align*}
	Moreover, for the jump conditions, we still take $U=(\rho,u,\theta)$ and use the conservative
	variables and flux of the quasineutral Euler system as in \eqref{eq:QNEuler}:
	\begin{equation}\notag
		\begin{gathered}
			Q(U)=\begin{pmatrix}\rho\\\rho u\\\mathscr E_{\mathrm{qn}}\end{pmatrix},
			\qquad
			f(U)=\begin{pmatrix}\rho u\\\rho u^2+\Peff\\u(\mathscr E_{\mathrm{qn}}+\Peff)\end{pmatrix},
			\\[1mm]
			\mathscr E_{\mathrm{qn}}=\rho\left(\frac{u^2}{2}+c_v\theta+\log\rho\right),
		\end{gathered}
	\end{equation}
	Setting $[h]=h_+-h_-$, the Rankine--Hugoniot conditions are
	\begin{equation}\label{eq:RH.intro}
		s[Q(U)]=[f(U)],\qquad \phi_\pm=-\log\rho_\pm.
	\end{equation}

	
	
	With all preparations above, we state the first result of this paper as follows.
	
	\begin{theorem}[Existence of small-amplitude NSFP shock profiles]\label{thm:existence}
		Fix $\gamma>1$, $\theta_0>0$ and $\mu,\kappa,\la>0$.
		There exist $\varepsilon_{*}>0$, $\vartheta>0$ and a neighborhood
		$\mathcal{N}\subset\R^4$ 
		of $Y_-:=(1,\theta_0,0,0)$ 
		such that
		if the right state satisfies the Rankine--Hugoniot conditions
		\eqref{eq:RH.intro}, the Lax inequalities
		\begin{equation}\label{eq:lax}
			0<s<c_0,\qquad u_++\sqrt{\gamma\theta_++1}<s,
		\end{equation}
		and $\varepsilon:=1-\rho_+>0$,
		$\varepsilon+|\theta_+-\theta_0|\leq\varepsilon_{*}$,
		then \eqref{eq:profile-second-order.intro} or equivalently \eqref{eq:vector-field.intro} admits a smooth solution $Y^S=(\rho^S,\theta^S,\phi^S,W^S)$ joining the left state $Y_-$ to the prescribed right state, and thus the original NSFP system in Eulerian coordinates \eqref{eq:NSFP} admits a smooth traveling wave $(\rho^S,u^S,\theta^S,\phi^S)$ with $u^S=s(1-\frac{1}{\rho^S})$. 
		It is unique up to translation among profiles for which
		$(\rho^S,\theta^S,\phi^S,(\phi^S)')(\zeta)\in\mathcal N$ for all $\zeta\in\R$.
		Moreover,
		\begin{align}\label{eq:qualitative}
			{\theta^S}\geq\frac{\theta_0}{2},\qquad
			(\rho^S)'<0,\qquad (\phi^S)'>0,\qquad
			C^{-1}|(\rho^S)'|\leq(\phi^S)'\leq C|(\rho^S)'|.
		\end{align}
		After fixing the translation by ${\rho^S}(0)=1-\varepsilon/2$,
		for every integer $k\geq0$, we have
		\begin{equation}\label{eq:decay}
			\left|\partial_\zeta^k\bigl(
			{\rho^S}-\rho_\pm,{u^S}-u_\pm,
			{\theta^S}-\theta_\pm,{\phi^S}-\phi_\pm\bigr)(\zeta)\right|
			\leq C_k\varepsilon^{k+1}e^{-\vartheta\varepsilon|\zeta|},
			\qquad \pm\zeta\geq0.
		\end{equation}
		The neighborhood $\mathcal N$, the strength bound $\varepsilon_{*}$ and the constants
		$C,C_k,\vartheta$ are independent of $\varepsilon$.
			\end{theorem}

	To study the dynamical stability of the shock profiles obtained in
	Theorem \ref{thm:existence} for the NSFP system \eqref{eq:NSFP}, we introduce the Lagrangian mass coordinates in terms of
	$\dd m=\rho\dd x-\rho u\dd t$.
	For simplicity, we still write $x$ in place of $m$. Therefore, from \eqref{eq:NSFP},  
	setting the specific volume $v=1/\rho$ and using the same notation
	for the transformed unknowns, we have the following NSFP system in Lagrangian coordinates:
	\begin{equation}\label{eq:NSP-1}
		\left\{\begin{aligned}
			&v_t-u_x=0,\\
			&u_t+p_x=\mu\left(\frac{u_x}{v}\right)_x
			+\frac{\phi_x}{v},\qquad p=\frac\theta v,\\
			&\cv\theta_t+p u_x
			=\kappa\left(\frac{\theta_x}{v}\right)_x+\mu\frac{u_x^2}{v},\\
			&\la^2\left(\frac{\phi_x}{v}\right)_x=1-ve^{-\phi},
		\end{aligned}\right.
	\end{equation}
	with $t\geq 0$ and $x\in \R$.
	We supplement \eqref{eq:NSP-1} with initial data
	\begin{equation}\label{lnspf.in}
		(v,u,\theta)(0,x)=(v_{\text{in}},u_{\text{in}},\theta_{\text{in}})(x),\quad x\in \R.
	\end{equation}
	For \eqref{eq:NSP-1}, we still denote the small-amplitude compressive $3$-shock by $(v^S,u^S,\theta^S,\phi^S)(\cdot)$ with $v_+>1$, strength
	$
	\delta_S:=v_+-v_-=\frac{\varepsilon}{1-\varepsilon}, 
	$
	and the shock speed $\sigma=(s-u_-)/v_-=(s-u_+)/v_+$. For more properties of the shock profile in Lagrangian coordinates, we refer to Sect.~\ref{sec:reference} later on.   
	
	We can now state the second result of this paper concerning the large time asymptotic stability of the shock profile $(v^S,u^S,\theta^S,\phi^S)(\cdot)$ in Lagrangian coordinates.

	\begin{theorem}[Stability result]\label{thm:AC}
		Assuming all the conditions of Theorem \ref{thm:existence} and taking $\varepsilon_*>0$ to be further smaller if necessary,
		there exists $\varepsilon_0>0$ such that the following holds. Suppose that the initial data in \eqref{lnspf.in} satisfy
		\begin{align*}
			\|(v_{\text{in}}-v^{S},u_{\text{in}}-u^{S},\theta_{\text{in}}-\theta^{S})\|_{H^2}<\varepsilon_0.
		\end{align*}
		Then the full compressible NSFP system \eqref{eq:NSP-1} in Lagrangian coordinates admits a unique global-in-time solution $(v,u,\theta,\phi)(t,x)$ for $t\geq 0$ and $x\in \R$, and there exists an absolutely continuous shift function $X(t): [0,+\infty)\to \R$ with $X(0)=0$
		such that
		\begin{gather*}
			(v,u,\theta)(t,x)-(v^{S},u^{S},\theta^{S})(x-\sigma t-X(t))\in C([0,+
			\infty);H^2(\R)),\\
			\phi(t,x)-\phi^{S}(x-\sigma t-X(t))\in C([0,+
			\infty);H^3(\R)).
		\end{gather*}
		Moreover, the solution $(v,u,\theta,\phi)(t,x)$ and shift function $X(t)$ satisfy
		\begin{align*}
			\lim_{t\rightarrow+\infty}\| (v,u,\theta,\phi)(t,\cdot)-(v^{S},u^{S},\theta^{S},\phi^{S})(\cdot-\sigma t-X(t))\|_{L^\infty}=0,
		\end{align*}
		and
		\begin{align*}
			\lim_{t\rightarrow+\infty}|\dot{X}(t)|=0.
		\end{align*}
	\end{theorem}
	
In what follows we present the main difficulties and the corresponding strategies for the two theorems. For Theorem~\ref{thm:existence}, the first key obstacle is that the electric force $-\rho\phi_x$ and the electric work $-\rho u\phi_x$ in \eqref{eq:NSFP} are not of the flux form. Without a conservative structure, the Rankine--Hugoniot conditions of a traveling wave cannot be compared with those of a hyperbolic system. The key observations are those of Lemma~\ref{lem:momentum} in Section~\ref{sec:conservative}:
\[
\rho\phi_x=\partial_x\Bigl(\tfrac12\la^2\phi_x^2-e^{-\phi}\Bigr),\qquad
\rho u\phi_x
=\partial_t\Bigl(\rho\phi+\tfrac12\la^2\phi_x^2+e^{-\phi}\Bigr)
+\partial_x\bigl(\rho u\phi-\la^2\phi_x\phi_t\bigr).
\]
They convert \eqref{eq:NSFP} into the viscous conservation laws \eqref{eq:NSFPcons}. After the traveling-wave reduction of Section~\ref{sec:profile}, the Rankine--Hugoniot conditions \eqref{eq:RH} coincide with those of the quasineutral Euler system \eqref{eq:QNEuler} with effective pressure $P_{\mathrm{eff}}=\rho(\theta+1)$. Integrating the profile equations from the left state produces the first-order system \eqref{eq:vector-field.intro} in $Y=(\rho,\theta,\phi,W)$. Equilibria of $\mathscr F$ are precisely the quasineutral Rankine--Hugoniot states of Lemma~\ref{lem:equilibria}. The second difficulty is the spectral issue. At the sonic left state $Y_-=(1,\theta_0,0,0)$, Lemma~\ref{lem:kernel} shows that $A(c_0)=D_Y\mathscr F(Y_-;c_0)$ has a simple zero eigenvalue with kernel
\[
\ker A(c_0)=\operatorname{span}\{R_0\},\qquad
R_0=\bigl(1,(\gamma-1)\theta_0,-1,0\bigr)^{\mathsf T},
\]
while Lemma~\ref{lem:inertia} supplies two stable eigenvalues and one unstable eigenvalue. A direct hyperbolic shooting is therefore unavailable. The speed-extended system \eqref{eq:extended} of Section~\ref{sec:spectrum} adjoins $\tau'=0$; Proposition~\ref{prop:extended-spectrum} then yields a two-dimensional center subspace
\[
E^c=\operatorname{span}\bigl\{(R_0,0),(0,0,0,0,1)^{\mathsf T}\bigr\}
\]
and a uniform hyperbolic gap. Section~\ref{sec:center} applies the center-manifold theorem of Appendix~\ref{app:center} and reduces the dynamics to $\eta'=g(\eta,\tau)$ on that manifold, with $\eta=\rho-1$ and $\tau$ frozen, see \eqref{eq:reduced-tau}. Restricting to the Rankine--Hugoniot speed $\tau=s(\varepsilon)$ with
\[
s(\varepsilon)=c_0-\tfrac12\Gamma\varepsilon+O(\varepsilon^2),\qquad
\Gamma=\frac{\gamma(\gamma+1)\theta_0+2}{2c_0}>0,
\]
we obtain the scalar equation $\eta'=\widehat g(\eta,\varepsilon)$. In Section~\ref{sec:heteroclinic}, we exclude interior zeros and combine this fact with $\widehat g_\eta(0,\varepsilon)>0$ to construct a heteroclinic from $0$ to $-\varepsilon$, unique up to translation. Lifting this orbit by \eqref{eq:lift-solves-system} produces the NSFP profile.

For Theorem~\ref{thm:AC}, the stability argument involves three main difficulties. Indeed, there is no zero-mass assumption, so the translational kernel of the shock must be removed dynamically. The temperature enters the hyperbolic flux as an independent unknown, while heat conduction contributes to parabolic dissipation. The electric work in the energy equation produces remainders that are not controlled by the fluid viscosity alone. Section~\ref{sec:AC} is devoted to the full proof of Theorem~\ref{thm:AC}. The argument is carried out  for \eqref{eq:NSP} in Lagrangian coordinates, and further for \eqref{eq:shiftedNSP} relative to the shifted profile as in \eqref{eq:shiftedprofile}. The strategy is $a$-contraction with the time-dependent shift function $X(t)$ to be explicitly defined in Section~\ref{sec:shift}. The weight
\begin{equation*}
a(\xi)=1+\delta_S^{-1/2}\bigl(v^S(\xi-X(t))-v_-\bigr)
\end{equation*}
in \eqref{eq:weight} satisfies $a_\xi=\delta_S^{-1/2}\bar v_\xi>0$ and makes the shock layer dissipative. The relative entropy of Section~\ref{sec:relative} is
\[
H_0=\frac{\tilde u^2}{2}+(\bar\theta+1)\Phi\Bigl(\frac{v}{\bar v}\Bigr)
+\frac{\bar\theta}{\gamma-1}\Phi\Bigl(\frac{\theta}{\bar\theta}\Bigr),\qquad
\Phi(z)=z-1-\log z,
\]
augmented by the electric energy $\eta_\phi$ in \eqref{eq:entropy}, which contains the mixed term $(\la^2/\bar v^2)\tilde v\,\tilde\phi_{\xi\xi}$. The sum $H_0+\eta_\phi$ is coercive by \eqref{eq:coercive}. Lemma~\ref{lem:identity} then yields the identity
\[
\frac{\mathrm d}{\mathrm d t}\mathcal E=\dot X\,\mathcal Y+\mathcal J^{\mathrm{bad}}-\mathcal J^{\mathrm{good}},
\qquad
\mathcal E(t)=\int_{\mathbb R}a(H_0+\eta_\phi)\,\mathrm d\xi.
\]
Section~\ref{sec:decomp} splits the flux into good terms $ \G, \D$, remainders $\B_i,\Y_i$ and shift terms, and the ODE \eqref{eq:shiftODE} is chosen so that
\[
\dot X=-\frac{M}{\delta_S}(\Y_1+\Y_2+\Y_3)
\]
converts the translational kernel into the dissipation $-(\delta_S/M)|\dot X|^2$. The remaining leading combination $-\frac{\delta_S}{2M}|\dot X|^2+\B_1+\B_2-\G-\frac34\D$ is estimated in Section~\ref{sec:leading}.  A key input from the center-manifold reduction is the sharp diffusion estimate of Lemma \ref{lem:sharp} in Appendix \ref{app:sharp}, obtained from the scalar profile equation \eqref{eq:scalar-profile}, \eqref{eq:phi-monotonicity}, and its factorization \eqref{eq:factorisation}. This estimate identifies the leading coefficient of the viscous dissipation in the normalized shock coordinate. Combined with the thermal and shift estimates and a weighted Poincar\'e inequality, it yields Lemma~\ref{lem:leading}.
The lower-order remainders are absorbed in Section~\ref{sec:zeroth}; spatial and time-differentiated Poisson estimates in Section~\ref{sec:poisson} control $\tilde\phi$ and the electric fluxes in $\mathcal R_\phi$; first- and second-order fluid energies are closed in Section~\ref{sec:derivatives}. Section~\ref{sec:closure} assembles these bounds into Proposition~\ref{thm:closure} and concludes the orbital stability of Theorem~\ref{thm:AC}.
		
	The rest of this paper is organized as follows. Section~\ref{sec:existence} is devoted to the proof of Theorem \ref{thm:existence}. We derive the conservative identities and the traveling-wave ODE, identify the Rankine–Hugoniot conditions, and construct the profiles and center-manifold reduction. We then establish their qualitative properties and uniqueness. In Section~\ref{sec:AC}, we prove Theorem \ref{thm:AC} by combining weighted relative-entropy estimates with higher-order fluid energy estimates and elliptic estimates for the Poisson equation. Appendix~\ref{app:center} recalls the center-manifold theorem, and Appendix~\ref{app:sharp} establishes sharp estimates for the diffusion terms.

	\section{Existence of small-amplitude shock profiles}\label{sec:existence}

This section proves Theorem~\ref{thm:existence}. In Section~\ref{sec:conservative} the electric force is rewritten in divergence form, so that the momentum and energy equations become viscous conservation laws. Section~\ref{sec:profile} substitutes the traveling-wave ansatz and integrates the resulting system: the Rankine--Hugoniot conditions reduce to those of the quasineutral Euler system with effective pressure $P_{\mathrm{eff}}=\rho(\theta+1)$, and the profile equations become a first-order autonomous ODE in $(\rho,\theta,\phi,W)$ whose equilibria are the quasineutral Rankine--Hugoniot states. Section~\ref{sec:spectrum} analyses the linearization at the sonic left state, which has a simple zero eigenvalue; adjoining the shock speed as a frozen parameter produces a two-dimensional center subspace. A $C^r$ center-manifold reduction is carried out in Section~\ref{sec:center}. Section~\ref{sec:heteroclinic} reduces the dynamics to a scalar ODE and constructs a unique small-amplitude heteroclinic, whose sign is controlled by genuine nonlinearity. Section~\ref{sec:qualitative} records the Lax inequalities, the monotonicity and exponential decay of the profile, and uniqueness up to translation. The argument is collected in Section~\ref{sec:proof-main}.
	
	\subsection{Conservative structure}\label{sec:conservative}
	\begin{lemma}\label{lem:momentum}
		For the solution of \eqref{eq:NSFP}, it holds
		\begin{align*}
			&\rho\phi_x=\partial_x\left(\frac{\la^2}{2}\phi_x^2-e^{-\phi}\right),\\
			&\rho u\phi_x
			=\partial_t\left(\rho\phi+\frac{\la^2}{2}\phi_x^2+e^{-\phi}\right)
			+\partial_x\left(\rho u\phi-\la^2\phi_x\phi_t\right).
		\end{align*}
	\end{lemma}
	\begin{proof}
		Multiply $\la^2\phi_{xx}=\rho-e^{-\phi}$ by $\phi_x$ and use
		$(e^{-\phi})_x=-e^{-\phi}\phi_x$ to obtain the first identity.
		For the second, the mass and Poisson equations give
		\begin{equation*}
			\begin{aligned}
				\rho u\phi_x
				&=\partial_x(\rho u\phi)+\phi\rho_t=\partial_t(\rho\phi)+\partial_x(\rho u\phi)
				-(\la^2\phi_{xx}+e^{-\phi})\phi_t.
			\end{aligned}
		\end{equation*}
		Rewriting the last term by the product rule yields the result.
	\end{proof}
	
	\subsection{Traveling-wave ODE and Rankine--Hugoniot conditions}\label{sec:profile}
	\subsubsection{Substitution}
	Set $W=\phi'$ and substitute $\zeta=x-st$ into \eqref{eq:NSFP} so as to have
	\begin{equation}\label{eq:profile-second-order}
		\left\{\begin{aligned}
			&-s\rho'+(\rho u)'=0,\\
			&-s(\rho u)'+(\rho u^2+\rho\theta)'-\rho W=\mu u'',\\
			&-s\bigl[\rho(\tfrac12u^2+c_v\theta)\bigr]'
			+\bigl[u\{\rho(\tfrac12u^2+c_v\theta)+\rho\theta\}\bigr]'
			-\rho uW=(\mu uu'+\kappa\theta')',\\
			&\phi'=W,\qquad \la^2W'=\rho-e^{-\phi}.
		\end{aligned}\right.
	\end{equation}
	The far field conditions are given in \eqref{eq:far}. We also assume that 	$u'$ and $\theta'$ tend to zero at both ends.

	\subsubsection{Mass flux}
	Integrating the first equation of \eqref{eq:profile-second-order} and using \eqref{eq:leftE} gives
	\begin{equation}\label{eq:mass}
		\rho(u-s)=-s,\qquad
		u=s\left(1-\frac1\rho\right),\qquad
		u'=s\rho^{-2}\rho'.
	\end{equation}
	\subsubsection{Integrated momentum}
	By Lemma \ref{lem:momentum}, integration from $-\infty$ gives
	\begin{equation}\label{eq:integrated-momentum}
		\mu u'=-s\rho u+\rho u^2+\rho\theta+e^{-\phi}
		-1-\theta_0-\frac{\la^2}{2}W^2=F_1,
	\end{equation}
	where, after using \eqref{eq:mass},
	\begin{equation}\label{eq:F1}
		F_1(\rho,\theta,\phi,W;s)
		=-\frac{s^2}{\rho}(\rho-1)+\rho\theta-\theta_0
		+e^{-\phi}-1-\frac{\la^2}{2}W^2.
	\end{equation}
	Using $u'=s\rho^{-2}\rho'$ in
	\eqref{eq:integrated-momentum}, we obtain, for $s\ne0$,
	\begin{equation}\label{eq:rho-ode}
		\rho'=\frac{\rho^2}{\mu s}F_1.
	\end{equation}
	\subsubsection{Integrated energy}
	Applying the second identity of Lemma \ref{lem:momentum} to the
	traveling wave and using \eqref{eq:mass}, integration of the energy
	equation in \eqref{eq:profile-second-order} from the left state gives
	\begin{equation*}
		\mu uu'+\kappa\theta'=G.
	\end{equation*}
	Using $\mu u'=F_1$, we obtain
	\begin{equation}\label{eq:theta-ode}
		\kappa\theta'=G-uF_1,
	\end{equation}
	\begin{equation}\label{eq:G}
		\begin{aligned}
			G(\rho,\theta,\phi,W;s)
			=&-s\left[\rho(\tfrac12u^2+c_v\theta)-c_v\theta_0\right]
			+u\left[\rho(\tfrac12u^2+c_v\theta)+\rho\theta\right]\\
			&+s\phi-s\left(\frac{\la^2}{2}W^2-e^{-\phi}+1\right).
		\end{aligned}
	\end{equation}
	
	\subsubsection{Rankine--Hugoniot conditions}\label{sec:RH}
	By Lemma \ref{lem:momentum}, the momentum equation takes the conservative form
	\begin{equation*}
		(\rho u)_t+
		\left(\rho u^2+\rho\theta+e^{-\phi}
		-\frac{\la^2}{2}\phi_x^2\right)_x=\mu u_{xx}.
	\end{equation*}
	At a quasineutral constant state, $e^{-\phi}=\rho$ and
	$\phi_x=0$, so we define the effective pressure as
	\begin{equation*}
		\Peff(\rho,\theta)=\rho(\theta+1).
	\end{equation*}
	Recall the energy density in \eqref{eq:NSFPcons},
	\begin{equation*}
		\mathscr E
		=\rho\left(\frac{u^2}{2}+c_v\theta-\phi\right)
		+\rho-e^{-\phi}-\frac{\la^2}{2}\phi_x^2.
	\end{equation*}
	By Lemma \ref{lem:momentum}, the energy equation reads
	\begin{equation}\label{eq:energy-conservative}
		\mathscr E_t+
		\left[
		u\left(\mathscr E+\rho\theta+e^{-\phi}
		+\frac{\la^2}{2}\phi_x^2\right)
		+\la^2\phi_x\phi_t
		\right]_x
		=(\mu uu_x+\kappa\theta_x)_x.
	\end{equation}
	At the
	quasineutral constant state with $\phi=-\log\rho$ and
	$\phi_x=\phi_t=0$, one has 
	\begin{equation*}
		\mathscr E_{\mathrm{qn}}:=\rho\left(\frac{u^2}{2}+c_v\theta+\log\rho\right).
	\end{equation*}
	The energy flux in \eqref{eq:energy-conservative} then reduces to
	$u(\mathscr E_{\mathrm{qn}}+\Peff)$.
	Thus, with $U=(\rho,u,\theta)$, the corresponding conservative
	variables and flux at the end states are
	\begin{equation}\label{eq:effective-Q}
		Q(U)=\begin{pmatrix}\rho\\\rho u\\\mathscr E_{\mathrm{qn}}\end{pmatrix},
		\qquad
		f(U)=\begin{pmatrix}\rho u\\\rho u^2+\Peff\\u(\mathscr E_{\mathrm{qn}}+\Peff)\end{pmatrix}.
	\end{equation}
	Integrating the traveling-wave conservation laws between the
	two constant end states, where the derivative terms vanish, gives
	the Rankine--Hugoniot conditions, with $[h]=h_+-h_-$,
	\begin{equation}\label{eq:RH}
		s[Q(U)]=[f(U)],\qquad \phi_\pm=-\log\rho_\pm.
	\end{equation}
		
	\subsubsection{The Hugoniot curve} 
	Set
	\begin{equation}\label{def.Gamma.ad} 
	\Gamma:=\frac{\gamma(\gamma+1)\theta_0+2}{2c_0}>0
	\end{equation}
	in order to establish genuine nonlinearity of the right-going
	acoustic field at the left state; see its exact derivation in \eqref{eq:Hugoniot-Gamma-computation} later.
	
	\begin{lemma}\label{lem:Hugoniot} 
		There are $\varepsilon_H>0$ and smooth maps
		$\varepsilon\mapsto(s,u_+,\theta_+,\phi_+)$ for
		$|\varepsilon|<\varepsilon_H$, with $\rho_+=1-\varepsilon$, such that
		\eqref{eq:RH} holds and
		\begin{equation}\label{eq:Hugoniot-expansion}
			\begin{aligned}
				s(\varepsilon)&=c_0-\frac{\Gamma}{2}\varepsilon+O(\varepsilon^2),
				&u_+(\varepsilon)&=-c_0\varepsilon+O(\varepsilon^2),\\
				\theta_+(\varepsilon)&=\theta_0-(\gamma-1)\theta_0\varepsilon+O(\varepsilon^2),
				&\phi_+(\varepsilon)&=\varepsilon+O(\varepsilon^2),
			\end{aligned}
		\end{equation}
		where $\Gamma$ is defined in \eqref{def.Gamma.ad}.
		This is the unique nontrivial local Hugoniot branch with speed near $c_0$.
		For sufficiently small  $\varepsilon>0$, it satisfies \eqref{eq:lax}.
	\end{lemma}
	
	\begin{proof}
		Differentiating the conservative
		variables $Q$ and flux $f$  in \eqref{eq:effective-Q} gives
		\begin{equation*}
			DQ(U)=\begin{pmatrix}
				1&0&0\\
				u&\rho&0\\
				\frac12u^2+c_v\theta+\log\rho+1&\rho u&\rho c_v
			\end{pmatrix},\qquad \det DQ(U)=c_v\rho^2>0,
		\end{equation*}
		\begin{equation}\label{eq:Hugoniot-flux-jacobian}
			Df(U)=DQ(U)A_E(U),\qquad
			A_E(U)=\begin{pmatrix}
				u&\rho&0\\
				(\theta+1)/\rho&u&1\\
				0&\theta/c_v&u
			\end{pmatrix}.
		\end{equation}
		The  eigenvalue and a normalized right eigenvector for $A_E(U)$ are
		\begin{gather*}
			\lambda_3(U)=u+c,\qquad c=\sqrt{\gamma\theta+1},\qquad
			r_3(U)=\begin{pmatrix}1\\c/\rho\\(\gamma-1)\theta/\rho\end{pmatrix},
			\\ A_E(U)r_3(U)=\lambda_3(U)r_3(U).
		\end{gather*}
		At the left state, let \(r\) and \(\ell\) be the right and left eigenvectors, respectively, associated with the eigenvalue \(c_0\), then it holds
		\begin{equation*}
			r=r_3(U_-)=\begin{pmatrix}1\\c_0\\(\gamma-1)\theta_0\end{pmatrix},
			\qquad
			\ell=\left(\frac{\theta_0+1}{2c_0^2},\frac1{2c_0},\frac1{2c_0^2}\right).
		\end{equation*}

		We first compute the derivatives of a $C^2$ Hugoniot curve; its existence is established as follows. With $\rho(\varepsilon)=1-\varepsilon$, it satisfies
		\begin{equation}\label{eq:Hugoniot-curve-RH}
			s(\varepsilon)\bigl[Q(U(\varepsilon))-Q(U_-)\bigr]
			=f(U(\varepsilon))-f(U_-),\qquad U(0)=U_-.
		\end{equation}
		Differentiation with respect to $\varepsilon$ gives
		\begin{equation}\label{eq:Hugoniot-first-derivative}
			(\partial_\varepsilon s)\bigl[Q(U)-Q(U_-)\bigr]
			+sDQ(U)\partial_\varepsilon U
			=Df(U)\partial_\varepsilon U.
		\end{equation}
		Multiplying \eqref{eq:Hugoniot-first-derivative}$|_{\varepsilon=0}$ by
		$DQ(U_-)^{-1}$ and using \eqref{eq:Hugoniot-flux-jacobian}, since the first term vanishes, one has
		\begin{equation*}
			\bigl(A_E(U_-)-s(0)I\bigr)\partial_\varepsilon U(0)=0.
		\end{equation*}
		Due to $\partial_\varepsilon\rho(0)=-1$, this vector is nonzero.
		The eigenvalues at $U_-$ are $-c_0,0,c_0$; the branch with speed near
		$c_0$ therefore has $s(0)=c_0$. Its eigenspace is spanned by $r$.
		The density normalization fixes the multiple:
		\begin{equation}\label{eq:Hugoniot-tangent}
			\partial_\varepsilon U(0)=-r,
			\qquad
			\partial_\varepsilon u_+(0)=-c_0,
			\qquad
			\partial_\varepsilon\theta_+(0)=-(\gamma-1)\theta_0.
		\end{equation}
		Quasineutrality and the mass relation give, respectively,
		\begin{equation*}
			\begin{aligned}
				\phi_+(\varepsilon)&=-\log(1-\varepsilon)
				=\varepsilon+\tfrac12\varepsilon^2+O(\varepsilon^3),\\
				u_+(\varepsilon)&=s(\varepsilon)\left(1-\frac1{1-\varepsilon}\right)
				=-\frac{s(\varepsilon)\varepsilon}{1-\varepsilon}
				=-c_0\varepsilon+O(\varepsilon^2).
			\end{aligned}
		\end{equation*}
		
		We now compute the derivative of the speed. Differentiate \eqref{eq:Hugoniot-first-derivative} once more and set
$\varepsilon=0$. Write $DQ_-=DQ(U_-)$ and similarly for $Df_-$ and
		the second derivatives. Since $\partial_\varepsilon U(0)=-r$, we have
		\begin{equation}\label{eq:Hugoniot-second-derivative}
			-2\partial_\varepsilon s(0)DQ_-r+c_0D^2Q_-[r,r]
			+c_0DQ_-\partial_\varepsilon^2U(0)
			=D^2f_-[r,r]+Df_-\partial_\varepsilon^2U(0).
		\end{equation}
		Multiplying \eqref{eq:Hugoniot-second-derivative} on the left by
		$\ell DQ_-^{-1}$ cancels the two terms containing
		$\partial_\varepsilon^2U(0)$, because
		$\ell DQ_-^{-1}Df_-=\ell A_E(U_-)=c_0\ell$. Thus, one has
		\begin{equation}\label{eq:Hugoniot-speed-projection}
			-2\partial_\varepsilon s(0)
			=\ell DQ_-^{-1}\bigl(D^2f_-[r,r]-c_0D^2Q_-[r,r]\bigr).
		\end{equation}
		To evaluate the right-hand side, differentiating
		$Df(U)r_3(U)=\lambda_3(U)DQ(U)r_3(U)$ in the direction $r$ at $U_-$ to obtain
		\begin{equation*}
			\begin{aligned}
				D^2f_-[r,r]+Df_-\bigl(Dr_3(U_-)r\bigr)
				=&\bigl(\nabla\lambda_3(U_-)\cdot r\bigr)DQ_-r\\
				&+c_0D^2Q_-[r,r]+c_0DQ_-\bigl(Dr_3(U_-)r\bigr).
			\end{aligned}
		\end{equation*}
		Multiplying by $\ell DQ_-^{-1}$ cancels the terms containing
		$Dr_3(U_-)r$ and gives
		\begin{equation}\label{eq:Hugoniot-Gamma-computation}
			\begin{aligned}
				&\ell DQ_-^{-1}\bigl(D^2f_-[r,r]-c_0D^2Q_-[r,r]\bigr)\\
				&=\nabla\lambda_3(U_-)\cdot r
				=c_0+\frac{\gamma(\gamma-1)\theta_0}{2c_0}=\frac{2(\gamma\theta_0+1)+\gamma(\gamma-1)\theta_0}{2c_0}
				=\Gamma.
			\end{aligned}
		\end{equation}
		Combining \eqref{eq:Hugoniot-speed-projection} and
		\eqref{eq:Hugoniot-Gamma-computation} yields
		\begin{equation}\label{111}
			\partial_\varepsilon s(0)=-\frac\Gamma2,
			\qquad s(\varepsilon)=c_0-\frac\Gamma2\varepsilon+O(\varepsilon^2).
		\end{equation}
		
		We are now ready to establish the existence and local uniqueness of the curve.
		Write $U=U_-+\varepsilon V$, with the density component $V^\rho=-1$.
		For $\varepsilon\ne0$, dividing \eqref{eq:Hugoniot-curve-RH} by
		$\varepsilon$ and using the fundamental theorem of calculus gives
		\begin{equation*}
			\mathscr R_{\mathrm H}(\varepsilon,s,V)
			:=\int_0^1\bigl[Df(U_-+t\varepsilon V)
			-sDQ(U_-+t\varepsilon V)\bigr]V\dd t=0.
		\end{equation*}
		This equation is smooth at $\varepsilon=0$, and
		$\mathscr R_{\mathrm H}(0,c_0,-r)=DQ_-(A_E(U_-)-c_0I)(-r)=0$, since $V=-r$.
		For variations with $\delta V^\rho=0$, its derivative is
		\begin{equation*}
			D_{(s,V)}\mathscr R_{\mathrm H}(0,c_0,-r)[\delta s,\delta V]
			=DQ_-\bigl[(A_E(U_-)-c_0I)\delta V+\delta s\,r\bigr].
		\end{equation*}
		In the coordinates $(s,V^u,V^\theta)$, this is the matrix
		\begin{equation*}
			D_{(s,V^u,V^\theta)}\mathscr R_{\mathrm H}(0,c_0,-r)
			=DQ_-\begin{pmatrix}
				1&1&0\\
				c_0&-c_0&1\\
				\theta_0/c_v&\theta_0/c_v&-c_0
			\end{pmatrix}.
		\end{equation*}
		Since $\det DQ_-=c_v$, expansion along the first row gives
		\begin{equation*}
			\det D_{(s,V^u,V^\theta)}\mathscr R_{\mathrm H}(0,c_0,-r)
			=c_v\left[\left(c_0^2-\frac{\theta_0}{c_v}\right)
			-\left(-c_0^2-\frac{\theta_0}{c_v}\right)\right]
			=2c_vc_0^2>0.
		\end{equation*}
		The derivative is invertible. The implicit function theorem therefore
		gives unique smooth functions $s(\varepsilon)$ and $V(\varepsilon)$
		near $(c_0,-r)$, justifying the preceding differentiations.
		
	In the end, we prove the Lax inequalities.
		By \eqref{eq:Hugoniot-tangent}, expansion of the right sound speed gives
		\begin{equation*}
			\begin{aligned}
				\sqrt{\gamma\theta_++1}
				&=\sqrt{c_0^2-\gamma(\gamma-1)\theta_0\varepsilon+O(\varepsilon^2)}=c_0-\frac{\gamma(\gamma-1)\theta_0}{2c_0}\varepsilon
				+O(\varepsilon^2).
			\end{aligned}
		\end{equation*}
		Adding $u_+=-c_0\varepsilon+O(\varepsilon^2)$ and using
		\eqref{eq:Hugoniot-Gamma-computation}, we obtain
		\begin{equation*}
			\begin{aligned}
				\lambda_3(U_+)&=u_++\sqrt{\gamma\theta_++1}
				=c_0-\Gamma\varepsilon+O(\varepsilon^2),\\
				s-\lambda_3(U_+)&=
				\left(c_0-\frac\Gamma2\varepsilon\right)
				-\left(c_0-\Gamma\varepsilon\right)+O(\varepsilon^2)
				=\frac\Gamma2\varepsilon+O(\varepsilon^2),\\
				\lambda_3(U_-)-s&=c_0-s
				=\frac\Gamma2\varepsilon+O(\varepsilon^2).
			\end{aligned}
		\end{equation*}
		Both gaps are positive for sufficiently small $\varepsilon>0$.
		Also $s=c_0+O(\varepsilon)>0$. Hence \eqref{eq:lax} holds, and
		$\rho_+=1-\varepsilon<\rho_-$. This completes the proof of Lemma \ref{lem:Hugoniot}.
	\end{proof}

	\subsubsection{The first-order system}
	Note that \eqref{eq:rho-ode} and \eqref{eq:theta-ode} together with the Poisson equation give
	\begin{equation}\label{eq:vector-field}
		Y'=\mathscr F(Y;s),\qquad Y=(\rho,\theta,\phi,W),\qquad
		\mathscr F(Y;s)=\begin{pmatrix}
			\dfrac{\rho^2}{\mu s}F_1(Y;s)\\[2mm]
			\dfrac1\kappa\bigl(G(Y;s)-uF_1(Y;s)\bigr)\\[2mm]
			W\\ \la^{-2}(\rho-e^{-\phi})
		\end{pmatrix}.
	\end{equation}
	\begin{lemma}\label{lem:equilibria}
		For $s\ne0$, $\mathscr F$ is smooth on $\{\rho>0\}$ and $\mathscr F(Y_-;s)=0$.
		Its equilibria are precisely the quasineutral states
		$Y=(\rho,\theta,-\log\rho,0)$ satisfying \eqref{eq:RH} with
		$u=s(1-\rho^{-1})$. A heteroclinic of \eqref{eq:vector-field}
		from $Y_-$ to $Y_+$ reconstructs a traveling wave of \eqref{eq:NSFP}.
	\end{lemma}
	\begin{proof}
		Smoothness and $\mathscr F(Y_-;s)=0$ follow directly
		from the definition. The equilibria are characterized by
		$W=0$, $\phi=-\log\rho$, and $F_1=G=0$.
		With $u=s(1-\rho^{-1})$, the mass jump condition holds
		automatically, while \eqref{eq:F1} and \eqref{eq:G}
		identify $F_1=G=0$ with the momentum and energy
		components of \eqref{eq:RH}.
		Finally, along a heteroclinic orbit,
		$\mu u'=F_1$ and $\mu uu'+\kappa\theta'=G$.
		Differentiating these identities and using
		\eqref{eq:mass} together with the last two components
		of \eqref{eq:vector-field} recovers
		\eqref{eq:profile-second-order}.
	\end{proof}
	\begin{remark}\label{rem:left-line}
		The left equilibrium is fixed while $s$ varies. Thus
		$\partial_s\mathscr F(Y_-;s)=0$, which will make the speed-extended Jacobian
		block diagonal.
	\end{remark}
	
	\subsection{Spectral analysis at the sonic state}\label{sec:spectrum}
	The traveling-wave problem has been reduced to a heteroclinic of the first-order system \eqref{eq:vector-field} connecting the quasineutral Rankine--Hugoniot states. In this subsection we linearize the vector field at the left sonic state $Y_-$ and determine the spectral structure of $A(s)=D_Y\mathscr F(Y_-;s)$. A simple zero eigenvalue appears at $s=c_0$; adjoining the shock speed then produces the two-dimensional center subspace to be used in Section~\ref{sec:center}.
	
	\subsubsection{Linearization}
	Put $q=\rho-1$ and $\Theta=\theta-\theta_0$.
	Differentiating the vector field \eqref{eq:vector-field}
	with respect to $Y$ and evaluating at $Y_-=(1,\theta_0,0,0)$,
	we obtain
	\begin{equation}\label{eq:Jacobian}
		A(s):=D_Y\mathscr F(Y_-;s)=
		\begin{pmatrix}
			\alpha&\beta&-\beta&0\\
			\delta&-\chi_{\mathrm h}&0&0\\
			0&0&0&1\\
			\nu&0&\nu&0
		\end{pmatrix},
	\end{equation}
	\begin{equation}\label{eq:coefficients}
		\alpha=\frac{\theta_0-s^2}{\mu s},\quad
		\beta=\frac1{\mu s},\quad
		\delta=\frac{s\theta_0}{\kappa},\quad
		\chi_{\mathrm h}=\frac{sc_v}{\kappa},\quad
		\nu=\la^{-2}.
	\end{equation}
	In the following calculation $s>0$ is near $c_0$, so
	$\beta,\delta,\chi_{\mathrm h},\nu>0$.
	\subsubsection{The characteristic polynomial}
	\begin{lemma}\label{lem:polynomial}
		The eigenvalues of $A(s)$ are the zeros of
		\begin{equation}\label{eq:polynomial}
			\begin{aligned}
				\mathfrak p(\Lambda;s)&:=-\det(\Lambda I-A(s))\\
				&=(\alpha-\Lambda)(\chi_{\mathrm h}+\Lambda)(\Lambda^2-\nu)
				+\beta\delta(\Lambda^2-\nu)-\beta\nu(\chi_{\mathrm h}+\Lambda)\\
				&=-\Lambda^4+(\alpha-\chi_{\mathrm h})\Lambda^3
				+(\alpha\chi_{\mathrm h}+\nu+\beta\delta)\Lambda^2-\nu(\alpha-\chi_{\mathrm h}+\beta)\Lambda
				-\nu(\alpha\chi_{\mathrm h}+\beta\delta+\beta\chi_{\mathrm h}).
			\end{aligned}
		\end{equation}
	\end{lemma}
	\begin{proof}
		
		Expanding \(\det(\Lambda I-A)\) along its third row gives
		\begin{equation*}
			\begin{aligned}
				\det(\Lambda I-A)
				&=\Lambda\begin{vmatrix}
					\Lambda-\alpha&-\beta&0\\
					-\delta&\Lambda+\chi_{\mathrm h}&0\\
					-\nu&0&\Lambda
				\end{vmatrix}
				+\begin{vmatrix}
					\Lambda-\alpha&-\beta&\beta\\
					-\delta&\Lambda+\chi_{\mathrm h}&0\\
					-\nu&0&-\nu
				\end{vmatrix}\\
				&=\Lambda^2\bigl[(\Lambda-\alpha)(\Lambda+\chi_{\mathrm h})-\beta\delta\bigr]
				-\nu\bigl[(\Lambda-\alpha)(\Lambda+\chi_{\mathrm h})-\beta\delta\bigr]
				+\beta\nu(\Lambda+\chi_{\mathrm h})\\
				&=(\Lambda^2-\nu)\bigl[(\Lambda-\alpha)(\Lambda+\chi_{\mathrm h})-\beta\delta\bigr]
				+\beta\nu(\Lambda+\chi_{\mathrm h})=-\mathfrak p(\Lambda;s).
			\end{aligned}
		\end{equation*}
		Finally,
		\begin{equation*}
			\begin{aligned}
				&(\alpha-\Lambda)(\chi_{\mathrm h}+\Lambda)
				=-\Lambda^2+(\alpha-\chi_{\mathrm h})\Lambda+\alpha\chi_{\mathrm h},\\
				&(\alpha-\Lambda)(\chi_{\mathrm h}+\Lambda)(\Lambda^2-\nu)
				=-\Lambda^4+(\alpha-\chi_{\mathrm h})\Lambda^3+(\alpha\chi_{\mathrm h}+\nu)\Lambda^2-\nu(\alpha-\chi_{\mathrm h})\Lambda-\nu\alpha\chi_{\mathrm h}.
			\end{aligned}
		\end{equation*}
		Adding $\beta\delta\Lambda^2-\beta\delta\nu-\beta\nu\chi_{\mathrm h}-\beta\nu\Lambda$
		gives the expanded form in \eqref{eq:polynomial}.
	\end{proof}
	
	\subsubsection{The sonic kernel}
	\begin{lemma}\label{lem:kernel}
		The constant term of $\mathfrak p(\Lambda;s)$ satisfies
		\begin{equation}\label{eq:p-zero}
			\mathfrak p(0;s)=-\frac{\nu c_v}{\mu\kappa}(c_0^2-s^2).
		\end{equation}
		Thus $\mathfrak p(0;s)=0$ if and only if $s^2=c_0^2$.
		At $s=c_0$, zero is an algebraically simple eigenvalue, and
		\begin{equation}\label{eq:kernel}
			\ker A(c_0)=\spanop\{R_0\},\qquad
			R_0=(1,(\gamma-1)\theta_0,-1,0)^{\mathsf T}.
		\end{equation}
	\end{lemma}
	\begin{proof}
		Set $\Lambda=0$ in \eqref{eq:polynomial}:
		\begin{equation*}
			\mathfrak p(0;s)=-\nu\alpha\chi_{\mathrm h}-\nu\beta\delta-\nu\beta\chi_{\mathrm h}
			=-\nu(\alpha\chi_{\mathrm h}+\beta\delta+\beta\chi_{\mathrm h}).
		\end{equation*}
		Substitution of \eqref{eq:coefficients} gives each product separately:
		\begin{equation*}
			\begin{aligned}
				\alpha\chi_{\mathrm h}&=\frac{\theta_0-s^2}{\mu s}\frac{sc_v}{\kappa}
				=\frac{c_v(\theta_0-s^2)}{\mu\kappa},\qquad
				\beta\delta=\frac1{\mu s}\frac{s\theta_0}{\kappa}
				=\frac{\theta_0}{\mu\kappa},\qquad
				\beta\chi_{\mathrm h}=\frac1{\mu s}\frac{sc_v}{\kappa}
				=\frac{c_v}{\mu\kappa}.
			\end{aligned}
		\end{equation*}
		Since $c_v(\gamma-1)=1$, we have $c_v+1=\gamma c_v$. Adding these terms gives
		\begin{equation*}
			\begin{aligned}
				\alpha\chi_{\mathrm h}+\beta\delta+\beta\chi_{\mathrm h}
				=&\frac{c_v(\theta_0-s^2)+\theta_0+c_v}{\mu\kappa}=\frac{(c_v+1)\theta_0+c_v(1-s^2)}{\mu\kappa}\\
				=&\frac{c_v(\gamma\theta_0+1-s^2)}{\mu\kappa}
				=\frac{c_v(c_0^2-s^2)}{\mu\kappa}.
			\end{aligned}
		\end{equation*}
		This proves \eqref{eq:p-zero}.
		
		To compute $\ker A(c_0)$, we evaluate the matrix in
		\eqref{eq:Jacobian} at $s=c_0$. We denote the
		corresponding values of the coefficients in \eqref{eq:coefficients} by
		\begin{align*}
			\alpha_0:=\alpha(c_0),\qquad
			\beta_0:=\beta(c_0),\qquad
			\delta_0:=\delta(c_0),\qquad
			\chi_{\mathrm h,0}:=\chi_{\mathrm h}(c_0).
		\end{align*}
		The equation $A(c_0)(q,\Theta,\phi,W)^{\mathsf T}=0$
		is equivalent to
		\begin{equation*}
			\alpha_0q+\beta_0\Theta-\beta_0\phi=0,\qquad
			\delta_0 q-\chi_{\mathrm h,0}\Theta=0,\qquad W=0,\qquad \nu(q+\phi)=0.
		\end{equation*}
		Then, we have
		\begin{equation*}
			\begin{aligned}
				&\phi=-q,\qquad\Theta=\frac{\delta_0}{\chi_{\mathrm h,0}} q
				=\frac{c_0\theta_0/\kappa}{c_0c_v/\kappa}q
				=\frac{\theta_0}{c_v}q=(\gamma-1)\theta_0q,\\
				&\alpha_0q+\beta_0\Theta-\beta_0\phi=\alpha_0q+\beta_0(\gamma-1)\theta_0q-\beta_0(-q)
				=\frac{\theta_0-c_0^2+(\gamma-1)\theta_0+1}{\mu c_0}q=0.
			\end{aligned}
		\end{equation*}
		Thus $q$ is free and every kernel vector equals $qR_0$, proving
		\eqref{eq:kernel}. To establish algebraic simplicity, return to
		general $s$ and differentiate the factored polynomial with respect
		to $\Lambda$:
		\begin{equation*}
			\begin{aligned}
				&	\mathfrak p_\Lambda(0;s)
				=\nu(\chi_{\mathrm h}-\alpha-\beta).
			\end{aligned}
		\end{equation*}
		At $s=c_0$, the factor in parentheses is
		\begin{equation}\label{eq:simple-zero}
			\begin{aligned}
					\mathfrak p_\Lambda(0;c_0)=\nu(\chi_{\mathrm h,0}-\alpha_0-\beta_0)=\nu(\frac{c_0c_v}{\kappa}+\frac{(\gamma-1)\theta_0}{\mu c_0})>0.
			\end{aligned}
		\end{equation}
		The last inequality proves that the zero of $\mathfrak p$ is simple. Since
		$\mathfrak p=-\det(\Lambda I-A)$, zero has algebraic multiplicity one as an
		eigenvalue of $A(c_0)$.
	\end{proof}
	
	\subsubsection{Exclusion of nonzero imaginary eigenvalues}
	\begin{lemma}\label{lem:no-imaginary}
		The matrix $A(c_0)$ has no eigenvalue on $i\R\setminus\{0\}$.
		Moreover, there exist constants $r_*,d_*,C_*>0$ and a smooth
		real-valued function $\mu_c(s)$, defined for $|s-c_0|<r_*$,
		such that $\mu_c(s)$ is an algebraically simple eigenvalue of
		$A(s)$ and
		\begin{equation*}
			\mu_c(c_0)=0,\qquad
			|\mu_c(s)|\le C_*|s-c_0|,\qquad |\mu_c(s)|<d_*.
		\end{equation*}
		Every other eigenvalue $\Lambda$ of $A(s)$ satisfies
		$
		|\operatorname{Re}\Lambda|\ge d_*.
		$
	\end{lemma}
	\begin{proof}
		First consider $s=c_0$. By \eqref{eq:simple-zero} and $\beta_0>0$,
		\begin{equation*}
			\chi_{\mathrm h,0}-\alpha_0-\beta_0>0,
			\qquad \chi_{\mathrm h,0}-\alpha_0>0.
		\end{equation*}
		For $\omega\in\R$, the polynomial \eqref{eq:polynomial} gives
		\begin{equation*}
			\begin{aligned}
				\operatorname{Im}\mathfrak p(i\omega;c_0)
				&=\omega\bigl[
				(\chi_{\mathrm h,0}-\alpha_0)(\omega^2+\nu)
				-\beta_0\nu\bigr]\\
				&=\omega\bigl[
				(\chi_{\mathrm h,0}-\alpha_0)\omega^2
				+\nu(\chi_{\mathrm h,0}-\alpha_0-\beta_0)\bigr].
			\end{aligned}
		\end{equation*}
		Since $\nu>0$, the bracket is strictly positive.
		Hence $\mathfrak p(i\omega;c_0)\ne0$ whenever $\omega\ne0$,
		which proves the first assertion.
		
		By Lemma \ref{lem:kernel},
		$\mathfrak p(0;c_0)=0$ and $\mathfrak p_\Lambda(0;c_0)>0$.
		Applying the implicit function theorem to the real equation
		$\mathfrak p(\Lambda;s)=0$ gives a smooth real root $\mu_c(s)$
		with $\mu_c(c_0)=0$. It remains algebraically simple because
		$\mathfrak p_\Lambda(\mu_c(s);s)\ne0$ for $s$ close to $c_0$.
		Smoothness yields $|\mu_c(s)|\le C_*|s-c_0|$ after restricting
		$s$ to a sufficiently small neighborhood of $c_0$.
		
		The other eigenvalues of $A(c_0)$ have nonzero real parts.
		Set
		\begin{equation*}
			d_*:=\frac12\min\bigl\{
			|\operatorname{Re}\Lambda|:
			\Lambda\in\operatorname{spec}A(c_0),\ \Lambda\ne0
			\bigr\}>0.
		\end{equation*}
		By continuity of polynomial roots, there is $r_*>0$ such that,
		for $|s-c_0|<r_*$, every eigenvalue other than $\mu_c(s)$ satisfies
		$|\operatorname{Re}\Lambda|\ge d_*$.
		Decreasing $r_*$ further, we also have $|\mu_c(s)|<d_*$.
		This proves the uniform spectral gap and the assertion about
		the strip.  
	\end{proof}
	
	\subsubsection{Stable and unstable dimensions}
	\begin{lemma}\label{lem:inertia}
		The spectrum of $A(c_0)$ consists of a simple zero, two eigenvalues
		with negative real part, and one positive real eigenvalue.
		The small eigenvalue of $A(s)$, denoted by $\mu_c(s)$, satisfies
		\begin{equation}\label{eq:small-eigenvalue}
			\mu_c(s)=\frac{c_0^2-s^2}{c_0 D}+O((s-c_0)^2),\qquad
			D:=\mu+\frac{\kappa\theta_0}{c_v^2c_0^2}>0.
		\end{equation}
		In particular, $\mu_c(s(\varepsilon))=(\Gamma/D)\varepsilon
		+O(\varepsilon^2)>0$ for small $\varepsilon>0$.
	\end{lemma}
	\begin{proof}
		The zero eigenvalue is simple by Lemma \ref{lem:kernel}, and Lemma
		\ref{lem:no-imaginary} excludes every other eigenvalue on the imaginary
		axis. It remains to determine the signs of the real parts. At $s=c_0$,
		factor the zero root from \eqref{eq:polynomial}:
		\begin{align*}
			\mathfrak p(\Lambda;c_0)=-\Lambda\big[&\Lambda^3-(\alpha_0-\chi_{\mathrm h,0})\Lambda^2
			-(\alpha_0\chi_{\mathrm h,0}+\nu+\beta_0\delta_0)\Lambda
			+\nu(\alpha_0-\chi_{\mathrm h,0}+\beta_0)\big].                 
		\end{align*}
		Thus the three nonzero eigenvalues are the roots of the monic cubic
		inside the brackets. Write it as $\Lambda^3+a\Lambda^2+b\Lambda+c=0$.
		Only the signs of $a$ and $c$ are needed. First,
		\begin{equation*}
			a=\chi_{\mathrm h,0}-\alpha_0>0,
			\qquad
			\alpha_0=\frac{\theta_0-c_0^2}{\mu c_0}
			=-\frac{(\gamma-1)\theta_0+1}{\mu c_0}<0.
		\end{equation*}
		For the constant term,
		\begin{align*}
			&\alpha_0-\chi_{\mathrm h,0}+\beta_0
			=\frac{\theta_0-c_0^2}{\mu c_0}
			-\frac{c_0c_v}{\kappa}+\frac{1}{\mu c_0} =-\frac{(\gamma-1)\theta_0}{\mu c_0}
			-\frac{c_0c_v}{\kappa}<0,                                  \\
			&c=\nu(\alpha_0-\chi_{\mathrm h,0}+\beta_0)<0.                              
		\end{align*}
		Let the three roots be $\Lambda_1,\Lambda_2,\Lambda_3$. Vieta's formulas
		give
		\begin{equation}\label{eq:cubic-vieta}
			\Lambda_1+\Lambda_2+\Lambda_3=-a<0,
			\qquad
			\Lambda_1\Lambda_2\Lambda_3=-c>0.
		\end{equation}
		If all roots are real, the positive product implies either three
		positive roots or one positive and two negative roots. The first option
		would make their sum positive, contrary to \eqref{eq:cubic-vieta}.
		Hence there is one positive root and two negative roots.
		
		If there is one real root and a complex conjugate pair, the product of
		the pair is its squared modulus. The positive product in
		\eqref{eq:cubic-vieta} therefore forces the real root to be positive.
		The negative sum then forces the conjugate pair to have negative real
		part. Thus, in either case, the three nonzero eigenvalues consist of one
		positive real eigenvalue and two eigenvalues with negative real part.
		
		We finally compute the small eigenvalue. The implicit function theorem,
		already used in Lemma \ref{lem:no-imaginary}, gives a real smooth root
		$\mu_c(s)$ with $\mu_c(c_0)=0$. By \eqref{eq:p-zero},
		\eqref{eq:simple-zero}, and the definition of $D$ \eqref{eq:small-eigenvalue},
		\begin{align*}
			\mathfrak p_\Lambda(0;c_0)
			&=\nu(\chi_{\mathrm h,0}-\alpha_0-\beta_0)
			=\frac{\nu c_vc_0}{\mu\kappa}
			\left(\mu+\frac{\kappa\theta_0}{c_v^2c_0^2}\right)
			=\frac{\nu c_vc_0}{\mu\kappa}D,                            \qquad
			\mathfrak p_s(0;c_0)=\frac{2\nu c_vc_0}{\mu\kappa}.                   
		\end{align*}
		Differentiating $\mathfrak p(\mu_c(s);s)=0$ at $s=c_0$ now gives
		\begin{equation*}
			\mu_c'(c_0)=-\frac{\mathfrak p_s(0;c_0)}{\mathfrak p_\Lambda(0;c_0)}=-\frac{2}{D}.
		\end{equation*}
		Consequently,
		\begin{align}
			&\mu_c(s)
			=-\frac{2}{D}(s-c_0)+O((s-c_0)^2)=\frac{c_0^2-s^2}{c_0D}+O((s-c_0)^2),\notag
			\\
			&\mu_c(s(\varepsilon))
			=\frac{\Gamma}{D}\varepsilon+O(\varepsilon^2)>0.           \label{eq:small-root-shock}
		\end{align}
		The last identity uses Lemma \ref{lem:Hugoniot}.
	\end{proof}
	
	\subsubsection{The speed-extended system}
	Adjoin the speed in \eqref{eq:vector-field} as a frozen variable:
	\begin{equation}\label{eq:extended}
		Y'=\mathscr F(Y;\tau),\qquad \tau'=0.
	\end{equation}
	\begin{proposition}\label{prop:extended-spectrum}
		At \((Y_-,c_0)\), the linearization of system \eqref{eq:extended} has a two-dimensional center subspace \(E^c\), a two-dimensional stable subspace \(E^s\), and a one-dimensional unstable subspace \(E^u\). The zero eigenvalue is semisimple, and
		\begin{equation}\label{eq:extended-center}
			E^c=\spanop\{(R_0,0),(0,0,0,0,1)^{\mathsf T}\}.
		\end{equation}
		Moreover, there is $\omega_0>0$, depending only on the fixed physical
		parameters, such that
		\begin{equation}\label{eq:extended-gap}
			\Re\Lambda\big|_{E^s}\le-\omega_0<0,
			\qquad
			\Re\Lambda\big|_{E^u}\ge\omega_0>0.
		\end{equation}
	\end{proposition}
	\begin{proof}
		In the coordinates $(Y,\tau)$, the Jacobian of \eqref{eq:extended} is
		\begin{equation}\label{eq:extended-jacobian}
			D_{(Y,\tau)}(\mathscr F(Y;\tau),0)\big|_{(Y_-,c_0)}
			=\begin{pmatrix}
				A(c_0)&\partial_s\mathscr F(Y_-;c_0)\\
				0&0
			\end{pmatrix}.
		\end{equation}
		Remark \ref{rem:left-line} states that $\mathscr F(Y_-;s)=0$ for every $s\ne0$.
		Differentiating this identity in $s$ gives
		$\partial_s\mathscr F(Y_-;c_0)=0$. Hence \eqref{eq:extended-jacobian} reduces to
		the block diagonal matrix
		$$
		\operatorname{diag}(A(c_0),0).
		$$
		Lemma \ref{lem:kernel} gives a simple zero eigenvalue of $A(c_0)$ with
		eigenvector $R_0$, while the appended scalar block gives the independent
		eigenvector $(0,0,0,0,1)^{\mathsf T}$. Since the off-diagonal block in
		\eqref{eq:extended-jacobian} vanishes, neither vector generates a
		generalised eigenvector coupling the two blocks. The double zero is
		therefore semisimple and its eigenspace is \eqref{eq:extended-center}.
		Lemma \ref{lem:inertia} supplies two stable directions and one unstable
		direction. Their finitely many nonzero eigenvalues are separated from
		the imaginary axis, which yields \eqref{eq:extended-gap} after decreasing
		$\omega_0$ if necessary.
	\end{proof}
	
	\subsection{Center-manifold reduction}\label{sec:center}
	Proposition~\ref{prop:extended-spectrum} supplies a two-dimensional center subspace of the speed-extended system \eqref{eq:extended} at $(Y_-,c_0)$, together with a uniform spectral gap on the hyperbolic complement. In this subsection we apply the center-manifold theorem of Appendix~\ref{app:center} and reduce the local dynamics to that manifold. 
	Fix a finite integer $r\geq6$. By the center-manifold theorem in
	Appendix \ref{app:center}, we have
	
	\begin{proposition}\label{prop:center}
		Near $(Y_-,c_0)$, system \eqref{eq:extended} has a $C^r$
		two-dimensional locally invariant manifold parameterised by
		\begin{equation}\label{eq:center-chart}
			(Y,\tau)=(Y(\eta,\tau),\tau),\qquad
			Y(\eta,\tau)=Y_-+\eta R_0+H(\eta,\tau),\qquad H^\rho\equiv0,
		\end{equation}
		\begin{equation}\label{eq:center-tangency}
			H(0,\tau)=0,\qquad DH(0,c_0)=0,
		\end{equation}
		\begin{equation}\label{eq:center-bounds}
			|H(\eta,\tau)|\leq C|\eta|(|\eta|+|\tau-c_0|),\qquad
			|H_\eta(\eta,\tau)|\leq C(|\eta|+|\tau-c_0|).
		\end{equation}
		Here $H^\rho$ is the first component of $H$. There is a fixed smaller neighborhood in which every complete orbit of
		\eqref{eq:extended} lies on this manifold. On the manifold the equations are
		\begin{equation}\label{eq:reduced-tau}
			\begin{aligned}
				&\eta'=g(\eta,\tau),\qquad \tau'=0,
				\qquad g(\eta,\tau)=\mathscr F^\rho(Y(\eta,\tau);\tau),\\
				&\mathscr F(Y(\eta,\tau);\tau)=g(\eta,\tau)Y_\eta(\eta,\tau).
			\end{aligned}
		\end{equation}
		Here $\mathscr F^\rho$ is the first component of $\mathscr F$. In particular, 
		$\rho=1+\eta$ in this chart.
	\end{proposition}
	
	\begin{proof}
		Proposition \ref{prop:extended-spectrum} gives the splitting
		$E^c\oplus E^s\oplus E^u$, with a two-dimensional center space and a
		uniform spectral gap on the hyperbolic space. Theorem \ref{thm:center}
		therefore supplies a two-dimensional $C^r$ locally invariant manifold,
		tangent to $E^c$ at $(Y_-,c_0)$, together with the stated containment
		property.

		It remains to choose coordinates on this manifold. Define
		$
		\Pi(Y,\tau)=(\rho-1,\tau).
		$
		Let $e_5=(0,0,0,0,1)^{\mathsf T}$. By
		\eqref{eq:extended-center} and $R_0^\rho=1$, the restriction
		\[
		D\Pi(Y_-,c_0)\big|_{E^c}:E^c\longrightarrow\R^2
		\]
		maps $(R_0,0)$ and $e_5$ to the standard basis of $\R^2$,
		and is therefore an isomorphism. Since the manifold is tangent
		to $E^c$ at $(Y_-,c_0)$, the inverse function theorem, applied
		to the restriction of $\Pi$ to this manifold, makes
		$(\eta,\tau)=(\rho-1,\tau)$ a local $C^r$ coordinate chart.
		In these coordinates, $\rho=1+\eta$, and the manifold has
		the representation \eqref{eq:center-chart}. In particular,
		$H^\rho\equiv0$ follows from this choice of coordinates.

		For every $\tau$ near $c_0$, the constant complete orbit $(Y_-,\tau)$
		lies on the manifold by the containment property. Its coordinates are
		$(0,\tau)$, so
		\[
		Y(0,\tau)=Y_-,\qquad H(0,\tau)=0,\qquad H_\tau(0,\tau)=0.
		\]
		The tangent $(Y_\eta(0,c_0),0)$ belongs to $E^c$. Its first and fifth
		components are $1$ and $0$, respectively. In the basis
		\eqref{eq:extended-center}, its coefficients are therefore $1$ and $0$,
		and hence
		\[
		Y_\eta(0,c_0)=R_0
		=\bigl(1,(\gamma-1)\theta_0,-1,0\bigr)^{\mathsf T}.
		\]
		Differentiating \eqref{eq:center-chart} with respect to $\eta$ gives
		$H_\eta(0,c_0)=Y_\eta(0,c_0)-R_0=0$. Together with
		$H(0,\tau)=0$ and $H_\tau(0,c_0)=0$, this proves
		\eqref{eq:center-tangency}.
		
		Since $H$ is $C^2$, the mean value estimate on a smaller
		neighborhood yields
		\[
		|H_\eta(\eta,\tau)|
		=|H_\eta(\eta,\tau)-H_\eta(0,c_0)|
		\leq C(|\eta|+|\tau-c_0|).
		\]
		Using $H(0,\tau)=0$, we further obtain
		\[
		\begin{aligned}
			H(\eta,\tau)&=\eta\int_0^1 H_\eta(t\eta,\tau)\,dt,\\
			|H(\eta,\tau)|
			&\leq C|\eta|\int_0^1(t|\eta|+|\tau-c_0|)\,dt
			\leq C|\eta|(|\eta|+|\tau-c_0|),
		\end{aligned}
		\]
		which proves \eqref{eq:center-bounds}.
		
		Along a solution in the locally invariant manifold, differentiate
 parameterization with respect to $\zeta$. The chain rule gives
		\begin{align*}
			\begin{pmatrix}\mathscr F^\rho\\\mathscr F^\theta\\\mathscr F^\phi\\\mathscr F^W\\0\end{pmatrix}=\begin{pmatrix} \rho'\\\theta'\\\phi'\\W'\\\tau'\end{pmatrix}
			=\begin{pmatrix}
				1&0\\ \theta_\eta&\theta_\tau\\ \phi_\eta&\phi_\tau\\
				W_\eta&W_\tau\\0&1
			\end{pmatrix}
			\begin{pmatrix}\eta'\\\tau'\end{pmatrix}.
		\end{align*}
		Here $\mathscr F$ is evaluated at $(Y(\eta,\tau);\tau)$.
		The last row gives $\tau'=0$, and the first gives $\eta'=\mathscr F^\rho$.
		Substitution into the other rows yields
		\[
		\mathscr F(Y(\eta,\tau);\tau)
		=\mathscr F^\rho(Y(\eta,\tau);\tau)
		\begin{pmatrix}1\\\theta_\eta\\\phi_\eta\\W_\eta\end{pmatrix}
		=\mathscr F^\rho(Y(\eta,\tau);\tau)Y_\eta(\eta,\tau).
		\]
		Thus $g=\mathscr F^\rho$ is the unique scalar satisfying $\mathscr F=gY_\eta$, proving
		\eqref{eq:reduced-tau}. Since $\tau'=0$, each fixed-$\tau$ slice is a locally invariant curve of the original four-dimensional system. This completes the proof of Proposition \ref{prop:center}.
	\end{proof}
	
	We now choose the Rankine--Hugoniot speed
	$\tau=s(\varepsilon)$ and write
	$\widehat g(\eta,\varepsilon)=g(\eta,s(\varepsilon))$.
	By \eqref{eq:Hugoniot-expansion},
	\[
	Y_+(\varepsilon)=Y_--\varepsilon R_0+O(\varepsilon^2),
	\]
	so, for sufficiently small $\varepsilon$, both
	$(Y_-,s(\varepsilon))$ and
	$(Y_+(\varepsilon),s(\varepsilon))$ lie on the center manifold.
	Since $\eta=\rho-1$, we have
	\begin{equation}\label{eq:reduced-endpoints}
		Y(0,s(\varepsilon))=Y_-,
		\qquad
		Y(-\varepsilon,s(\varepsilon))=Y_+(\varepsilon),
		\qquad
		\widehat g(0,\varepsilon)
		=\widehat g(-\varepsilon,\varepsilon)=0.
	\end{equation}

	\subsection{The reduced scalar equation and the heteroclinic}\label{sec:heteroclinic}
	Both Rankine--Hugoniot states now lie on the fixed-speed slice $\tau=s(\varepsilon)$ of the center manifold, and $\widehat g(\cdot,\varepsilon)$ vanishes at the endpoints $\eta=0$ and $\eta=-\varepsilon$. A small heteroclinic of \eqref{eq:vector-field} is therefore equivalent to a solution of the scalar equation $\eta'=\widehat g(\eta,\varepsilon)$ connecting $0$ to $-\varepsilon$. In this subsection we show that $\widehat g$ has no interior zero on $[-\varepsilon,0]$, determine its sign from genuine nonlinearity, and construct the unique small-amplitude orbit.
	
	\begin{lemma}\label{lem:no-interior-zero}
		For sufficiently small $\varepsilon>0$,
		$\widehat g(\cdot,\varepsilon)$ vanishes on $[-\varepsilon,0]$ only at the two
		endpoints.
	\end{lemma}
	\begin{proof}
		Suppose $\widehat g(\eta_*,\varepsilon)=0$ for some
		$\eta_*\in[-\varepsilon,0]$, and set
		$Y_*=Y(\eta_*,s(\varepsilon))$. The invariance identity in
		\eqref{eq:reduced-tau} gives $\mathscr F(Y_*;s(\varepsilon))=0$.
		With $Y_*=(\rho_*,\theta_*,\phi_*,W_*)$,
		the vector field \eqref{eq:vector-field} gives
		\begin{equation*}
			W_*=0,\qquad \rho_*=e^{-\phi_*},\qquad F_1(Y_*;s(\varepsilon))=0, \qquad G(Y_*;s(\varepsilon))=0.
		\end{equation*}
		By Lemma
		\ref{lem:equilibria}, $Y_*$ is therefore a quasineutral
		Rankine--Hugoniot state at the fixed speed $s(\varepsilon)$.
		If $\rho_*=1$, then $\eta_*=0$ by the exact density
		coordinate. Suppose now that $\rho_*\ne1$.

		To identify this state, we first write the
		Rankine--Hugoniot conditions relative to the fixed
		left state in scalar form. For a general quasineutral
		state $(\rho,u,\theta,\phi)$ and speed $s$, substitution of
		the mass relation \eqref{eq:mass} gives
		\begin{equation}\notag
			\frac{G}{s}=-c_v(\theta-\theta_0)-\log\rho
			-\frac12\bigl[\rho(\theta+1)+\theta_0+1\bigr](\rho^{-1}-1)
			+\frac{\rho-1}{2\rho}F_1.
		\end{equation}
		Thus $F_1=G=0$ gives the energy relation \eqref{eq:RH-energy} below.
		Together with the mass and momentum relations, we obtain
		\begin{align}
			&u=s(1-\rho^{-1}),\qquad
			\phi=-\log\rho,
			\label{eq:RH-mass}\\
			&s^2(\rho^{-1}-1)+\rho(\theta+1)-(\theta_0+1)=0,
			\label{eq:RH-momentum}\\
			&c_v(\theta-\theta_0)+\log\rho
			+\frac12\bigl[\rho(\theta+1)+\theta_0+1\bigr]
			(\rho^{-1}-1)=0.
			\label{eq:RH-energy}
		\end{align}
		The energy relation
		\eqref{eq:RH-energy} is linear in $\theta$, with coefficient
		$c_v+(1-\rho)/2>0$ near $\rho=1$. Hence every nearby
		Rankine--Hugoniot state satisfies $\theta=\theta_H(\rho)$,
		where
		\begin{equation}\notag
			\theta_H(\rho)=
			\frac{c_v\theta_0-\log\rho-\frac12(1-\rho)
				-\frac{\theta_0+1}{2\rho}(1-\rho)}
			{c_v+\frac12(1-\rho)}.
		\end{equation}
		For $\rho\ne1$, the momentum relation
		\eqref{eq:RH-momentum} then determines the positive speed
		$s_H(\rho)$ through
		\begin{equation}\notag
			s_H(\rho)^2
			=\rho\,\frac{\rho(\theta_H(\rho)+1)-(\theta_0+1)}
			{\rho-1}.
		\end{equation}
		Applying them to the branch of Lemma \ref{lem:Hugoniot}
		gives $s_H(\rho)=s(1-\rho)$. Thus $s_H$ extends smoothly
		to $\rho=1$ with $s_H(1)=c_0$, and by \eqref{111}, we have
		\begin{equation}\notag
			\partial_\rho s_H(1)
			=-\partial_\varepsilon s(0)=\frac{\Gamma}{2}>0.
		\end{equation}
		After shrinking the neighborhood, $s_H$ is therefore
		strictly increasing.
		
		Since $Y_*$ has speed $s(\varepsilon)$, we obtain
		\begin{equation}\notag
			\theta_*=\theta_H(\rho_*),\qquad
			s_H(\rho_*)=s(\varepsilon)=s_H(1-\varepsilon).
		\end{equation}
		Strict monotonicity forces $\rho_*=1-\varepsilon$,
		and hence $\eta_*=-\varepsilon$.
		Thus the two zeros in \eqref{eq:reduced-endpoints}
		exhaust the closed segment, and no interior zero
		is possible.
	\end{proof}
	
	\begin{lemma}\label{lem:scalar-sign}
		For sufficiently small $\varepsilon>0$,
		\begin{equation*}
			\widehat g_\eta(0,\varepsilon)=\mu_c(s(\varepsilon))
			\geq\frac{\Gamma}{2D}\varepsilon>0,
			\qquad \widehat g(\eta,\varepsilon)<0\quad(-\varepsilon<\eta<0).
		\end{equation*}
	\end{lemma}
	\begin{proof}
		From
		\eqref{eq:center-chart}--\eqref{eq:center-bounds}, it is direct to see that
		\begin{equation*}
			Y_\eta(0,s)=R_0+O(|s-c_0|)\ne0.
		\end{equation*}
		Differentiate the invariance identity
		$\mathscr F(Y(\eta,s);s)=g(\eta,s)Y_\eta(\eta,s)$ with respect to $\eta$.
		At $\eta=0$, the identities $Y(0,s)=Y_-$ and
		$\mathscr F(Y_-;s)=0$ give $g(0,s)=0$. The term containing $g$ therefore
		vanishes, and by \eqref{eq:Jacobian}, we obtain
		\begin{equation*}
			A(s)Y_\eta(0,s)=g_\eta(0,s)Y_\eta(0,s).
		\end{equation*}
		Thus $g_\eta(0,s)$ is an eigenvalue of $A(s)$ converging to the simple
		zero eigenvalue of $A(c_0)$. It must be the unique small real branch
		$\mu_c(s)$ from Lemma \ref{lem:inertia}. At the Hugoniot speed,
		\eqref{eq:small-root-shock} yields
		\begin{equation}\label{eq:reduced-left-slope}
			\widehat g_\eta(0,\varepsilon)
			=\frac{\Gamma}{D}\varepsilon+O(\varepsilon^2)
			\geq\frac{\Gamma}{2D}\varepsilon>0
		\end{equation}
		after decreasing the strength bound. Taylor expansion at $\eta=0$ now gives
		\begin{equation}\label{eq:reduced-local-sign}
			\widehat g(\eta,\varepsilon)
			=\widehat g_\eta(0,\varepsilon)\eta+o(|\eta|)<0
			\qquad(\eta<0,\ |\eta|\ \hbox{small}).
		\end{equation}
		The continuous function $\widehat g(\cdot,\varepsilon)$ has no zero in the
		connected interval $(-\varepsilon,0)$ by Lemma
		\ref{lem:no-interior-zero}. It cannot change sign there. The local sign
		in \eqref{eq:reduced-local-sign} therefore holds on the whole interval.
	\end{proof}
	
	\begin{proposition}\label{prop:heteroclinic}
		For sufficiently small $\varepsilon>0$, the scalar initial-value problem
		\begin{equation}\label{eq:scalar-profile}
			\eta'=\widehat g(\eta,\varepsilon),
			\qquad
			\eta(0)=-\varepsilon/2
		\end{equation}
		has a unique complete solution satisfying
		\begin{equation}\label{eq:scalar-limits}
			-\varepsilon<\eta(\zeta)<0,\qquad
			\eta'(\zeta)<0,\qquad
			\lim_{\zeta\to-\infty}\eta(\zeta)=0,\qquad
			\lim_{\zeta\to+\infty}\eta(\zeta)=-\varepsilon.
		\end{equation}
		Moreover, the heteroclinic solution is unique up to translation.
	\end{proposition}
	
	\begin{proof}
		Since $\widehat g(\cdot,\varepsilon)$ is $C^1$ on a neighborhood of
		$[-\varepsilon,0]$, the initial-value problem \eqref{eq:scalar-profile}
		has a unique maximal solution. The endpoint identities
		\eqref{eq:reduced-endpoints} and uniqueness prevent this solution from
		reaching either $0$ or $-\varepsilon$ in finite time. Hence
		$-\varepsilon<\eta<0$, and Lemma \ref{lem:scalar-sign} gives
		$\eta'<0$. Since the solution remains in the compact interval
		$[-\varepsilon,0]$, the continuation theorem extends it to all
		$\zeta\in\R$.
		
		Monotonicity and boundedness give limits
		\begin{equation*}
			L_-:=\lim_{\zeta\to-\infty}\eta(\zeta),\qquad
			L_+:=\lim_{\zeta\to+\infty}\eta(\zeta),\qquad
			L_\pm\in[-\varepsilon,0].
		\end{equation*}
		Both limits must be zeros of $\widehat g(\cdot,\varepsilon)$;
		otherwise $|\eta'|$ would be bounded away from zero near the
		corresponding limit, contradicting convergence. By Lemma
		\ref{lem:no-interior-zero}, strict decrease, and
		$\eta(0)=-\varepsilon/2$, we obtain $L_-=0$ and $L_+=-\varepsilon$.
		This proves \eqref{eq:scalar-limits}.
		
		Let $\widetilde\eta$ be another complete solution with the same
		end-state limits. By continuity, there exists $\zeta_0\in\R$ such that
		$\widetilde\eta(\zeta_0)=-\varepsilon/2$. Uniqueness for
		\eqref{eq:scalar-profile} then gives
		$\widetilde\eta(\zeta+\zeta_0)=\eta(\zeta)$ for all $\zeta\in\R$.
		This proves uniqueness up to translation.
	\end{proof}
	
	We are ready to give the construction of the traveling-wave profile.
	Let $\eta$ be the solution in Proposition \ref{prop:heteroclinic}, and set
	\[
	\bar Y(\zeta):=Y(\eta(\zeta),s(\varepsilon))
	=(\rho^S,\theta^S,\phi^S,W^S)(\zeta).
	\]
	By \eqref{eq:reduced-tau} and \eqref{eq:scalar-profile},
	\begin{equation}\label{eq:lift-solves-system}
		\bar Y'=Y_\eta(\eta,s(\varepsilon))\eta'
		=\widehat g(\eta,\varepsilon)Y_\eta(\eta,s(\varepsilon))
		=\mathscr F(\bar Y;s(\varepsilon)).
	\end{equation}
	The limits in \eqref{eq:scalar-limits} and the endpoint identities
	\eqref{eq:reduced-endpoints} show that $\bar Y$ connects
	$Y_-$ to $Y_+(\varepsilon)$. Since
	$\rho^S=1+\eta\in(1-\varepsilon,1)$ and $\mathscr F$ is smooth on
	$\{\rho>0\}$, equation \eqref{eq:lift-solves-system} yields
	$\bar Y\in C^\infty(\R)$. With
	$u^S=s(\varepsilon)(1-1/\rho^S)$ given by the mass relation
	\eqref{eq:mass}, Lemma \ref{lem:equilibria} shows that
	$(\rho^S,u^S,\theta^S,\phi^S)$ is a smooth traveling-wave profile
	for \eqref{eq:NSFP} with speed $s(\varepsilon)$.
	
	\subsection{Qualitative properties and uniqueness}\label{sec:qualitative}
	The preceding subsection produces a smooth traveling-wave profile $(\rho^S,u^S,\theta^S,\phi^S)$ of speed $s(\varepsilon)$ connecting $Y_-$ to $Y_+(\varepsilon)$. In this subsection we record its monotonicity, a uniform lower bound on the temperature, the exponential decay to the end states, and uniqueness up to translation in a neighborhood of $Y_-$.
	\begin{lemma}\label{lem:monotonicity}
		The traveling-wave profile constructed above satisfies
		$(\rho^S)'<0$, $(\phi^S)'>0$, and
		$C^{-1}|(\rho^S)'|\leq|(\phi^S)'|\leq C|(\rho^S)'|$,
		with $C$ independent of small $\varepsilon>0$.
	\end{lemma}
	\begin{proof}
		The exact density coordinate and \eqref{eq:center-chart} give
		\begin{equation}
			 \label{eq:phi-monotonicity}
			{\rho^S}=1+\eta,\qquad
			{\phi^S}=-\eta+H^\phi(\eta,s(\varepsilon)),
			\qquad
			(\rho^S)'=\eta',\quad
			(\phi^S)'=\bigl(-1+H^\phi_\eta\bigr)\eta'.
		\end{equation}
		On the orbit, $|\eta|+|s(\varepsilon)-c_0|\leq C\varepsilon$.
		The second estimate in \eqref{eq:center-bounds} therefore gives
		$|H^\phi_\eta|\leq C\varepsilon$ uniformly on
		$[-\varepsilon,0]$. Decrease the strength bound until
		$|H^\phi_\eta|\leq1/2$. Proposition \ref{prop:heteroclinic} gives
		$\eta'<0$, and hence
		\begin{equation}\notag
			(\rho^S)'<0,\qquad (\phi^S)'>0,
			\qquad
			\frac12|(\rho^S)'|\leq|(\phi^S)'|
			\leq\frac32|(\rho^S)'|.
		\end{equation}
		Thus, we proved Lemma \ref{lem:monotonicity}.
	\end{proof}
	
	\begin{lemma}\label{lem:temperature}
		For sufficiently small $\varepsilon>0$,  it holds ${\theta^S}\geq\theta_0/2$.
	\end{lemma}
	\begin{proof}
		Since the temperature component of $R_0$ is
		$(\gamma-1)\theta_0$, \eqref{eq:center-chart} gives
		\begin{equation}\notag
			{\theta^S}=\theta_0+(\gamma-1)\theta_0\eta
			+H^\theta(\eta,s(\varepsilon)),
			\qquad
			|H^\theta|\leq C|\eta|
			\bigl(|\eta|+|s(\varepsilon)-c_0|\bigr)
			\leq C\varepsilon^2.
		\end{equation}
		Because $-\varepsilon\leq\eta\leq0$, this implies
		$|{\theta^S}-\theta_0|\leq C\varepsilon$. Decrease the strength bound
		so that $C\varepsilon\leq\theta_0/2$.
	\end{proof}
	
	\begin{lemma}\label{lem:decay}
		With ${\rho^S}(0)=1-\varepsilon/2$, the estimates \eqref{eq:decay}
		hold for every integer $k\geq0$, with a common strength bound
		$\varepsilon_{*}$ and decay rate $\vartheta>0$. In addition,
		\begin{equation}\label{eq:electric-decay}
			\begin{aligned}
				|\partial_\zeta^k{W^S}(\zeta)|
				&\leq C_k\varepsilon^{k+2}e^{-\vartheta\varepsilon|\zeta|},\\
				|\partial_\zeta^k({\rho^S}-e^{-{\phi^S}})(\zeta)|
				&\leq C_k\la^2\varepsilon^{k+3}e^{-\vartheta\varepsilon|\zeta|}.
			\end{aligned}
		\end{equation}
	\end{lemma}
	\begin{proof}
		Since $\widehat g(0,\varepsilon)=\widehat g(-\varepsilon,\varepsilon)=0$ in \eqref{eq:reduced-endpoints}, by Hadamard's lemma,  we have
		\begin{equation}\label{eq:factorisation}
			\widehat g(\eta,\varepsilon)=\eta(\eta+\varepsilon)b(\eta,\varepsilon),
			\qquad b\in C^{r-2}.
		\end{equation}
		Differentiating \eqref{eq:factorisation} at the two endpoints with \eqref{eq:reduced-left-slope} gives
		\begin{equation}\notag
			\begin{aligned}
				&\varepsilon b(0,\varepsilon)=\widehat g_\eta(0,\varepsilon)
				=
				\frac{\Gamma}{D}\varepsilon+O(\varepsilon^2),\qquad \widehat g_\eta(-\varepsilon,\varepsilon)
				=-\varepsilon b(-\varepsilon,\varepsilon).
			\end{aligned}
		\end{equation}
		The first equality and \eqref{eq:small-eigenvalue} show that
		$b(0,0)=\Gamma/D>0$. By continuity, after decreasing the fixed
		neighborhood and the strength bound,
		\begin{equation*}
			\frac{\Gamma}{2D}\leq b(\eta,\varepsilon)\leq\frac{3\Gamma}{2D}
			\qquad(-\varepsilon\leq\eta\leq0).
		\end{equation*}
		
		Fix the translation by $\eta(0)=-\varepsilon/2$. By Proposition \ref{prop:heteroclinic}, for $\zeta\leq0$ one
		has $\eta+\varepsilon\ge\varepsilon/2$, and for $\zeta\ge0$, one
		has $\eta\le-\varepsilon/2$. Thus,  \eqref{eq:factorisation} gives
		\begin{equation}\notag
			\frac{(-\eta)'}{-\eta}=(\eta+\varepsilon)b
			\geq\frac{\Gamma}{4D}\varepsilon\quad(\zeta\leq0),
			\qquad
			\frac{(\eta+\varepsilon)'}{\eta+\varepsilon}=\eta b
			\leq-\frac{\Gamma}{4D}\varepsilon\quad(\zeta\geq0).
		\end{equation}
		Integrating the two inequalities from the midpoint and taking
		$\vartheta=\Gamma/(4D)$ yields
		\begin{equation}\label{eq:eta-decay}
			\begin{gathered}
				|\eta(\zeta)-\eta_\pm|
				\leq\frac{\varepsilon}{2}e^{-\vartheta\varepsilon|\zeta|},
				\qquad
				|\eta'(\zeta)|\leq C\varepsilon^2e^{-\vartheta\varepsilon|\zeta|},\\
				\pm\zeta\geq0,\qquad \eta_-=0,\qquad\eta_+=-\varepsilon.
			\end{gathered}
		\end{equation}
		
		We next justify every power of $\varepsilon$ in the higher derivatives.
		On the whole segment, Proposition \ref{prop:heteroclinic} and \eqref{eq:factorisation} give
		\begin{equation*}
			|\widehat g_\eta(\eta,\varepsilon)|\leq C\varepsilon,
			\qquad |\partial_\eta^m \widehat g(\eta,\varepsilon)|\leq C_m
			\quad(m\geq2).
		\end{equation*}
		The first two differentiations illustrate the induction:
		\begin{equation*}
			\eta''=\widehat g_\eta\eta',
			\qquad
			\eta'''=\widehat g_{\eta\eta}(\eta')^2+\widehat g_\eta\eta''.
		\end{equation*}
		Suppose the desired estimate is known below order $k$. In
		$\partial_\zeta^k\eta=\partial_\zeta^{k-1}\widehat g(\eta,\varepsilon)$, the term
		linear in the highest derivative is
		$\widehat g_\eta\partial_\zeta^{k-1}\eta$ and is bounded by
		$C_k\varepsilon^{k+1}e^{-\vartheta\varepsilon|\zeta|}$.
		Every remaining term contains a bounded derivative
		$\partial_\eta^m \widehat g$ with $m\geq2$ and a product of $m$ derivatives of
		$\eta$, whose positive orders add to $k-1$. The induction hypothesis
		gives the power
		$\varepsilon^{(k-1)+m}\leq\varepsilon^{k+1}$ and at least the same
		exponential factor. Starting with the estimate for $\eta'$ in
		\eqref{eq:eta-decay}, induction proves
		\begin{equation}\label{eq:eta-higher}
			|\partial_\zeta^k\eta(\zeta)|
			\leq C_k\varepsilon^{k+1}e^{-\vartheta\varepsilon|\zeta|},
			\qquad k\geq1.
		\end{equation}
		This is valid for every order allowed by the chosen finite regularity.
		
		We now transfer the bounds to the full profile. For $k=0$, boundedness
		of $Y_\eta$ and the exact endpoint identities imply
		$|\bar Y-Y_\pm|\leq C|\eta-\eta_\pm|$. For $k\geq1$, differentiating
		$Y(\eta(\zeta),s(\varepsilon))$ produces the leading term
		$Y_\eta\partial_\zeta^k\eta$; every other term contains at least two
		lower derivatives of $\eta$ and therefore has an additional power of
		$\varepsilon$. Thus \eqref{eq:eta-higher} gives the asserted bounds for
		${\rho^S},{\theta^S}$ and ${\phi^S}$. Since ${\rho^S}\geq1/2$, the
		mass formula ${u^S}=s(1-1/{\rho^S})$ transfers the same bounds to
		${u^S}$. This proves \eqref{eq:decay} to the same finite order.
		
		To obtain all finite orders with one $\varepsilon_{*}$, fix the
		low-regularity construction above. Given $k$, choose a $C^r$ center
		manifold with $r\geq k+4$. For $0<\varepsilon<\varepsilon_k$, its
		containment property places the already constructed complete profile on
		this manifold, so the preceding calculation applies to order $k$.
		For $\varepsilon_k\leq\varepsilon\leq\varepsilon_{*}$, successive
		$\zeta$-derivatives along the smooth original ODE are smooth functions of
		$(Y,s)$ that vanish at each equilibrium. On a fixed compact
		neighborhood the mean-value theorem bounds them by
		$C_k|Y-Y_\pm|$. The zeroth-order estimate therefore gives
		$C_k\varepsilon e^{-\vartheta\varepsilon|\zeta|}$, which is bounded by
		$C_k\varepsilon_k^{-k}\varepsilon^{k+1}
		e^{-\vartheta\varepsilon|\zeta|}$. Thus only $C_k$, not
		$\varepsilon_{*}$ or $\vartheta$, changes with $k$.
		
		Finally, ${W^S}=(\phi^S)'$ and
		${\rho^S}-e^{-{\phi^S}}=\la^2(\phi^S)''$. Applying
		\eqref{eq:decay} at orders $k+1$ and $k+2$ gives
		\eqref{eq:electric-decay}.
	\end{proof}
	
	\begin{lemma}\label{lem:uniqueness}
		There is a fixed neighborhood $\mathcal N$ of $Y_-$ such that
		any two smooth traveling-wave profiles with the prescribed
		end states, whose phase curves remain in $\mathcal N$ and
		converge to $Y_-$ and $Y_+(\varepsilon)$ at the two ends,
		differ only by translation.
	\end{lemma}
	
	\begin{proof}
		Let $\mathcal U$ be the neighborhood in the containment
		assertion of Proposition \ref{prop:center}. Choose a fixed
		neighborhood $\mathcal N$ of $Y_-$ and an interval $I$
		containing $c_0$ such that $\mathcal N\times I\subset\mathcal U$.
		For sufficiently small $\varepsilon>0$, we have
		$s(\varepsilon)\in I$.
		Let $(\widetilde\rho,\widetilde u,\widetilde\theta,
		\widetilde\phi)$ be any profile satisfying the hypotheses,
		with speed $\widetilde s$, and write
		$
		\widetilde Y
		=
		(\widetilde\rho,\widetilde\theta,
		\widetilde\phi,\widetilde\phi').
		$
		The mass Rankine--Hugoniot condition gives
		$
		\widetilde s(\rho_+-1)=\rho_+u_+,$ and $
		\widetilde s
		=\frac{\rho_+u_+}{\rho_+-1}
		=s(\varepsilon),
		$
		since $\rho_+-1=-\varepsilon\ne0$.
		
		Integrating the mass equation gives \eqref{eq:mass}.
		The integrated momentum equation and the convergence of
		$\widetilde Y$ imply that $\widetilde u'$ has a limit at
		$-\infty$. This limit must be zero because $\widetilde u$
		converges, so the integration constant is fixed by the left
		state as in \eqref{eq:integrated-momentum}. The integrated
		energy equation similarly gives $\widetilde\theta'\to0$
		at $-\infty$ and fixes the remaining constant as in
		\eqref{eq:theta-ode}. Together with Poisson's equation,
		these identities show that
		$
		\widetilde Y'
		=
		\mathscr F(\widetilde Y;s(\varepsilon)).
		$
		
		Thus $(\widetilde Y(\zeta),s(\varepsilon))$ is a complete
		orbit of \eqref{eq:extended} contained in $\mathcal U$.
		By Proposition \ref{prop:center}, it lies on the center
		manifold constructed there. Hence, with
		$\widetilde\eta=\widetilde\rho-1$,
		\[
		\widetilde Y(\zeta)
		=
		Y(\widetilde\eta(\zeta),s(\varepsilon)),
		\quad
		\widetilde\eta'
		=
		\widehat g(\widetilde\eta,\varepsilon),\quad\lim_{\zeta\to-\infty}\widetilde\eta(\zeta)=0,
		\quad
		\lim_{\zeta\to+\infty}\widetilde\eta(\zeta)=-\varepsilon.
		\]
		By Proposition \ref{prop:heteroclinic},
		$\widetilde\eta$ is a translate of the scalar solution
		constructed there. The common parameterization
		$Y(\eta,s(\varepsilon))$ and the mass relation
		\eqref{eq:mass} then show that
		$(\widetilde\rho,\widetilde u,\widetilde\theta,\widetilde\phi)$
		is the same translate of
		$(\rho^S,u^S,\theta^S,\phi^S)$.
		Thus any two such profiles differ only by translation.
	\end{proof}

	\subsection{Proof of Theorem \ref{thm:existence}}
	\label{sec:proof-main}
	
	Take $\varepsilon_*>0$ sufficiently small for the preceding
	results to apply. By \eqref{eq:RH-mass}--\eqref{eq:RH-energy}
	and Lemma \ref{lem:Hugoniot}, the prescribed right state is
	$Y_+(\varepsilon)$ with speed $s(\varepsilon)$.
	Proposition \ref{prop:heteroclinic} and the construction following
	its proof give the required smooth traveling-wave profile.
	Lemmas \ref{lem:monotonicity}--\ref{lem:decay} yield
	\eqref{eq:qualitative}--\eqref{eq:decay}, while
	Lemma \ref{lem:uniqueness} gives uniqueness up to translation
	in $\mathcal N$. The uniformity assertions follow from these
	results; in particular, Lemma \ref{lem:decay} provides a common
	strength bound and decay rate for all finite derivative orders.
	This completes the proof. \hfill$\square$

\medskip
To conclude this section, we list three restrictions implicit in Theorem~\ref{thm:existence}.

	\begin{remark}
		The assumption $\theta_0>0$ is used in the physical state space and
		the uniform lower bound for temperature. A cold left state with $\theta_0=0$  is not
		covered by Theorem \ref{thm:existence}.
	\end{remark}
	
	\begin{remark}
		Both the construction and uniqueness are local in the full phase space
		$(\rho,\theta,\phi,W)$. No large-amplitude existence or global
		uniqueness assertion is made.
	\end{remark}
	
	\begin{remark}
		The Debye length $\la>0$ is fixed. Although the leading scalar
		coefficient $\Gamma/D$ is independent of $\la$, the neighborhoods
		and remainder bounds need not be. A uniform quasineutral limit requires
		additional estimates.
	\end{remark}
	
	\section{Asymptotic stability for the shock profile}\label{sec:AC}
	
	This section proves Theorem~\ref{thm:AC}. In Section~\ref{sec:reference} the system is written in Lagrangian mass coordinates and the shock profile of Theorem~\ref{thm:existence} is transferred to those coordinates. Section~\ref{sec:shift} introduces the time-dependent shift $X(t)$ of the reference profile. The a priori framework is stated in Section~\ref{sec:apriori}. Section~\ref{sec:relative} derives the weighted relative-entropy identity, including the electric-energy correction. Section~\ref{sec:decomp} decomposes the identity into good terms, remainder terms and shift terms, and fixes the shift ODE. The leading quadratic contributions of the hyperbolic flux, the shift dissipation and the viscous-thermal diffusion associated with $P_{\mathrm{eff}}$ are estimated in Section~\ref{sec:leading}. Section~\ref{sec:zeroth} controls the remaining lower-order terms and yields the zeroth-order energy inequality. Poisson estimates in Section~\ref{sec:poisson} absorb the electric remainders. First- and second-order fluid energies are obtained in Section~\ref{sec:derivatives}. Section~\ref{sec:closure} closes the a priori estimate and concludes global existence together with the orbital stability of the profile.
	
	\subsection{Lagrangian formulation and shock profiles}\label{sec:reference}
	Starting from \eqref{eq:NSP-1}, we introduce the moving coordinate
	$\xi=x-\sigma t$ and retain the same notation for the transformed unknowns.
	Then $\partial_t$ becomes $\partial_t-\sigma\partial_\xi$ and
	$\partial_x$ becomes $\partial_\xi$, so the system reads
	\begin{equation}\label{eq:NSP}
		\left\{\begin{aligned}
			&v_t-\sigma v_\xi-u_\xi=0,\\
			&u_t-\sigma u_\xi+p_\xi=\mu\left(\frac{u_\xi}{v}\right)_\xi
			+\frac{\phi_\xi}{v},\qquad p=\frac\theta v,\\
			&\cv(\theta_t-\sigma\theta_\xi)+p u_\xi
			=\kappa\left(\frac{\theta_\xi}{v}\right)_\xi+\mu\frac{u_\xi^2}{v},\\
			&\la^2\left(\frac{\phi_\xi}{v}\right)_\xi=1-ve^{-\phi}.
		\end{aligned}\right.
	\end{equation}
	Due to the fact that
	\begin{align}\notag
		&\frac{\phi_\xi}{v}
		=\left[\frac{\la^2}{2}\left(\frac{\phi_\xi}{v}\right)^2-e^{-\phi}\right]_\xi,\qquad
		\pp(v,\theta):=\Peff(v^{-1},\theta)=\frac{\theta+1}{v},
	\end{align}
	the momentum equation in \eqref{eq:NSP}$_2$ can be written as 
	\begin{align}
		u_t-\sigma u_\xi+(\pp)_\xi
		&=\mu\left(\frac{u_\xi}{v}\right)_\xi+F_\xi,
		\label{eq:effective}
	\end{align}
	where
	\begin{align}       
		F=\frac1v-e^{-\phi}+\frac{\la^2\phi_\xi^2}{2v^2}
		=\la^2\left(\frac{\phi_{\xi\xi}}{v^2}
		-\frac{v_\xi\phi_\xi}{v^3}+\frac{\phi_\xi^2}{2v^2}\right).\label{eq:stress}
	\end{align}
	The mass equation implies that $(s-u^S)/v^S=\sigma$ along the shock.
	The Lagrangian mass transformation gives the following relations between
	$\xi$ and the Eulerian traveling coordinate $\zeta$:
	\begin{align}
		&\sigma=\frac{s-u_-}{v_-}=\frac{s-u_+}{v_+},\qquad
		\frac{\dd\xi}{\dd\zeta}=\frac1{v^S},\qquad
		\frac{\dd}{\dd\xi}=v^S\frac{\dd}{\dd\zeta}.\label{eq:speed}
	\end{align}
	Without loss of generality, the midpoint condition fixes the translation of the profile:
	\begin{align}
		&(v^S,u^S,\theta^S,\phi^S)(-\infty)=(v_-,u_-,\theta_-,\log v_-),\quad
		(v^S,u^S,\theta^S,\phi^S)(+\infty)=(v_+,u_+,\theta_+,\log v_+),\notag
		\\
		&\delta_S=v_+-v_->0,\quad \sigma>0,\quad
		(v^S)' >0,\quad (u^S)'=-\sigma(v^S)',\quad
		v^S(0)=\frac{v_-+v_+}{2}.\label{eq:shock}
	\end{align}
	The profiles in \eqref{eq:shock} are supplied by Theorem \ref{thm:existence}:
	\begin{align}\label{eq:profile-identification}
		&v^S(\xi):=\frac{1}{\rho^S(\zeta(\xi))},\qquad
		(u^S,\theta^S,\phi^S)(\xi):=(u^S,\theta^S,\phi^S)(\zeta(\xi)),\notag\\
		&(v_-,u_-,\theta_-,\phi_-)=(1,0,\theta_0,0),\qquad
		\delta_S=v_+-v_-=\frac{\varepsilon}{1-\varepsilon},\qquad
		P(v,\theta)=\Peff(v^{-1},\theta).
	\end{align}
	For the profile estimates below, set
	\begin{align}
		&
		\sigma^*=\frac{\sqrt{\gamma\theta_-+1}}{v_-},\qquad
		\alpha^*=\frac{\gamma(\gamma+1)\theta_-+2}{2(v_-)^3\sigma^*},\label{eq:constants}
	\end{align}
	where \(\sigma^*\) is the limiting wave speed in mass coordinates, while \(\alpha^*\) measures the quadratic variation of the pressure.
	The fixed smooth center manifold in Proposition~\ref{prop:center},
	together with Lemma \ref{lem:decay} and \eqref{eq:speed}, gives
	\begin{align}
		&|(v^S,\theta^S)-(v_-,\theta_-)|+|\sigma-\sigma^*|\le C\delta_S,
		\qquad 0<(v^S)'\le C\delta_S^2e^{-c\delta_S|\xi|},\label{eq:profilebounds}\\
		&|\partial_\xi^j(v^S,u^S,\theta^S,\phi^S)|
		\le C\delta_S^{j-1}(v^S)',\quad 1\le j\le3,
		\label{eq:profilederivatives}\\
		&\left|\frac{(\theta^S)'}{(v^S)'}
		+\frac{(\gamma-1)\theta_-}{v_-}\right|\le C\delta_S,
		\quad |v^S-e^{\phi^S}|\le C\delta_S(v^S)',
		\quad |(v^S-e^{\phi^S})'|\le C\delta_S^2(v^S)'.
		\label{eq:profileslopes}
	\end{align}
	The sharp diffusion estimates used in the stability proof are established in Lemma \ref{lem:sharp} in Appendix~\ref{app:sharp}.
	
	We study perturbations of the  shock profile in \eqref{eq:profile-identification} translated by $X(t)$ to be defined in \eqref{eq:shiftODE} later:
	\begin{equation}\label{eq:shiftedprofile}
		(\bar v,\bar u,\bar\theta,\bar\phi)(t,\xi)
		=(v^S,u^S,\theta^S,\phi^S)(\xi-X(t)).
	\end{equation}
	With $\bar p=\bar\theta/\bar v$, $\bar\pp=(\bar\theta+1)/\bar v$
	and $\bar F=F(\bar v,\bar\phi)$, the profile \eqref{eq:shiftedprofile} satisfies
	\begin{equation}\label{eq:shiftedNSP}
		\left\{\begin{aligned}
			&\bar v_t-\sigma\bar v_\xi+\dot X\bar v_\xi-\bar u_\xi=0,\\
			&\bar u_t-\sigma\bar u_\xi+\dot X\bar u_\xi+\bar p_\xi
			=\mu\left(\frac{\bar u_\xi}{\bar v}\right)_\xi+\frac{\bar\phi_\xi}{\bar v},\\
			&\cv(\bar\theta_t-\sigma\bar\theta_\xi+\dot X\bar\theta_\xi)
			+\bar p\bar u_\xi
			=\kappa\left(\frac{\bar\theta_\xi}{\bar v}\right)_\xi
			+\mu\frac{\bar u_\xi^2}{\bar v},\\
			&\la^2\left(\frac{\bar\phi_\xi}{\bar v}\right)_\xi=1-\bar v e^{-\bar\phi}.
		\end{aligned}\right.
	\end{equation}
	The perturbations are
	\begin{equation}\label{eq:perturbations}
		\tv=v-\bar v,\qquad \tu=u-\bar u,\qquad
		\tz=\theta-\bar\theta,\qquad \tp=\phi-\bar\phi.
	\end{equation}
	The standard local theory \cite{KKS,KVW} gives a unique strong solution of \eqref{eq:NSP}
	for compatible $H^2$ perturbations with positive $v,\theta$, which continues
	as long as $v,\theta$ stay bounded away from zero and the fluid perturbation
	remains bounded in $H^2$ on finite time intervals.
	
	\subsection{Construction of the shift}\label{sec:shift}
	The shift in \eqref{eq:shiftedprofile} is chosen as in
	Subsection~\ref{sec:shiftchoice} and solves
	\begin{equation}\label{eq:shiftODE}
		\left\{\begin{aligned}
			\dot X(t)&=-\frac M{\delta_S}\intR a(\xi-X(t))\left[
			\bar u_\xi\tu
			+\frac{\bar\theta+1}{\bar v^2}\bar v_\xi\tilde v
			+\frac{\bar\theta_\xi}{(\gamma-1)\bar\theta} \tilde \theta
			\right]\dd\xi,\\
			X(0)&=0,
		\end{aligned}\right.
	\end{equation}
	where the weight $a$ is defined in \eqref{eq:weight}, and
	\begin{equation}\label{eq:M}
		M=\frac{3\alpha^*}{2(\sigma^*)^2}
		\frac{\mu(\gamma\theta_-+1)+2\kappa(\gamma-1)^2\theta_-}
		{\mu(\gamma\theta_-+1)+\kappa(\gamma-1)^2\theta_-}>0.
	\end{equation}
	The weight satisfies $1\le a\le1+\sqrt{\delta_S}\le2$.
	The choice \eqref{eq:shiftODE} extracts the linear shift terms in
	\eqref{eq:Y123}; its dissipative contribution is computed in
	\eqref{eq:shiftidentity}. The existence and bounds for $X(t)$ are also
	given in Subsection~\ref{sec:shiftchoice}.
	
	\subsection{Statement of the a priori estimates}\label{sec:apriori}
	The key estimate for Theorem \ref{thm:AC} is stated below.
	\begin{proposition}\label{thm:closure}
		For the left state as in
		Subsection~\ref{sec:reference}, there exist $\delta_0,\varepsilon_1,C>0$,
		independent of $T$, such that the following holds. Let $(v,u,\theta,\phi)$
		be a strong solution of \eqref{eq:NSP} on $[0,T]$ with compatible Poisson
		data, and let $(\bar v,\bar u,\bar\theta,\bar\phi)$ be the profile
		\eqref{eq:shiftedprofile} with the shift $X(t)$ and being
		the absolutely continuous solution to 
		\eqref{eq:shiftODE} with weight function $a$ defined in \eqref{eq:weight}. Assume that the
		perturbations \eqref{eq:perturbations} satisfy
		\[
		\begin{aligned}
			\tv&\in C([0,T];H^2(\R)),\quad
			(\tu,\tz)\in C([0,T];H^2(\R))\cap L^2(0,T;H^3(\R)),\quad
			\tp\in C([0,T];H^3(\R)),
		\end{aligned}
		\]
		and
		\begin{equation}\label{eq:smallness}
			\begin{gathered}
				0<\delta_S\le\delta_0,\qquad
				2\sqrt{\delta_S}\le\delta_0,\\
				\sup_{0\le t\le T}\norm{(\tv,\tu ,\tz,\tp)(t)}_{H^2(\R)}\le\varepsilon_1,
				\qquad \delta_0,\varepsilon_1\ll1.
			\end{gathered}
		\end{equation}
		Then, for $0\le t\le T$,
		\begin{equation}\label{eq:closedestimate}
			{\begin{aligned}
					&\sup_{0\le s\le t}
					\left[\norm{(\tv,\tu ,\tz)(s)}_{H^2}^2+\norm{\tp(s)}_{H^3}^2\right]\\
					&\quad+\int_0^t\left[\delta_S|\dot{X}|^2+\Gs+\G_1+\G_2
					+\norm{\tv_\xi}_{H^1}^2+\norm{\tu _\xi,\tz_\xi}_{H^2}^2
					+\norm{\tp_\xi}_{H^2}^2+\norm\chi_{H^2}^2\right]\dd s\\
					&\le C\norm{(\tv,\tu ,\tz)(0)}_{H^2}^2,
			\end{aligned}}
		\end{equation}
		where $\Gs,\G_1,\G_2$ and $\chi$ are defined in
		\eqref{eq:gooddef}, \eqref{eq:goodsquares} and \eqref{eq:chi}, respectively.
	\end{proposition}
	
	The proof of Proposition~\ref{thm:closure} is completed in
	Subsection~\ref{sec:closure} by combining the relative-entropy and
	derivative estimates. We first establish the basic energy estimate in the following subsection.
	
	\subsection{Relative-entropy estimates}\label{sec:relative}	
	\begin{lemma}\label{lem:basicenergy}
		Under the hypotheses of Proposition~\ref{thm:closure}, there is a constant
		$C>0$, independent of $\delta_0,\varepsilon_1,T$, such that for $0\le t\le T$,
		\begin{equation}\label{eq:basicenergy}
			\begin{aligned}
				&\sup_{0\le s\le t}\left[
				\norm{(\tv,\tu,\tz)(s)}_2^2
				+\la^2\norm{\tp_\xi(s)}_2^2+\la^4\norm{\tp_{\xi\xi}(s)}_2^2\right]\\
				&\quad+\int_0^t\left[\delta_S|\dot X|^2+\Gs+\G_1+\G_2
				+\norm{(\tu_\xi,\tz_\xi)}_2^2\right]\dd s\\
				&\le C\left[\norm{(\tv,\tu,\tz)(0)}_2^2
				+\la^2\norm{\tp_\xi(0)}_2^2+\la^4\norm{\tp_{\xi\xi}(0)}_2^2\right]\\
				&\quad+C(\sqrt{\delta_0}+\varepsilon_1)\int_0^t
				\left(\norm{\tv_\xi}_2^2+\norm{\tu_{\xi\xi}}_2^2\right)\dd s,
			\end{aligned}
		\end{equation}
		where $\Gs,\G_1,\G_2$ are defined in
		\eqref{eq:gooddef}--\eqref{eq:goodsquares}.
	\end{lemma}
	We prove Lemma \ref{lem:basicenergy} by $a$-contraction method  to obtain energy estimate; the Poisson estimates in
	Subsection~\ref{sec:poisson} control the electric remainder. We define the weight by
	\begin{equation}\label{eq:weight}
		a(\xi)=1+\frac{1}{\sqrt{\delta_S}}(v^S(\xi)-v_-),\qquad
		a^{-X}(t,\xi)=a(\xi-X(t)),
	\end{equation}
	with $\delta_S$ satisfying \eqref{eq:smallness}. For convenience, we
	abbreviate $a^{-X}$ to $a$. By \eqref{eq:shock},
	\begin{equation}\label{eq:weightderivative}
		a_\xi=\frac{1}{\sqrt{\delta_S}}\bar v_\xi>0,\qquad
		a_t=-\dot{X}a_\xi,\qquad
		1\le a\le1+\sqrt{\delta_S}.
	\end{equation}
	For the convex function $\Phi(z)=z-1-\log z$, the 
	relative-entropy for fluid, the electric correction and the weighted energy are
	\begin{align}
		H_0&=\frac{\tu ^2}{2}+(\bar\theta+1)\Phi\left(\frac v{\bar v}\right)
		+\frac{\bar\theta}{\gamma-1}\Phi\left(\frac\theta{\bar\theta}\right),\label{eq:H0}\\
		\eta_\phi
		&=\frac{\la^2}{\bar v^2}\tv\tp_{\xi\xi}
		+\frac{\la^4e^{\bar\phi}}{2\bar v^3}\tp_{\xi\xi}^2
		+\frac{\la^2e^{\bar\phi}}{2\bar v^2}\tp_\xi^2,
		\qquad\E(t)=\intR a(H_0+\eta_\phi)\dd\xi.
		\label{eq:entropy}
	\end{align}
	For small $\delta_S$ and
	$\norm{(\tv,\tz)}_\infty$, direct Taylor expansion gives
	\begin{equation}\label{eq:coercive}
		C^{-1}(\tv^2+\tu ^2+\tz^2+\la^2\tp_\xi^2+\la^4\tp_{\xi\xi}^2)
		\le H_0+\eta_\phi
		\le C(\tv^2+\tu ^2+\tz^2+\la^2\tp_\xi^2+\la^4\tp_{\xi\xi}^2).
	\end{equation}

	\begin{lemma}\label{lem:identity}
		Let $a$ be given by \eqref{eq:weight}. Under the hypotheses of Proposition~\ref{thm:closure}, for a strong solution of
		\eqref{eq:NSP} and the profile \eqref{eq:shiftedprofile}, the energy $\mathcal E(t)$ \eqref{eq:entropy} satisfies
		\begin{equation}\label{eq:identity}
			{\quad\frac{\dd}{\dd t}\E
				=\dot{X}\mathcal Y+\mathcal J^{\rm bad}-\mathcal J^{\rm good},\quad}
		\end{equation}
		\begin{equation}\label{eq:Jfull}
			\mathcal Y=\mathcal Y_0+\sum_{i=7}^9\Y_i,\quad
			\mathcal J^{\rm bad}=\mathcal J_0^{\rm bad}+\sum_{j=1}^5\Pp_j,\quad
			\mathcal J^{\rm good}=\sigma\intR a_\xi(H_0+\eta_\phi)\dd\xi+\D\ge0.
		\end{equation}
		The quantities $\mathcal Y_0,\D,\mathcal J_0^{\rm bad}$, $\Pp_1,\ldots,\Pp_5$ and $\Y_7,\Y_8,\Y_9$  are defined
		in \eqref{eq:Y0}, \eqref{eq:Ddef}, \eqref{eq:Jbad}, 
		\eqref{eq:P12}--\eqref{eq:P5} and \eqref{eq:Y7}--\eqref{eq:Y89}, respectively.
		The nonnegativity of $\mathcal J^{\rm good}$ follows from \eqref{eq:coercive}.\end{lemma}
	
	\begin{proof}
		By \eqref{eq:entropy} and
		\eqref{eq:weightderivative}, a direct calculation gives
		\begin{align}	\label{eq:energyproductrule}
			\frac{\dd}{\dd t}\E
			=&-(\dot X+\sigma)\intR a_\xi H_0\dd\xi+\intR a(\partial_t-\sigma\partial_\xi)H_0\dd\xi +\frac{\dd}{\dd t}\intR a\eta_\phi\dd\xi.
		\end{align}
		To compute the second integral in \eqref{eq:energyproductrule},
		subtract \eqref{eq:shiftedNSP} from \eqref{eq:NSP}, using \eqref{eq:effective}--\eqref{eq:stress} for the momentum equation.
		This gives the perturbation system
		\begin{equation}\label{eq:perturbPDE}
			\left\{\begin{aligned}
				&(\dt)\tv=\tu_\xi+\dot X\bar v_\xi,\\
				&(\dt)\tu=-(\pp-\bar\pp)_\xi
				+\mu\left(\frac{u_\xi}{v}-\frac{\bar u_\xi}{\bar v}\right)_\xi
				+(F-\bar F)_\xi+\dot X\bar u_\xi,\\
				&\cv(\dt)\tz=-p\tu_\xi-(p-\bar p)\bar u_\xi
				+\kappa\left(\frac{\theta_\xi}{v}-\frac{\bar\theta_\xi}{\bar v}\right)_\xi+\mu\left(\frac{u_\xi^2}{v}-\frac{\bar u_\xi^2}{\bar v}\right)
				+\cv\dot X\bar\theta_\xi,\\
				&\la^2\left[\frac{\tp_\xi}{v}
				+\bar\phi_\xi\left(\frac1v-\frac1{\bar v}\right)\right]_\xi
				=e^{-\bar\phi}\left[-\tv+v(1-e^{-\tp})\right].
			\end{aligned}\right.
		\end{equation}
		To combine the pressure terms in \eqref{eq:perturbPDE}, we use the
		relative effective pressure
		\begin{align}
			\pp(v,\theta\mid\bar v,\bar\theta)
			&:=\pp-\bar\pp+\frac{\bar\theta+1}{\bar v^2}\tv-\frac\tz{\bar v}
			=\frac{(\bar\theta+1)\tv^2}{v\bar v^2}-\frac{\tv\tz}{v\bar v}.\label{eq:relativepressure}
		\end{align}
		Since each background solution in \eqref{eq:shiftedprofile} depends on
		$\xi-X(t)$,  \eqref{eq:shiftedNSP}$_{1,3}$ imply
		\begin{align}
			\cv(-\sigma\bar\theta_\xi)+\bar p\bar u_\xi
			&=\kappa\left(\frac{\bar\theta_\xi}{\bar v}\right)_\xi
			+\mu\frac{\bar u_\xi^2}{\bar v},\quad
			\bar u_\xi=-\sigma\bar v_\xi.\label{eq:bartemperature}
		\end{align}
		We first compute the volume term in \eqref{eq:energyproductrule}.
		Using $\Phi'(z)=1-1/z$,  \eqref{eq:NSP}$_1$ and
		\eqref{eq:perturbPDE}, and \eqref{eq:bartemperature}, we obtain
		\begin{align}
			&(\dt)\left[(\bar\theta+1)\Phi\left(\frac v{\bar v}\right)\right]\notag
			\\
			=&-(\sigma+\dot{X})\bar\theta_\xi\Phi\left(\frac v{\bar v}\right)
			+(\bar\theta+1)\left(\frac1{\bar v}-\frac1v\right)u_\xi+(\sigma+\dot{X})\frac{\bar\theta+1}{\bar v^2}\bar v_\xi\tv\notag
			\\
			=&(\bar\theta+1)\left(\frac1{\bar v}-\frac1v\right)\tu _\xi
			-\bar u_\xi\frac{(\bar\theta+1)\tv^2}{v\bar v^2}
			-\sigma\bar\theta_\xi\Phi\left(\frac v{\bar v}\right)+\dot{X}\left[
			\frac{\bar\theta+1}{\bar v^2}\bar v_\xi\tv
			-\bar\theta_\xi\Phi\left(\frac v{\bar v}\right)\right].
			\label{eq:volumecalc}
		\end{align}
		For the temperature term in \eqref{eq:energyproductrule}, substitute
		\eqref{eq:NSP}$_3$ and \eqref{eq:perturbPDE}, and
		use \eqref{eq:bartemperature} to collect the quadratic terms:
		\begin{align}
			(\dt)\left[\cv\bar\theta\Phi\left(\frac\theta{\bar\theta}\right)\right]\notag
			&=\frac\tz\theta\left[-pu_\xi+
			\kappa\left(\frac{\theta_\xi}{v}\right)_\xi+\mu\frac{u_\xi^2}{v}\right]
			-(\sigma+\dot{X})\cv\bar\theta_\xi
			\left[\Phi\left(\frac\theta{\bar\theta}\right)-\frac\tz{\bar\theta}\right]\notag
			\\
			&=-\frac\tz\theta p \tu _\xi
			+\bar u_\xi\frac{\tv\tz}{v\bar v}
			-\sigma\cv\bar\theta_\xi\Phi\left(\frac\theta{\bar\theta}\right)
			+\kappa\frac\tz\theta\left(\frac{\theta_\xi}{v}-\frac{\bar\theta_\xi}{\bar v}\right)_\xi+\mu\frac\tz\theta\left(\frac{u_\xi^2}{v}-\frac{\bar u_\xi^2}{\bar v}\right)\notag
			\\
			&\quad
			-\frac{\tz^2}{\theta\bar\theta}\left[
			\kappa\left(\frac{\bar\theta_\xi}{\bar v}\right)_\xi+\mu\frac{\bar u_\xi^2}{\bar v}\right]+\dot{X}\cv\bar\theta_\xi
			\left[\frac\tz{\bar\theta}-\Phi\left(\frac\theta{\bar\theta}\right)\right].
			\label{eq:temperaturecalc}
		\end{align}
		The two cancellations used in \eqref{eq:temperaturecalc} are
		\begin{align*}
			\bar u_\xi\tz\left(\frac{\bar p}{\bar\theta}-\frac p\theta\right)
			&=\bar u_\xi\frac{\tv\tz}{v\bar v},\\
			\left(\frac\tz\theta-\frac\tz{\bar\theta}\right)
			\left[\kappa\left(\frac{\bar\theta_\xi}{\bar v}\right)_\xi
			+\mu\frac{\bar u_\xi^2}{\bar v}\right]
			&=-\frac{\tz^2}{\theta\bar\theta}
			\left[\kappa\left(\frac{\bar\theta_\xi}{\bar v}\right)_\xi
			+\mu\frac{\bar u_\xi^2}{\bar v}\right].
		\end{align*}
		For the velocity term in \eqref{eq:energyproductrule}, multiply the second equation of \eqref{eq:perturbPDE} by $\tu $.
		\begin{align}
			(\dt)\frac{\tu ^2}{2}
			=&-\tu (\pp-\bar\pp)_\xi
			+\mu \tu \left(\frac{u_\xi}{v}-\frac{\bar u_\xi}{\bar v}\right)_\xi
			+\tu (F-\bar F)_\xi+\dot{X}\bar u_\xi \tu .
			\label{eq:kineticcalc}
		\end{align}
		To add \eqref{eq:volumecalc}, \eqref{eq:temperaturecalc} and
		\eqref{eq:kineticcalc}, we first combine their  coefficients of $\tilde u_\xi$ in \eqref{eq:volumecalc} and \eqref{eq:temperaturecalc} to obtain
		\begin{equation}\label{eq:pressurecancel}
			(\bar\theta+1)\left(\frac1{\bar v}-\frac1v\right)-\frac\tz\theta p
			=\frac{\bar\theta+1}{\bar v}-\frac{\theta+1}{v}=\bar\pp-\pp.
		\end{equation}
		For the viscous and heat fluxes in \eqref{eq:kineticcalc} and
		\eqref{eq:temperaturecalc}, integration by parts gives
		\begin{align}
			\intR a\mu \tu \left(\frac{u_\xi}{v}-\frac{\bar u_\xi}{\bar v}\right)_\xi\dd\xi
			=&-\intR a_\xi\mu \tu \left(\frac{u_\xi}{v}-\frac{\bar u_\xi}{\bar v}\right)\dd\xi
			-\intR\frac{a\mu}{v}\tu _\xi^2\dd\xi-\intR a\mu\bar u_\xi\left(\frac1v-\frac1{\bar v}\right)\tu _\xi\dd\xi,
			\label{eq:viscIBP}\\
			\intR a\kappa\frac\tz\theta
			\left(\frac{\theta_\xi}{v}-\frac{\bar\theta_\xi}{\bar v}\right)_\xi\dd\xi
			=&-\intR a_\xi\kappa\frac\tz\theta
			\left(\frac{\theta_\xi}{v}-\frac{\bar\theta_\xi}{\bar v}\right)\dd\xi
			-\intR\frac{a\kappa}{v\theta}\tz_\xi^2\dd\xi\notag
			\\
			&-\intR\frac{a\kappa\bar\theta_\xi}{\theta}
			\left(\frac1v-\frac1{\bar v}\right)\tz_\xi\dd\xi+\intR\frac{a\kappa\tz\theta_\xi}{\theta^2}
			\left(\frac{\theta_\xi}{v}-\frac{\bar\theta_\xi}{\bar v}\right)\dd\xi.
			\label{eq:heatIBP}
		\end{align}
		Substituting  \eqref{eq:volumecalc}, 
		\eqref{eq:temperaturecalc} and \eqref{eq:kineticcalc} into the $H_0$ contribution to
		\eqref{eq:energyproductrule}, integrating the transport term by parts,
		and using \eqref{eq:pressurecancel}--\eqref{eq:heatIBP}, we obtain
		\begin{equation}\label{eq:fluididentity}
			\frac{\dd}{\dd t}\intR aH_0\dd\xi
			=\dot{X}\mathcal Y_0+\mathcal J_0^{\rm bad}
			-\sigma\intR a_\xi H_0\dd\xi-\D
			-\intR(a_\xi \tu +a\tu _\xi)(F-\bar F)\dd\xi,
		\end{equation}
		where
		\begin{align}
			&\mathcal Y_0=-\intR a_\xi H_0\dd\xi
			+\intR a\left[\bar u_\xi \tu +\frac{\bar\theta+1}{\bar v^2}\bar v_\xi\tv
			+\frac{\bar\theta_\xi\tz}{(\gamma-1)\bar\theta}\right]\dd\xi-\intR a\bar\theta_\xi\left[\Phi\left(\frac v{\bar v}\right)
			+\cv\Phi\left(\frac\theta{\bar\theta}\right)\right]\dd\xi,
			\label{eq:Y0}\\
			&\D=\intR a\left[\frac\mu v \tu _\xi^2+\frac\kappa{v\theta}\tz_\xi^2\right]\dd\xi,
			\label{eq:Ddef}\\
			&\mathcal J_0^{\rm bad}=\intR a_\xi \tu (\pp-\bar\pp)\dd\xi-\intR a_\xi\left[\mu \tu \left(\frac{u_\xi}{v}-\frac{\bar u_\xi}{\bar v}\right)
			+\frac{\kappa\tz}{\theta}\left(\frac{\theta_\xi}{v}-\frac{\bar\theta_\xi}{\bar v}\right)\right]\dd\xi\notag
			\\
			&\qquad\qquad+\intR a\left\{-\sigma\bar\theta_\xi\left[\Phi\left(\frac v{\bar v}\right)
			+\cv\Phi\left(\frac\theta{\bar\theta}\right)\right]
			-\bar u_\xi\pp(v,\theta\mid\bar v,\bar\theta)\right\}\dd\xi\notag
			\\
			&\qquad\qquad+\intR a\left\{-\mu\bar u_\xi\left(\frac1v-\frac1{\bar v}\right)\tu _\xi
			-\frac{\kappa\bar\theta_\xi}{\theta}\left(\frac1v-\frac1{\bar v}\right)\tz_\xi\right.\notag
			\\
			&\qquad\qquad\qquad+\frac{\kappa\tz\theta_\xi}{\theta^2}
			\left(\frac{\theta_\xi}{v}-\frac{\bar\theta_\xi}{\bar v}\right)
			+\frac{\mu\tz}{\theta}\left(\frac{u_\xi^2}{v}-\frac{\bar u_\xi^2}{\bar v}\right)\left.-\frac{\tz^2}{\theta\bar\theta}
			\left[\kappa\left(\frac{\bar\theta_\xi}{\bar v}\right)_\xi
			+\mu\frac{\bar u_\xi^2}{\bar v}\right]\right\}\dd\xi.
			\label{eq:Jbad}
		\end{align}
		
		We next compute the $\frac{\dd}{\dd t}\intR a\eta_\phi\dd\xi$ term in \eqref{eq:energyproductrule}.
		We introduce $\chi$ to separate the explicit shift
		contribution, $R_F$ to extract the leading part of $F-\bar F$, and
		$\mathcal N_\phi$ to collect the terms involving $\chi_{\xi\xi}$ for integration by parts:
		\begin{align}
			&\chi:=(\dt)\phi+\sigma\bar\phi_\xi,
			\qquad(\dt)\tp=\chi+\dot{X}\bar\phi_\xi,\label{eq:chi}\\
			&F-\bar F=\frac{\la^2}{\bar v^2}\tp_{\xi\xi}+R_F, \qquad
			\mathcal N_\phi:=\tv+\frac{\la^2e^{\bar\phi}}{\bar v}\tp_{\xi\xi}-e^{\bar\phi}\tp,
			\label{eq:N}\\
			&R_F=\la^2\left[\left(\frac1{v^2}-\frac1{\bar v^2}\right)\tp_{\xi\xi}
			+\bar\phi_{\xi\xi}\left(\frac1{v^2}-\frac1{\bar v^2}\right)
			-\left(\frac{v_\xi\phi_\xi}{v^3}-\frac{\bar v_\xi\bar\phi_\xi}{\bar v^3}\right)
			+\frac12\left(\frac{\phi_\xi^2}{v^2}-\frac{\bar\phi_\xi^2}{\bar v^2}\right)\right].
			\label{eq:RFexpand}
		\end{align}
		To differentiate the three terms of $\eta_\phi$ in \eqref{eq:entropy},
		use \eqref{eq:perturbPDE}$_1$ and \eqref{eq:chi}.
		The product rule gives
		\begin{align}
			(\dt)\left(\frac{\la^2}{\bar v^2}\tv\tp_{\xi\xi}\right)
			=&\frac{\la^2}{\bar v^2}\tu _\xi\tp_{\xi\xi}
			+\frac{\la^2}{\bar v^2}\tv\chi_{\xi\xi}
			-\sigma\left(\frac{\la^2}{\bar v^2}\right)_\xi\tv\tp_{\xi\xi}\notag
			\\
			&+\dot{X}\left[\frac{\la^2}{\bar v^2}\bar v_\xi\tp_{\xi\xi}
			+\frac{\la^2}{\bar v^2}\tv\bar\phi_{\xi\xi\xi}
			-\left(\frac{\la^2}{\bar v^2}\right)_\xi\tv\tp_{\xi\xi}\right],
			\label{eq:mixcalc}\\
			(\dt)\left(\frac{\la^4e^{\bar\phi}}{2\bar v^3}\tp_{\xi\xi}^2\right)
			=&\frac{\la^4e^{\bar\phi}}{\bar v^3}\tp_{\xi\xi}\chi_{\xi\xi}
			-\frac\sigma2\left(\frac{\la^4e^{\bar\phi}}{\bar v^3}\right)_\xi\tp_{\xi\xi}^2\notag
			\\
			&+\dot{X}\left[\frac{\la^4e^{\bar\phi}}{\bar v^3}\tp_{\xi\xi}\bar\phi_{\xi\xi\xi}
			-\frac12\left(\frac{\la^4e^{\bar\phi}}{\bar v^3}\right)_\xi\tp_{\xi\xi}^2\right],
			\label{eq:secondcalc}\\
			(\dt)\left(\frac{\la^2e^{\bar\phi}}{2\bar v^2}\tp_\xi^2\right)
			=&\frac{\la^2e^{\bar\phi}}{\bar v^2}\tp_\xi\chi_\xi
			-\frac\sigma2\left(\frac{\la^2e^{\bar\phi}}{\bar v^2}\right)_\xi\tp_\xi^2\notag
			\\
			&+\dot{X}\left[\frac{\la^2e^{\bar\phi}}{\bar v^2}\tp_\xi\bar\phi_{\xi\xi}
			-\frac12\left(\frac{\la^2e^{\bar\phi}}{\bar v^2}\right)_\xi\tp_\xi^2\right].
			\label{eq:firstcalc}
		\end{align}
		Adding \eqref{eq:mixcalc}--\eqref{eq:firstcalc}, multiplying by $a$,
		and integrating, we use \eqref{eq:weightderivative} to obtain
		\begin{align}\label{eq:phienergy}
			\frac{\dd}{\dd t}\intR a\eta_\phi\dd\xi
			=\intR a(\dt)\eta_\phi\dd\xi
			-(\sigma+\dot{X})\intR a_\xi\eta_\phi\dd\xi=\sum_{j=1}^{5}S_j,
		\end{align}
		where
		\begin{align*}
			S_1&=\intR\frac{a\la^2}{\bar v^2}\tu_\xi\tp_{\xi\xi}\dd\xi,\qquad
			S_2=\intR a\left[\left(\frac{\la^2}{\bar v^2}\tv
			+\frac{\la^4e^{\bar\phi}}{\bar v^3}\tp_{\xi\xi}\right)\chi_{\xi\xi}
			+\frac{\la^2e^{\bar\phi}}{\bar v^2}\tp_\xi\chi_\xi\right]\dd\xi,\\
			S_3&=-\sigma\intR a\left[\left(\frac{\la^2}{\bar v^2}\right)_\xi\tv\tp_{\xi\xi}
			+\frac12\left(\frac{\la^4e^{\bar\phi}}{\bar v^3}\right)_\xi\tp_{\xi\xi}^2
			+\frac12\left(\frac{\la^2e^{\bar\phi}}{\bar v^2}\right)_\xi\tp_\xi^2\right]\dd\xi,\\
			S_4&=-\sigma\intR a_\xi\eta_\phi\dd\xi,\qquad
			S_5=\dot{X}(\Y_7+\Y_8+\Y_9).
		\end{align*}
		Adding
		\eqref{eq:phienergy} to \eqref{eq:fluididentity} gives \eqref{eq:identity} and \eqref{eq:Jfull}, where
		the quantities $\Pp_1,\ldots,\Pp_5$ and $\Y_7,\Y_8,\Y_9$ are defined by
		\begin{align}
			\Pp_1&=-\intR a_\xi \tu (F-\bar F)\dd\xi,\qquad
			\Pp_2=-\intR a\tu _\xi R_F\dd\xi,\label{eq:P12}\\
			\Pp_3&=\intR\frac{a\la^2}{\bar v^2}\mathcal N_\phi\chi_{\xi\xi}\dd\xi,\qquad
			\Pp_4=-\intR\left(\frac{a\la^2e^{\bar\phi}}{\bar v^2}\right)_\xi\tp\chi_\xi\dd\xi,\label{eq:P34}\\
			\Pp_5&=-\sigma\intR a\left[\left(\frac{\la^2}{\bar v^2}\right)_\xi\tv\tp_{\xi\xi}
			+\frac12\left(\frac{\la^4e^{\bar\phi}}{\bar v^3}\right)_\xi\tp_{\xi\xi}^2
			+\frac12\left(\frac{\la^2e^{\bar\phi}}{\bar v^2}\right)_\xi\tp_\xi^2\right]\dd\xi,
			\label{eq:P5}\\
			\Y_7&=-\intR a_\xi\eta_\phi\dd\xi
			-\intR a\left[\left(\frac{\la^2}{\bar v^2}\right)_\xi\tv\tp_{\xi\xi}
			+\frac12\left(\frac{\la^4e^{\bar\phi}}{\bar v^3}\right)_\xi\tp_{\xi\xi}^2
			+\frac12\left(\frac{\la^2e^{\bar\phi}}{\bar v^2}\right)_\xi\tp_\xi^2\right]\dd\xi,
			\label{eq:Y7}\\
			\Y_8&=\intR a\left[\left(\frac{\la^2}{\bar v^2}\tv
			+\frac{\la^4e^{\bar\phi}}{\bar v^3}\tp_{\xi\xi}\right)\bar\phi_{\xi\xi\xi}
			+\frac{\la^2e^{\bar\phi}}{\bar v^2}\tp_\xi\bar\phi_{\xi\xi}\right]\dd\xi,
			\quad
			\Y_9=\intR\frac{a\la^2}{\bar v^2}\bar v_\xi\tp_{\xi\xi}\dd\xi.
			\label{eq:Y89}
		\end{align}
		We briefly explain how $\Pp_1,\ldots,\Pp_5$ are obtained by this procedure.
		The part of \eqref{eq:fluididentity} containing $a_\xi\tu$ is $\Pp_1$.
		By \eqref{eq:N}, combining $S_1$ with the part containing $a\tu_\xi$
		in \eqref{eq:fluididentity} yields $\Pp_2$:
		\[
		-\intR a\tu_\xi(F-\bar F)\dd\xi+S_1
		=-\intR a\tu_\xi R_F\dd\xi=\Pp_2.
		\]
		Moreover, \eqref{eq:N} gives
		\[
		\frac{\la^2}{\bar v^2}\tv
		+\frac{\la^4e^{\bar\phi}}{\bar v^3}\tp_{\xi\xi}
		=\frac{\la^2e^{\bar\phi}}{\bar v^2}\tp+\frac{\la^2}{\bar v^2}\mathcal N_\phi.
		\]
		Hence integration by parts yields
		\[
		\begin{aligned}
			S_2
			&=\intR\frac{a\la^2}{\bar v^2}\mathcal N_\phi\chi_{\xi\xi}\dd\xi
			+\intR\frac{a\la^2e^{\bar\phi}}{\bar v^2}
			(\tp\chi_{\xi\xi}+\tp_\xi\chi_\xi)\dd\xi\\
			&=\intR\frac{a\la^2}{\bar v^2}\mathcal N_\phi\chi_{\xi\xi}\dd\xi
			-\intR\left(\frac{a\la^2e^{\bar\phi}}{\bar v^2}\right)_\xi
			\tp\chi_\xi\dd\xi
			=\Pp_3+\Pp_4.
		\end{aligned}
		\]
		And $S_3=\Pp_5$ with $S_4$ included in the good term and
		$S_5$ in the shift term, gives Lemma \ref{lem:identity}. 
	\end{proof}
	
	\subsection{Decompositions and the shift ODE}\label{sec:decomp}
	Lemma~\ref{lem:identity} supplies the weighted relative-entropy identity. In this subsection we split the flux into good terms, remainder terms and shift terms, assemble the resulting energy balance, and fix the shift ODE so that the translational kernel becomes dissipative. Throughout we work under the a priori assumptions \eqref{eq:smallness}.

	\subsubsection{\texorpdfstring{The good terms and the bad terms}{The good terms and the bad terms}}
	This subsection rewrites \eqref{eq:identity} as \eqref{eq:decomposition}
	to identify the terms that provide dissipation and those that must be
	controlled.  We rewrite \eqref{eq:Jfull} as
	\begin{align}
		\mathcal J^{\rm good}
		&=\sigma\intR a_\xi H_0 d\xi+\sigma\intR a_\xi\eta_\phi\dd\xi+\D
		:=\G+\G_\phi+\D.
		\label{eq:Jgoodassembled}
	\end{align}
	The main good terms in the basic energy estimate are
	$\delta_S|\dot X|^2$, $\Gs$, $\G_1$, $\G_2$, and $\D$;
	$\Gs,\G_1,\G_2$ are defined below and come from $\mathcal J^{\rm good}$,  and $\D$ is defined in \eqref{eq:Ddef}.
	The leading bad terms are $\B_1$ and $\B_2$ in
	\eqref{eq:B1} and \eqref{eq:B2};
	$\B_3,\ldots,\B_6$ in
	\eqref{eq:B3}--\eqref{eq:B6} are the remaining terms to be estimated. We define
	\begin{align}
		\Gs&=\intR\bar v_\xi(\tv^2+\tu ^2+\tz^2)\dd\xi,
		\label{eq:gooddef}, \qquad
		\G_1=\frac{\sigma^*(\theta_-+1)}{2(v_-)^2}
		\intR a_\xi\left(\tv+\frac \tu {\sigma^*}\right)^2\dd\xi,\\
		\G_2&=\frac{\sigma^*}{2(\gamma-1)\theta_-}
		\intR a_\xi\left(\tz-\frac{(\gamma-1)\theta_-}{v_-\sigma^*}\tu \right)^2\dd\xi,
		\label{eq:goodsquares} \\
		\B_1&=\intR a\bar v_\xi\left[
		\frac{\sigma^*((\gamma+1)\theta_-+2)}{2(v_-)^3}\tv^2
		+\frac{\sigma^*}{2v_-\theta_-}\tz^2
		-\frac{\sigma^*}{(v_-)^2}\tv\tz\right]\dd\xi\label{eq:B1}.
	\end{align}
	We now explain how \eqref{eq:B1} is obtained. By \eqref{eq:relativepressure}, the terms in
	\eqref{eq:Jbad} containing one profile derivative split as
	\begin{align}
		&-\sigma\bar\theta_\xi\left[\Phi\left(\frac v{\bar v}\right)
		+\cv\Phi\left(\frac\theta{\bar\theta}\right)\right]
		-\bar u_\xi\pp(v,\theta\mid\bar v,\bar\theta)\notag
		\\
		=&-\sigma\bar\theta_\xi\Phi\left(\frac v{\bar v}\right)
		-\bar u_\xi\frac{\bar\theta+1}{v\bar v^2}\tv^2
		-\frac{\sigma\bar\theta_\xi}{\gamma-1}\Phi\left(\frac\theta{\bar\theta}\right)
		+\frac{\bar u_\xi}{v\bar v}\tv\tz.
		\label{eq:profiletermsplit}
	\end{align}
	The translated versions of \eqref{eq:profilebounds} and
	\eqref{eq:profileslopes}, together with \eqref{eq:bartemperature}, give
	\begin{gather}
		\bar u_\xi=-\sigma\bar v_\xi,\qquad \bar v_\xi>0,\notag
		\\
		|\sigma-\sigma^*|+|\bar v-v_-|+|\bar\theta-\theta_-|
		+\left|\frac{\bar\theta_\xi}{\bar v_\xi}
		+\frac{(\gamma-1)\theta_-}{v_-}\right|\le C\delta_S.
		\label{eq:decompprofilecoefficients}
	\end{gather}
	Since $\Phi(1)=\Phi'(1)=0$, $\Phi''(1)=1$, Taylor expansion yields
	\begin{equation}\label{eq:decompPhiTaylor}
		\Phi\left(\frac v{\bar v}\right)
		=\frac{\tv^2}{2\bar v^2}+O(|\tv|^3),\qquad
		\Phi\left(\frac\theta{\bar\theta}\right)
		=\frac{\tz^2}{2\bar\theta^2}+O(|\tz|^3).
	\end{equation}
	For the first two terms on the right of \eqref{eq:profiletermsplit},
	use \eqref{eq:decompprofilecoefficients} and $v=\bar v+\tv$:
	\begin{align}
		-\frac{\sigma}{2\bar v^2}\frac{\bar\theta_\xi}{\bar v_\xi}
		&=\frac{\sigma^*(\gamma-1)\theta_-}{2(v_-)^3}+O(\delta_S),\notag
		\\
		\frac{\sigma(\bar\theta+1)}{v\bar v^2}
		&=\frac{\sigma(\bar\theta+1)}{\bar v^3}
		-\frac{\sigma(\bar\theta+1)\tv}{v\bar v^3}
		=\frac{\sigma^*(\theta_-+1)}{(v_-)^3}
		+O(\delta_S+|\tv|).
		\label{eq:decompvolumecoefficients}
	\end{align}
	Consequently, \eqref{eq:decompPhiTaylor} and
	\eqref{eq:decompvolumecoefficients} imply
	\begin{align}
		-\sigma\bar\theta_\xi\Phi\left(\frac v{\bar v}\right)
		-\bar u_\xi\frac{\bar\theta+1}{v\bar v^2}\tv^2
		=&\bar v_\xi\left[-\frac{\sigma}{2\bar v^2}
		\frac{\bar\theta_\xi}{\bar v_\xi}
		+\frac{\sigma(\bar\theta+1)}{v\bar v^2}\right]\tv^2
		-\sigma\bar\theta_\xi\left[\Phi\left(\frac v{\bar v}\right)
		-\frac{\tv^2}{2\bar v^2}\right]\notag
		\\
		=&\frac{\sigma^*}{2(v_-)^3}
		\big[(\gamma-1)\theta_-+2(\theta_-+1)\big]\bar v_\xi\tv^2
		+O\big((\delta_S+|\tv|)\bar v_\xi\tv^2\big)\notag
		\\
		=&\frac{\sigma^*((\gamma+1)\theta_-+2)}{2(v_-)^3}\bar v_\xi\tv^2
		+O\big((\delta_S+\varepsilon_1)\bar v_\xi\tv^2\big).
		\label{eq:decompvolumeterms}
	\end{align}
	Here the last line uses \eqref{eq:smallness}. Similarly,
	for the last two terms of \eqref{eq:profiletermsplit}, since
	\begin{align*}
		-\frac{\sigma}{2(\gamma-1)\bar\theta^2}
		\frac{\bar\theta_\xi}{\bar v_\xi}
		=\frac{\sigma^*}{2v_-\theta_-}+O(\delta_S), \qquad
		\frac{\sigma}{v\bar v}
		=\frac{\sigma}{\bar v^2}-\frac{\sigma\tv}{v\bar v^2}
		=\frac{\sigma^*}{(v_-)^2}+O(\delta_S+|\tv|),
	\end{align*}
	we have
	\begin{align}
		-\frac{\sigma\bar\theta_\xi}{\gamma-1}
		\Phi\left(\frac\theta{\bar\theta}\right)
		+\frac{\bar u_\xi}{v\bar v}\tv\tz
		=&\bar v_\xi\left[-\frac{\sigma}{2(\gamma-1)\bar\theta^2}
		\frac{\bar\theta_\xi}{\bar v_\xi}\tz^2
		-\frac{\sigma}{v\bar v}\tv\tz\right]
		-\frac{\sigma\bar\theta_\xi}{\gamma-1}
		\left[\Phi\left(\frac\theta{\bar\theta}\right)-\frac{\tz^2}{2\bar\theta^2}\right]\notag
		\\
		=&\bar v_\xi\left[\frac{\sigma^*}{2v_-\theta_-}\tz^2
		-\frac{\sigma^*}{(v_-)^2}\tv\tz\right]
		+O\big(\bar v_\xi[\delta_S\tz^2
		+(\delta_S+|\tv|)|\tv\tz|+|\tz|^3]\big)\notag
		\\
		=&\bar v_\xi\left[\frac{\sigma^*}{2v_-\theta_-}\tz^2
		-\frac{\sigma^*}{(v_-)^2}\tv\tz\right]
		+O\big((\delta_S+\varepsilon_1)\bar v_\xi(\tv^2+\tz^2)\big),
		\label{eq:decomptemperatureterms}
	\end{align}
	the final step uses
	$2|\tv\tz|\le\tv^2+\tz^2$ and
	$\norm{(\tv,\tz)}_\infty\le C\varepsilon_1$.
	Adding \eqref{eq:decompvolumeterms} and \eqref{eq:decomptemperatureterms}
	in \eqref{eq:profiletermsplit}, we obtain
	\begin{align}
		&-\sigma\bar\theta_\xi\left[\Phi\left(\frac v{\bar v}\right)
		+\cv\Phi\left(\frac\theta{\bar\theta}\right)\right]
		-\bar u_\xi\pp(v,\theta\mid\bar v,\bar\theta)\notag
		\\
		=&\bar v_\xi\left[
		\frac{\sigma^*((\gamma+1)\theta_-+2)}{2(v_-)^3}\tv^2
		+\frac{\sigma^*}{2v_-\theta_-}\tz^2
		-\frac{\sigma^*}{(v_-)^2}\tv\tz\right]
		+O\big((\delta_S+\varepsilon_1)\bar v_\xi(\tv^2+\tz^2)\big).
		\label{eq:decompquadraticsum}
	\end{align}
	Multiplication by $a$ and integration now give
	\begin{align}
		&\left|-\intR a\left\{\sigma\bar\theta_\xi
		\left[\Phi\left(\frac v{\bar v}\right)
		+\cv\Phi\left(\frac\theta{\bar\theta}\right)\right]
		+\bar u_\xi\pp(v,\theta\mid\bar v,\bar\theta)\right\}\dd\xi-\B_1\right|\notag
		\\
		&\qquad\le C(\delta_S+\varepsilon_1)
		\intR a\bar v_\xi(\tv^2+\tz^2)\dd\xi
		\le C(\delta_S+\varepsilon_1)\Gs.
		\label{eq:decompB5bound}
	\end{align}
	Using \eqref{eq:decompquadraticsum}--\eqref{eq:decompB5bound}
	and retaining the sign of the remainder,
	we decompose \eqref{eq:Jbad} as
	\begin{equation}\label{eq:J0assembled}
		\mathcal J_0^{\rm bad}=\B_1+\B_2+\B_3+\B_4+\B_5,
	\end{equation}
	where $\B_2,\ldots,\B_5$ are defined below. We also define $\B_6$
	in \eqref{eq:B6} to combine $\sum_{j=1}^5\Pp_j$ from
	\eqref{eq:Jfull} with $-\G_\phi$ from \eqref{eq:Jgoodassembled}:
	\begin{align}
		\B_2&=\intR a_\xi \tu (\pp-\bar\pp)\dd\xi,\label{eq:B2}\\
		\B_3&=-\intR a_\xi\left[\mu \tu \left(\frac{u_\xi}{v}-\frac{\bar u_\xi}{\bar v}\right)
		+\frac{\kappa\tz}{\theta}\left(\frac{\theta_\xi}{v}-\frac{\bar\theta_\xi}{\bar v}\right)\right]\dd\xi,
		\label{eq:B3}\\
		\B_4&=\intR a\left\{-\mu\bar u_\xi\left(\frac1v-\frac1{\bar v}\right)\tu _\xi
		-\frac{\kappa\bar\theta_\xi}{\theta}\left(\frac1v-\frac1{\bar v}\right)\tz_\xi\right.\notag
		\\
		&\qquad+\frac{\kappa\tz\theta_\xi}{\theta^2}
		\left(\frac{\theta_\xi}{v}-\frac{\bar\theta_\xi}{\bar v}\right)
		+\frac{\mu\tz}{\theta}\left(\frac{u_\xi^2}{v}-\frac{\bar u_\xi^2}{\bar v}\right)\left.-\frac{\tz^2}{\theta\bar\theta}
		\left[\kappa\left(\frac{\bar\theta_\xi}{\bar v}\right)_\xi
		+\mu\frac{\bar u_\xi^2}{\bar v}\right]\right\}\dd\xi,\label{eq:B4}\\
		\B_5&=-\intR a\left\{\sigma\bar\theta_\xi
		\left[\Phi\left(\frac v{\bar v}\right)+\cv\Phi\left(\frac\theta{\bar\theta}\right)\right]
		+\bar u_\xi\pp(v,\theta\mid\bar v,\bar\theta)\right\}\dd\xi-\B_1,
		\label{eq:B5}\\
		\B_6&=\sum_{j=1}^5\Pp_j-\G_\phi.
		\label{eq:B6}
	\end{align}
	By \eqref{eq:decompB5bound} and \eqref{eq:B5}, we have 
	\begin{align}\label{eq:B5final}
		|\B_5|\le C(\delta_S+\varepsilon_1)\Gs.
	\end{align}
	\subsubsection{\texorpdfstring{The shift terms}{The shift terms}}
	To choose the shift $X(t)$ in \eqref{eq:shiftODE} and estimate
	$\dot X\mathcal Y$ in \eqref{eq:identity}, we decompose
	$\mathcal Y_0$ from \eqref{eq:Y0} into the linear terms
	$\Y_1,\Y_2,\Y_3$ and the remaining terms $\Y_4,\Y_5,\Y_6$
	\begin{align}
    &\mathcal Y_0
        =\Y_1+\Y_2+\Y_3+\Y_4+\Y_5+\Y_6,\notag\\
		&\Y_1=\intR a\bar u_\xi \tu \dd\xi,\qquad
		\Y_2=\intR a\frac{\bar\theta+1}{\bar v^2}\bar v_\xi\tv\dd\xi,\qquad
		\Y_3=\frac1{\gamma-1}\intR a\frac{\bar\theta_\xi}{\bar\theta}\tz\dd\xi,
		\label{eq:Y123}\\
		&\Y_4=-\intR a\bar\theta_\xi\Phi\left(\frac v{\bar v}\right)\dd\xi,\quad
		\Y_5=-\frac1{\gamma-1}\intR a\bar\theta_\xi\Phi\left(\frac\theta{\bar\theta}\right)\dd\xi,\quad
		\Y_6=-\intR a_\xi H_0\dd\xi.
		\label{eq:Y456}
	\end{align}
    Together with $\Y_7,\Y_8,\Y_9$ in
	\eqref{eq:Y7}--\eqref{eq:Y89}, this gives   the decomposition
	\begin{equation}\label{eq:Yassembled}
		\mathcal Y=\mathcal Y_0+\Y_7+\Y_8+\Y_9
		=\sum_{i=1}^9\Y_i.
	\end{equation}
	\subsubsection{Assembly of the energy decomposition} Substituting \eqref{eq:J0assembled} into \eqref{eq:Jfull}
	and using \eqref{eq:B6}, we obtain
	\begin{align}
		\mathcal J^{\rm bad}
		&=\mathcal J_0^{\rm bad}+\sum_{j=1}^5\Pp_j=\sum_{i=1}^5\B_i+\B_6+\G_\phi.
		\label{eq:Jbadassembled}
	\end{align}
	Finally, substitute \eqref{eq:Yassembled},
	\eqref{eq:Jbadassembled}, and \eqref{eq:Jgoodassembled} into
	\eqref{eq:identity}. The two occurrences of $\G_\phi$ cancel:
	\begin{align}\label{eq:decomposition}
			\frac{\dd}{\dd t}\E
			&=\dot{X}\mathcal Y+\mathcal J^{\rm bad}-\mathcal J^{\rm good}\notag\\
			&=\dot{X}\left(\sum_{i=1}^6\Y_i+\sum_{i=7}^9\Y_i\right)
			+\left(\sum_{i=1}^5\B_i+\B_6+\G_\phi\right)
			-\left(\G+\G_\phi+\D\right)\notag\\
			&=\dot{X}\sum_{i=1}^9\Y_i+\sum_{i=1}^6\B_i-\G-\D.
	\end{align}
	Applying \eqref{eq:decompB5bound} to \eqref{eq:B5} also gives
	the inequality:
	\[
	\frac{\dd}{\dd t}\E
	\le\dot{X}\sum_{i=1}^9\Y_i
	+\sum_{i=1}^4\B_i+\B_6
	+C(\delta_S+\varepsilon_1)\Gs-\G-\D.
	\]
	\subsubsection{\texorpdfstring{Choice and elementary bounds for the shift}{Choice and elementary bounds for the shift}}\label{sec:shiftchoice}
	With $M$ fixed by \eqref{eq:M}, substituting \eqref{eq:Y123} into
	\eqref{eq:shiftODE} gives
	\[
	\dot X=-\frac M{\delta_S}(\Y_1+\Y_2+\Y_3).
	\]
	Thus \eqref{eq:profilederivatives}, \eqref{eq:weightderivative} and
	\eqref{eq:Yassembled} yield
	\begin{align}
		|\dot{X}|&\le\frac C{\delta_S}\intR\bar v_\xi
		|(\tv,\tu ,\tz)|\dd\xi\le C\norm{(\tv,\tu ,\tz)}_\infty\le C\varepsilon_1,
		\label{eq:shiftbound}\\
		\dot{X}\mathcal Y
		&=-\frac{\delta_S}{M}|\dot{X}|^2
		+\dot{X}\sum_{i=4}^9\Y_i.\label{eq:shiftidentity}
	\end{align}
	For a fixed strong solution, the smoothness and decay of the profile
	make the right side of \eqref{eq:shiftODE} locally Lipschitz in $X$ and bounded.
	Thus the shift exists uniquely on $[0,T]$, and \eqref{eq:shiftbound} gives
	\begin{equation*}
		|X(t)|\le \int_0^t|\dot{X}(s)|\dd s
		\le C\varepsilon_1t,\qquad 0\le t\le T.
	\end{equation*}
	
	\subsection{Leading order estimates}\label{sec:leading}
	The shift ODE of Section~\ref{sec:shiftchoice} converts the translational kernel into the dissipation $-\frac{\delta_S}{M}|\dot X|^2$. What remains at leading order are the hyperbolic flux $\B_1+\B_2$, the relative-entropy production $\G$, and the viscous-thermal diffusion $\D$ associated with $P_{\mathrm{eff}}$. In this subsection we estimate these terms and obtain Lemma~\ref{lem:leading}.
	First fix $t\in[0,T]$ and write the main terms in the variables $y$ and $\tu $:
	\begin{equation}\label{11111}
		y=\frac{v^S(\xi-X(t))-v_-}{\delta_S}
		=\frac{\bar v-v_-}{\delta_S}.
	\end{equation}
	By \eqref{eq:shock} and \eqref{eq:shiftbound},
	\begin{gather*}
		|X(t)|\le C\varepsilon_1T,\qquad
		\lim_{\xi\to-\infty}y=0,\qquad \lim_{\xi\to+\infty}y=1,\\
		\frac{\dd y}{\dd\xi}
		=\frac{(v^S)'(\xi-X(t))}{\delta_S}
		=\frac{\bar v_\xi}{\delta_S}>0.
	\end{gather*}
	Thus $\xi\in\R\mapsto y\in(0,1)$ is a well-defined change of
	variable. We keep the notation $\tv,\tu ,\tz$ for the functions expressed in $y$.
	\eqref{eq:weight}--\eqref{eq:weightderivative} give
	\begin{gather}
		a=1+\sqrt{\delta_S} y,\qquad a_\xi=\sqrt{\delta_S}\frac{\dd y}{\dd\xi}>0,\qquad
		a_\xi\dd\xi=\sqrt{\delta_S}\dd y,\qquad
		\bar v_\xi\dd\xi=\delta_S\dd y,
		\label{eq:B2ymeasures}\\
		\intR\bar v_\xi\dd\xi=\delta_S,\qquad
		\bar w=\int_0^1\tu \dd y,\qquad
		\Gs=\delta_S\int_0^1(\tv^2+\tu ^2+\tz^2)\dd y.
		\label{eq:B2yGs}
	\end{gather}
	\begin{lemma}\label{lem:leading}
		Under the hypotheses of Proposition~\ref{thm:closure}, there exists constant $c>0$ such that
		\begin{equation}\label{eq:leading}
			-\frac{\delta_S}{2M}|\dot{X}|^2
			+\B_1+\B_2-\G-\frac34\D
			\le-c\Gs-c(\G_1+\G_2).
		\end{equation}
	\end{lemma}\begin{proof}\leavevmode\par
		\subsubsection{\texorpdfstring{Estimates on $\B_2-\G$}{Estimates on B2-G}}
		First, \eqref{eq:Jgoodassembled}, \eqref{eq:B2},  and
		\eqref{eq:B2ymeasures} yield
		\begin{equation}\label{eq:B2integrand}
			\B_2-\G
			=\intR a_\xi\big[\tu (\pp-\bar\pp)-\sigma H_0\big]\dd\xi
			=\sqrt{\delta_S}\int_0^1\big[\tu (\pp-\bar\pp)-\sigma H_0\big]\dd y.
		\end{equation}
		By \eqref{eq:H0}, the integrand in \eqref{eq:B2integrand} is
		\begin{equation}\label{eq:B2density}
			\tu (\pp-\bar\pp)-\sigma H_0
			=\tu (\pp-\bar\pp)
			-\sigma\left[(\bar\theta+1)\Phi\left(\frac v{\bar v}\right)
			+\frac{\bar\theta}{\gamma-1}\Phi\left(\frac\theta{\bar\theta}\right)
			+\frac{\tu ^2}{2}\right].
		\end{equation}
		We use the exact pressure identity
		\begin{equation}\label{eq:B2pressureidentity}
			\pp-\bar\pp
			=\frac{\theta+1}{v}-\frac{\bar\theta+1}{\bar v}
			=\frac{\theta-\bar\theta}{v}
			-\frac{\bar\theta+1}{v\bar v}(v-\bar v)
			=\frac\tz v-\frac{\bar\pp}{v}\tv.
		\end{equation}
        For the entropy terms, we use the following identity and estimates
		\begin{align*}
			\Big|\Phi(1+s)-\frac{s^2}{2}\Big|=\Big|\int_0^s\frac r{1+r}\dd r-\frac{s^2}{2}\Big|
			=\Big|s^3\int_0^1\frac{r^2}{1+rs}\dd r\Big|\le\frac23|s|^3\qquad\left(|s|\le\frac12\right),
            \end{align*}
            thus, we obtain
            \begin{align}
			\left|\Phi\left(\frac v{\bar v}\right)-\frac{\tv^2}{2\bar v^2}\right|
			\le C|\tv|^3,\qquad
			\left|\Phi\left(\frac\theta{\bar\theta}\right)-\frac{\tz^2}{2\bar\theta^2}\right|
			\le C|\tz|^3.
			\label{eq:B2Phierrors}
		\end{align}
		The profile bounds \eqref{eq:profilebounds} and the smallness
		assumption \eqref{eq:smallness} imply
		\begin{gather}
			|\bar v-v_-|+|\bar\theta-\theta_-|+|\sigma-\sigma^*|\le C\delta_S,
			\qquad \norm{(\tv,\tu ,\tz)}_\infty\le C\varepsilon_1,\notag
			\\
			\frac{v_-}{2}\le v,\bar v\le2v_-,\qquad
			\frac{\theta_-}{2}\le\theta,\bar\theta\le2\theta_-.
			\label{eq:B2coefficientrange}
		\end{gather}
		In particular, the errors in replacing the coefficients by
		their left-state values satisfy
		\begin{align}
			&\frac1v-\frac1{v_-}
			=-\frac{\tv+\bar v-v_-}{vv_-},\qquad
			\frac{\bar\pp}{v}-\frac{\theta_-+1}{(v_-)^2}
			=\frac{\bar\theta-\theta_-}{v\bar v}
			-\frac{(\theta_-+1)\big[\bar v\tv+(\bar v+v_-)(\bar v-v_-)\big]}
			{v\bar v(v_-)^2},\notag
			\\
			&\left|\frac1v-\frac1{v_-}\right|
			+\left|\frac{\bar\pp}{v}-\frac{\theta_-+1}{(v_-)^2}\right|
			\le C(\delta_S+|\tv|),\notag
			\\
			&\left|\frac{\sigma(\bar\theta+1)}{2\bar v^2}
			-\frac{\sigma^*(\theta_-+1)}{2(v_-)^2}\right|
			+\left|\frac{\sigma}{2(\gamma-1)\bar\theta}
			-\frac{\sigma^*}{2(\gamma-1)\theta_-}\right|
			\le C\delta_S.
			\label{eq:B2coefficienterrors}
		\end{align}
		Substituting \eqref{eq:B2pressureidentity} and  \eqref{eq:B2Phierrors} into
		\eqref{eq:B2density}, and then using
		\eqref{eq:B2coefficienterrors}, we obtain
		\begin{align}
			\tu (\pp-\bar\pp)-\sigma H_0
			=&\tu \left(\frac\tz v-\frac{\bar\pp}{v}\tv\right)
			-\sigma\left[(\bar\theta+1)\Phi\left(\frac v{\bar v}\right)
			+\frac{\bar\theta}{\gamma-1}\Phi\left(\frac\theta{\bar\theta}\right)
			+\frac{\tu ^2}{2}\right]\notag
			\\
			\le&\tu \left[\frac\tz{v_-}-\frac{\theta_-+1}{(v_-)^2}\tv\right]
			-\frac{\sigma^*(\theta_-+1)}{2(v_-)^2}\tv^2
			-\frac{\sigma^*}{2(\gamma-1)\theta_-}\tz^2
			-\frac{\sigma^*}{2}\tu ^2\notag
			\\
			&+C\delta_S(\tv^2+\tu ^2+\tz^2)
			+C\big(|\tv \tu \tz|+|\tu |\tv^2+|\tv|^3+|\tz|^3\big)\notag
			\\
			\le&\tu \left[\frac\tz{v_-}-\frac{\theta_-+1}{(v_-)^2}\tv\right]
			-\frac{\sigma^*(\theta_-+1)}{2(v_-)^2}\tv^2
			-\frac{\sigma^*}{2(\gamma-1)\theta_-}\tz^2
			-\frac{\sigma^*}{2}\tu ^2\notag
			\\
			&+C\delta_S|(\tv,\tu ,\tz)|^2+C|(\tv,\tu ,\tz)|^3.
			\label{eq:B2frozen}
		\end{align}
		To complete the squares in \eqref{eq:B2frozen}, we use
		\eqref{eq:constants} to obtain $\frac{\theta_-+1}{2(v_-)^2\sigma^*}
		+\frac{(\gamma-1)\theta_-}{2(v_-)^2\sigma^*}
		=\frac{\gamma\theta_-+1}{2(v_-)^2\sigma^*}
		=\frac{\sigma^*}{2}$ and $(\sigma^*)^2=\frac{\gamma\theta_-+1}{(v_-)^2}$. Thus, we have
		\begin{align*}
			&\tu \left[\frac\tz{v_-}-\frac{\theta_-+1}{(v_-)^2}\tv\right]
			-\sigma^*\left[\frac{\theta_-+1}{2(v_-)^2}\tv^2
			+\frac{\tz^2}{2(\gamma-1)\theta_-}+\frac{\tu ^2}{2}\right]\\
			=&-\frac{\sigma^*(\theta_-+1)}{2(v_-)^2}\tv^2
			-\frac{\theta_-+1}{(v_-)^2}\tv \tu 
			-\frac{\sigma^*}{2(\gamma-1)\theta_-}\tz^2
			+\frac{\tz \tu }{v_-}-\left[\frac{\theta_-+1}{2(v_-)^2\sigma^*}
			+\frac{(\gamma-1)\theta_-}{2(v_-)^2\sigma^*}\right]\tu ^2\\
			=&-\frac{\sigma^*(\theta_-+1)}{2(v_-)^2}
			\left(\tv+\frac \tu {\sigma^*}\right)^2
			-\frac{\sigma^*}{2(\gamma-1)\theta_-}
			\left(\tz-\frac{(\gamma-1)\theta_-}{v_-\sigma^*}\tu \right)^2.
		\end{align*}
		Consequently, \eqref{eq:B2frozen} becomes
		\begin{align}
			\tu (\pp-\bar\pp)-\sigma H_0
			\le&-\frac{\sigma^*(\theta_-+1)}{2(v_-)^2}
			\left(\tv+\frac \tu {\sigma^*}\right)^2-\frac{\sigma^*}{2(\gamma-1)\theta_-}
			\left(\tz-\frac{(\gamma-1)\theta_-}{v_-\sigma^*}\tu \right)^2
			\notag
			\\
			&+C\delta_S|(\tv,\tu ,\tz)|^2+C|(\tv,\tu ,\tz)|^3.
			\label{eq:B2pointwise}
		\end{align}
		Multiplying \eqref{eq:B2pointwise} by $a_\xi>0$ and integrating,
		we obtain the estimate:
		\begin{equation}\label{eq:B2withBnew}
			\B_2-\G\le-\G_1-\G_2+\B_{\rm new}.
		\end{equation}
		Here $\B_{\rm new}$ is the nonnegative bound obtained by multiplying
		the last two terms of \eqref{eq:B2pointwise} by $a_\xi$ and integrating,
		with the constant $C$ from \eqref{eq:B2pointwise}:
		\begin{equation}\label{eq:Bnewdefinition}
			\B_{\rm new}
			:=C\delta_S\intR a_\xi|(\tv,\tu ,\tz)|^2\dd\xi
			+C\intR a_\xi|(\tv,\tu ,\tz)|^3\dd\xi.
		\end{equation}
		Recall the definition of  $\G_1$ and $\G_2$  \eqref{eq:gooddef} and \eqref{eq:goodsquares}; by
		\eqref{eq:B2ymeasures}, their expressions in $y$ are
		\begin{align*}
			\G_1&=\frac{\sqrt{\delta_S}\sigma^*(\theta_-+1)}{2(v_-)^2}
			\int_0^1\left(\tv+\frac \tu {\sigma^*}\right)^2\dd y, \qquad
			\G_2=\frac{\sqrt{\delta_S}\sigma^*}{2(\gamma-1)\theta_-}
			\int_0^1\left(\tz-\frac{(\gamma-1)\theta_-}{v_-\sigma^*}\tu \right)^2\dd y.
		\end{align*}
		We now estimate the two integrals in \eqref{eq:Bnewdefinition}.
		First, \eqref{eq:B2ymeasures}--\eqref{eq:B2yGs} give
		\begin{equation}\label{eq:Bnewquadratic}
			\delta_S\intR a_\xi|(\tv,\tu ,\tz)|^2\dd\xi
			=\delta_S^{3/2}\int_0^1|(\tv,\tu ,\tz)|^2\dd y
			=\sqrt{\delta_S}\Gs.
		\end{equation}
		For the cubic integral, use
		\begin{gather}
			\tv=\left(\tv+\frac \tu {\sigma^*}\right)-\frac \tu {\sigma^*},\qquad
			\tz=\left(\tz-\frac{(\gamma-1)\theta_-}{v_-\sigma^*}\tu \right)
			+\frac{(\gamma-1)\theta_-}{v_-\sigma^*}\tu ,\notag
			\\
			\left\|\tv+\frac \tu {\sigma^*}\right\|_\infty
			+\left\|\tz-\frac{(\gamma-1)\theta_-}{v_-\sigma^*}\tu \right\|_\infty
			\le C\norm{(\tv,\tu ,\tz)}_{H^1(\R)}\le C\varepsilon_1.
			\label{eq:Bnewcubicsplit}
		\end{gather}
		It follows from \eqref{eq:Bnewcubicsplit} and
		\eqref{eq:gooddef}--\eqref{eq:goodsquares} that
		\begin{align}
			\intR a_\xi|(\tv,\tu ,\tz)|^3\dd\xi
			\le&C\intR a_\xi\left|\tv+\frac \tu {\sigma^*}\right|^3\dd\xi
			+C\intR a_\xi|\tu |^3\dd\xi+C\intR a_\xi
			\left|\tz-\frac{(\gamma-1)\theta_-}{v_-\sigma^*}\tu \right|^3\dd\xi\notag
			\\
			\le&C\varepsilon_1\intR a_\xi
			\left[\left(\tv+\frac \tu {\sigma^*}\right)^2
			+\left(\tz-\frac{(\gamma-1)\theta_-}{v_-\sigma^*}\tu \right)^2\right]\dd\xi
			+C\intR a_\xi|\tu |^3\dd\xi\notag
			\\
			\le&C\varepsilon_1(\G_1+\G_2)+C\intR a_\xi|\tu |^3\dd\xi.
			\label{eq:Bnewcubicgood}
		\end{align}
		To estimate the last integral in \eqref{eq:Bnewcubicgood},
		we use interpolation in the original spatial variable $\xi$.
		By \eqref{eq:smallness}, \eqref{eq:B2coefficientrange}, and \eqref{eq:Ddef},
		\begin{gather}
			\norm{\tu (\xi)}_{L^\infty}^2
			\le2\norm \tu _{L^2(\R)}\norm{\tu _\xi}_{L^2(\R)}\lesssim \varepsilon_1 \sqrt{\D}.
			\label{eq:B2interpolation}
		\end{gather}
		Consequently, \eqref{eq:B2ymeasures}--\eqref{eq:B2yGs},
		and \eqref{eq:B2interpolation} give
		\begin{align}
			\intR a_\xi|\tu |^3\dd\xi
			&=\sqrt{\delta_S}\int_0^1|\tu |^3\dd y
			=\frac{1}{\sqrt{\delta_S}}\intR\bar v_\xi|\tu |^3\dd\xi\notag
			\\
			&\le\frac{1}{\sqrt{\delta_S}}\norm \tu _{L^\infty(\R)}^2
			\intR\bar v_\xi|\tu |\dd\xi\le C\varepsilon_1\sqrt\D\sqrt{\Gs}.
			\label{eq:Bnewvelocitycubic}
		\end{align}
		Combining \eqref{eq:Bnewquadratic}, \eqref{eq:Bnewcubicgood},
		and \eqref{eq:Bnewvelocitycubic}, we conclude that
		\begin{align}
			\B_{\rm new}
			&\le C\sqrt{\delta_S}\Gs+C\varepsilon_1(\G_1+\G_2)
			+C\varepsilon_1\sqrt\D\sqrt{\Gs}\notag
			\\
			&\le C\varepsilon_1(\G_1+\G_2)
			+C(\sqrt{\delta_S}+\varepsilon_1)\Gs+C\varepsilon_1\D.
			\label{eq:Bnewbound}
		\end{align}
		Finally, substituting \eqref{eq:Bnewbound} into
		\eqref{eq:B2withBnew} yields
		\begin{equation}\label{eq:B2bound}
			{\quad\B_2-\G
				\le-(1-C\varepsilon_1)(\G_1+\G_2)
				+C(\sqrt{\delta_S}+\varepsilon_1)\Gs+C\varepsilon_1\D.\quad}
		\end{equation}
		\subsubsection{\texorpdfstring{Estimates on $-\frac{\delta_S}{2M}|\dot{X}|^2$}{Estimates on the shift square}}
		We first estimate $\Y_1,\Y_2,\Y_3$ in \eqref{eq:shiftODE}.
		By \eqref{eq:Y123}, $\bar u_\xi=-\sigma\bar v_\xi$, and
		\eqref{eq:B2ymeasures},
		\begin{align*}
			\Y_1
			=-\sigma\intR a\bar v_\xi \tu \dd\xi
			=-\delta_S\sigma\int_0^1a\tu \dd y.
		\end{align*}
		Using \eqref{eq:profilebounds}, $|a-1|\le\sqrt{\delta_S}$, and
		$\delta_S\le\sqrt{\delta_S}$ from \eqref{eq:smallness}, we have
		\begin{align}
			\left|\Y_1+\delta_S\sigma^*\bar w\right|
			&=\delta_S\left|\int_0^1(\sigma a-\sigma^*)\tu \dd y\right|\le\delta_S\int_0^1
			\big(a|\sigma-\sigma^*|+\sigma^*|a-1|\big)|\tu |\dd y\notag
			\\
			&\le C\delta_S(\delta_S+\sqrt{\delta_S})\int_0^1|\tu |\dd y
			\le C\delta_S^{3/2}\int_0^1|\tu |\dd y.
			\label{eq:shiftY1error}
		\end{align}
		For $\Y_2$, first use \eqref{eq:Y123} and
		\eqref{eq:B2ymeasures}, and then add and subtract $\tu /\sigma^*$:
		\begin{align}
			\Y_2
			=\delta_S\int_0^1a\frac{\bar\theta+1}{\bar v^2}\tv\dd y
			=-\frac{\delta_S}{\sigma^*}\int_0^1a\frac{\bar\theta+1}{\bar v^2}\tu \dd y
			+\delta_S\int_0^1a\frac{\bar\theta+1}{\bar v^2}
			\left(\tv+\frac \tu {\sigma^*}\right)\dd y.
			\label{eq:shiftY2change}
		\end{align}
		The coefficient error is estimated by \eqref{eq:profilebounds}:
		\begin{align}
			a\frac{\bar\theta+1}{\bar v^2}-\frac{\theta_-+1}{(v_-)^2}
			&=(a-1)\frac{\bar\theta+1}{\bar v^2}
			+\frac{\bar\theta-\theta_-}{\bar v^2}
			-\frac{(\theta_-+1)(\bar v+v_-)(\bar v-v_-)}{\bar v^2(v_-)^2},\notag
			\\
			\left|a\frac{\bar\theta+1}{\bar v^2}-\frac{\theta_-+1}{(v_-)^2}\right|
			&\le C\big(|a-1|+|\bar\theta-\theta_-|+|\bar v-v_-|\big)
			\le C(\sqrt{\delta_S}+\delta_S).
			\label{eq:shiftY2coefficient}
		\end{align}
		Consequently, \eqref{eq:shiftY2change}--\eqref{eq:shiftY2coefficient}
		and $\dd y=a_\xi\dd\xi/\sqrt{\delta_S}$ give
		\begin{align}
			\Y_2+\frac{(\theta_-+1)\delta_S}{(v_-)^2\sigma^*}\bar w
			=&-\frac{\delta_S}{\sigma^*}\int_0^1
			\left[a\frac{\bar\theta+1}{\bar v^2}-\frac{\theta_-+1}{(v_-)^2}\right]\tu \dd y+\delta_S\int_0^1a\frac{\bar\theta+1}{\bar v^2}
			\left(\tv+\frac \tu {\sigma^*}\right)\dd y,\notag
			\\
			\left|\Y_2+\frac{(\theta_-+1)\delta_S}{(v_-)^2\sigma^*}\bar w\right|
			\le&C\delta_S(\sqrt{\delta_S}+\delta_S)\int_0^1|\tu |\dd y
			+C\delta_S\int_0^1\left|\tv+\frac \tu {\sigma^*}\right|\dd y.
			\label{eq:shiftY2error}
		\end{align}
		Likewise, for $\Y_3$ we obtain from \eqref{eq:Y123} and
		\eqref{eq:B2ymeasures} that
		\begin{align}
			\Y_3
			=\frac{\delta_S}{\gamma-1}\int_0^1
			a\frac{\bar\theta_\xi}{\bar v_\xi\bar\theta}\tz\dd y
			=\frac{\theta_-\delta_S}{v_-\sigma^*}\int_0^1
			a\frac{\bar\theta_\xi}{\bar v_\xi\bar\theta}\tu \dd y+\frac{\delta_S}{\gamma-1}\int_0^1
			a\frac{\bar\theta_\xi}{\bar v_\xi\bar\theta}
			\left(\tz-\frac{(\gamma-1)\theta_-}{v_-\sigma^*}\tu \right)\dd y.
			\label{eq:shiftY3change}
		\end{align}
		Here \eqref{eq:profileslopes} and \eqref{eq:profilebounds} imply
		\begin{align}
			&a\frac{\bar\theta_\xi}{\bar v_\xi\bar\theta}+\frac{\gamma-1}{v_-}
			=\frac a{\bar\theta}
			\left[\frac{\bar\theta_\xi}{\bar v_\xi}
			+\frac{(\gamma-1)\theta_-}{v_-}\right]
			+\frac{\gamma-1}{v_-\bar\theta}
			\big[(\bar\theta-\theta_-)-(a-1)\theta_-\big],\notag\\
			&\left|a\tfrac{\bar\theta_\xi}{\bar v_\xi\bar\theta}+\tfrac{\gamma-1}{v_-}\right|
			\le C\left|\tfrac{\bar\theta_\xi}{\bar v_\xi}
			+\tfrac{(\gamma-1)\theta_-}{v_-}\right|
			+C\big(|\bar\theta-\theta_-|+|a-1|\big)
			\le C(\delta_S+\sqrt{\delta_S}).
			\label{eq:shiftY3coefficient}
		\end{align}
		Using \eqref{eq:shiftY3coefficient} in \eqref{eq:shiftY3change},
		and changing the remainder integral back by \eqref{eq:B2ymeasures}, we have
		\begin{align}
			&\Y_3+\frac{(\gamma-1)\theta_-\delta_S}{(v_-)^2\sigma^*}\bar w
			=\frac{\theta_-\delta_S}{v_-\sigma^*}\int_0^1
			\left[a\frac{\bar\theta_\xi}{\bar v_\xi\bar\theta}
			+\frac{\gamma-1}{v_-}\right]\tu \dd y+\frac{\delta_S}{\gamma-1}\int_0^1
			a\frac{\bar\theta_\xi}{\bar v_\xi\bar\theta}
			\left(\tz-\frac{(\gamma-1)\theta_-}{v_-\sigma^*}\tu \right)\dd y,\notag
			\\
			&\left|\Y_3+\frac{(\gamma-1)\theta_-\delta_S}{(v_-)^2\sigma^*}\bar w\right|\notag
			\\
			&\le C\delta_S(\sqrt{\delta_S}+\delta_S)\int_0^1|\tu |\dd y
			+C\sqrt{\delta_S}\intR a_\xi
			\left|\tz-\frac{(\gamma-1)\theta_-}{v_-\sigma^*}\tu \right|\dd\xi.
			\label{eq:shiftY3error}
		\end{align}
		We now add the three estimates. By \eqref{eq:constants},
		the coefficients of $\delta_S\bar w$ satisfy
		\begin{equation*}
			\sigma^*+\frac{\theta_-+1}{(v_-)^2\sigma^*}
			+\frac{(\gamma-1)\theta_-}{(v_-)^2\sigma^*}
			=\sigma^*+\frac{\gamma\theta_-+1}{(v_-)^2\sigma^*}
			=2\sigma^*.
		\end{equation*}
		Therefore, the ODE \eqref{eq:shiftODE} gives the identity
		\begin{align}
			&\dot{X}-2M\sigma^*\bar w
			=-\frac M{\delta_S}
			\big[\Y_1+\Y_2+\Y_3+2\sigma^*\delta_S\bar w\big]\notag\\
			&=-\frac M{\delta_S}\left[
			\big(\Y_1+\sigma^*\delta_S\bar w\big)
			+\left(\Y_2+\frac{(\theta_-+1)\delta_S}{(v_-)^2\sigma^*}\bar w\right)
			+\left(\Y_3+\frac{(\gamma-1)\theta_-\delta_S}{(v_-)^2\sigma^*}\bar w\right)
			\right].
			\label{eq:shiftODEdeviation}
		\end{align}
		Substituting \eqref{eq:shiftY1error}, \eqref{eq:shiftY2error},
		and \eqref{eq:shiftY3error} into \eqref{eq:shiftODEdeviation}, we obtain
		\begin{align}
			\left|\dot{X}-2M\sigma^*\bar w\right|
			\le&C(\sqrt{\delta_S}+\delta_0)\int_0^1|\tu |\dd y
			+\frac C{\sqrt{\delta_S}}\intR a_\xi
			\left|\tv+\frac \tu {\sigma^*}\right|\dd\xi\notag\\
			&+\frac C{\sqrt{\delta_S}}\intR a_\xi
			\left|\tz-\frac{(\gamma-1)\theta_-}{v_-\sigma^*}\tu \right|\dd\xi.
			\label{eq:shiftODEerror}
		\end{align}
		Next, Cauchy--Schwarz, \eqref{eq:B2ymeasures}, and
		\eqref{eq:gooddef}--\eqref{eq:goodsquares}  yield
		\begin{gather}
			\begin{aligned}
				\left(\intR a_\xi\left|\tv+\frac \tu {\sigma^*}\right|\dd\xi\right)^2
				&\le\left(\intR a_\xi\dd\xi\right)
				\intR a_\xi\left(\tv+\frac \tu {\sigma^*}\right)^2\dd\xi
				\\&=\frac{2(v_-)^2\sqrt{\delta_S}}{\sigma^*(\theta_-+1)}\G_1,\\
				\left(\intR a_\xi
				\left|\tz-\frac{(\gamma-1)\theta_-}{v_-\sigma^*}\tu \right|\dd\xi\right)^2
				&\le\left(\intR a_\xi\dd\xi\right)
				\intR a_\xi
				\left(\tz-\frac{(\gamma-1)\theta_-}{v_-\sigma^*}\tu \right)^2\dd\xi\\&=\frac{2(\gamma-1)\theta_-\sqrt{\delta_S}}{\sigma^*}\G_2.
			\end{aligned}
			\label{eq:shiftgoodCS}
		\end{gather}
		The reverse triangle inequality, followed by the square of
		\eqref{eq:shiftODEerror} and \eqref{eq:shiftgoodCS}, gives
		\begin{align}
			\left(\left|2M\sigma^*\bar w\right|-|\dot{X}|\right)^2
			&\le C(\sqrt{\delta_S}+\delta_0)^2\int_0^1\tu ^2\dd y
			+\frac C{\sqrt{\delta_S}}(\G_1+\G_2).
			\label{eq:shiftreversetriangle}
		\end{align}
		Finally, combining the elementary identity
		\begin{align*}
			\frac12\left|2M\sigma^*\bar w\right|^2-|\dot{X}|^2
			&=\left(\left|2M\sigma^*\bar w\right|-|\dot{X}|\right)^2
			-\frac12\left(\left|2M\sigma^*\bar w\right|-2|\dot{X}|\right)^2\le\left(\left|2M\sigma^*\bar w\right|-|\dot{X}|\right)^2,
		\end{align*}
		with \eqref{eq:shiftreversetriangle}  gives
		\begin{equation}\label{eq:shiftlowerbound}
			2M^2(\sigma^*)^2\bar w^2-|\dot{X}|^2
			\le C(\sqrt{\delta_S}+\delta_0)^2\int_0^1\tu ^2\dd y
			+\frac C{\sqrt{\delta_S}}(\G_1+\G_2).
		\end{equation}
		Multiplying \eqref{eq:shiftlowerbound} by $\delta_S/(2M)>0$
		and moving the mean term to the right, we conclude that
		\begin{equation}\label{eq:shiftleading}
			{\begin{aligned}
					-\frac{\delta_S}{2M}|\dot{X}|^2
					\le&-(\sigma^*)^2M\delta_S\bar w^2
					+C\delta_S(\sqrt{\delta_S}+\delta_0)^2\int_0^1\tu ^2\dd y
					+C\sqrt{\delta_S}(\G_1+\G_2).
			\end{aligned}}
		\end{equation}
		\subsubsection{\texorpdfstring{Estimates on $\B_1$}{Estimates on B1}}
		First write \eqref{eq:B1} as
		\begin{equation}\label{eq:B1threeintegrals}
			\begin{aligned}
				\B_1&=\B_{11}+\B_{12}+\B_{13},\\
				\B_{11}&:=\frac{\sigma^*((\gamma+1)\theta_-+2)}{2(v_-)^3}
				\intR a\bar v_\xi\tilde v^2\dd\xi,\\
				\B_{12}&:=\frac{\sigma^*}{2v_-\theta_-}
				\intR a\bar v_\xi\tilde\theta^2\dd\xi,\qquad
				\B_{13}:=-\frac{\sigma^*}{(v_-)^2}
				\intR a\bar v_\xi\tilde v\tilde\theta\dd\xi.
			\end{aligned}
		\end{equation}
		First, Young's inequality and $1\le a\le1+\sqrt{\delta_S}$ give
		\begin{align}
			\B_{11}
			=&\frac{\sigma^*((\gamma+1)\theta_-+2)}{2(v_-)^3}
			\intR a\bar v_\xi
			\left|\left(\tv+\frac \tu {\sigma^*}\right)-\frac \tu {\sigma^*}\right|^2\dd\xi\notag
			\\
			\le&\frac{\sigma^*((\gamma+1)\theta_-+2)}{2(v_-)^3}
			(1+\sqrt{\delta_S})\left[1+\delta_S^{1/4}\right]
			\intR\bar v_\xi\frac{\tu ^2}{(\sigma^*)^2}\dd\xi\notag
			\\
			&+\frac{\sigma^*((\gamma+1)\theta_-+2)}{2(v_-)^3}
			(1+\sqrt{\delta_S})\left[1+\delta_S^{-1/4}\right]
			\intR\bar v_\xi\left(\tv+\frac \tu {\sigma^*}\right)^2\dd\xi.
			\label{eq:B11Young}
		\end{align}
		By \eqref{eq:B2ymeasures} and \eqref{eq:gooddef}--\eqref{eq:goodsquares}, the two integrals are
		\begin{align}
			&\intR\bar v_\xi \tu ^2\dd\xi
			=\delta_S\int_0^1\tu ^2\dd y,\notag
			\\
			&\intR\bar v_\xi\left(\tv+\frac \tu {\sigma^*}\right)^2\dd\xi
			=\sqrt{\delta_S}\intR a_\xi
			\left(\tv+\frac \tu {\sigma^*}\right)^2\dd\xi
			=\frac{2(v_-)^2}{\sigma^*(\theta_-+1)}\sqrt{\delta_S}\G_1.
			\label{eq:B11measures}
		\end{align}
		Consequently, $\left[1+\delta_S^{-1/4}\right]
		\sqrt{\delta_S}
		\le2\delta_S^{1/4}$ and \eqref{eq:B11Young}--\eqref{eq:B11measures}  yield
		\begin{align}
			\B_{11}\le\frac{(\gamma+1)\theta_-+2}{2(v_-)^3\sigma^*}
			\left[1+C\sqrt{\delta_S}+C\delta_S^{1/4}\right]
			\delta_S\int_0^1\tu ^2\dd y
			+C\delta_S^{1/4}\G_1.
			\label{eq:B11final}
		\end{align}
		Similarly, for the second integral in \eqref{eq:B1threeintegrals},
		\begin{align}
			\B_{12}
			=&\frac{\sigma^*}{2v_-\theta_-}\intR a\bar v_\xi
			\left|\left(\tz-\frac{(\gamma-1)\theta_-}{v_-\sigma^*}\tu \right)
			+\frac{(\gamma-1)\theta_-}{v_-\sigma^*}\tu \right|^2\dd\xi\notag
			\\
			\le&\frac{\sigma^*}{2v_-\theta_-}(1+\sqrt{\delta_S})
			\left[1+\delta_S^{1/4}\right]
			\frac{(\gamma-1)^2(\theta_-)^2}{(v_-)^2(\sigma^*)^2}
			\intR\bar v_\xi \tu ^2\dd\xi\notag
			\\
			&+\frac{\sigma^*}{2v_-\theta_-}(1+\sqrt{\delta_S})
			\left[1+\delta_S^{-1/4}\right]
			\intR\bar v_\xi
			\left(\tz-\frac{(\gamma-1)\theta_-}{v_-\sigma^*}\tu \right)^2\dd\xi,
			\label{eq:B12Young}
		\end{align}
		and \eqref{eq:goodsquares} gives
		\begin{align}
			&\intR\bar v_\xi
			\left(\tz-\frac{(\gamma-1)\theta_-}{v_-\sigma^*}\tu \right)^2\dd\xi=\sqrt{\delta_S}\intR a_\xi
			\left(\tz-\frac{(\gamma-1)\theta_-}{v_-\sigma^*}\tu \right)^2\dd\xi
			=\frac{2(\gamma-1)\theta_-}{\sigma^*}\sqrt{\delta_S}\G_2.
			\label{eq:B12measure}
		\end{align}
		Substituting \eqref{eq:B12measure} into \eqref{eq:B12Young}, together with
		\eqref{eq:B11measures}, yields
		\begin{equation}\label{eq:B12final}
			\B_{12}\le\frac{(\gamma-1)^2\theta_-}{2(v_-)^3\sigma^*}
			\left[1+C\sqrt{\delta_S}+C\delta_S^{1/4}\right]
			\delta_S\int_0^1\tu ^2\dd y
			+C\delta_S^{1/4}\G_2.
		\end{equation}
		For the mixed integral, expand the two factors:
		\begin{align}
			\B_{13}
			=&-\frac{\sigma^*}{(v_-)^2}\intR a\bar v_\xi\tv\tz\dd\xi\notag
			\\
			=&\frac{\sigma^*}{(v_-)^2}\intR a\bar v_\xi
			\left\{\frac{(\gamma-1)\theta_-}{v_-(\sigma^*)^2}\tu ^2
			-\left(\tv+\frac \tu {\sigma^*}\right)
			\left(\tz-\frac{(\gamma-1)\theta_-}{v_-\sigma^*}\tu \right)\right.\notag
			\\
			&\left.-\frac{(\gamma-1)\theta_-}{v_-\sigma^*}\tu 
			\left(\tv+\frac \tu {\sigma^*}\right)
			+\frac \tu {\sigma^*}
			\left(\tz-\frac{(\gamma-1)\theta_-}{v_-\sigma^*}\tu \right)\right\}\dd\xi.
			\label{eq:B13expand}
		\end{align}
		Young's inequality gives, for the three mixed products,
		\begin{align}
			&\left|\left(\tv+\frac \tu {\sigma^*}\right)
			\left(\tz-\frac{(\gamma-1)\theta_-}{v_-\sigma^*}\tu \right)\right|\le\frac12\left(\tv+\frac \tu {\sigma^*}\right)^2
			+\frac12\left(\tz-\frac{(\gamma-1)\theta_-}{v_-\sigma^*}\tu \right)^2,\notag
			\\
			&|\tu |\left(\left|\tv+\frac \tu {\sigma^*}\right|
			+\left|\tz-\frac{(\gamma-1)\theta_-}{v_-\sigma^*}\tu \right|\right)\notag
			\\
			&\le\delta_S^{1/4}\tu ^2
			+\frac12\delta_S^{-1/4}
			\left[\left(\tv+\frac \tu {\sigma^*}\right)^2
			+\left(\tz-\frac{(\gamma-1)\theta_-}{v_-\sigma^*}\tu \right)^2\right].
			\label{eq:B13Young}
		\end{align}
		Using \eqref{eq:B13Young} in \eqref{eq:B13expand}, and then
		\eqref{eq:B11measures} and \eqref{eq:B12measure}, we obtain
		\begin{align}
			\B_{13}
			\le&\frac{(\gamma-1)\theta_-}{(v_-)^3\sigma^*}
			(1+\sqrt{\delta_S})\left[1+C\delta_S^{1/4}\right]
			\intR\bar v_\xi \tu ^2\dd\xi\notag
			\\
			&+C\delta_S^{-1/4}\intR\bar v_\xi
			\left[\left(\tv+\frac \tu {\sigma^*}\right)^2
			+\left(\tz-\frac{(\gamma-1)\theta_-}{v_-\sigma^*}\tu \right)^2\right]\dd\xi\notag
			\\
			\le&\frac{(\gamma-1)\theta_-}{(v_-)^3\sigma^*}
			\left[1+C\sqrt{\delta_S}+C\delta_S^{1/4}\right]
			\delta_S\int_0^1\tu ^2\dd y
			+C\delta_S^{1/4}(\G_1+\G_2).
			\label{eq:B13final}
		\end{align}
		Adding the three leading coefficients of \(\delta_S\int_0^1\widetilde u^2\,dy\) in \eqref{eq:B11final}, \eqref{eq:B12final}, and \eqref{eq:B13final}, and using \eqref{eq:constants}, we obtain
		\begin{align*}
			&\frac{((\gamma+1)\theta_-+2)+(\gamma-1)^2\theta_-+2(\gamma-1)\theta_-}
			{2(v_-)^3\sigma^*}
			=\frac{\gamma(\gamma+1)\theta_-+2}{2(v_-)^3\sigma^*}=\alpha^*.
		\end{align*}
		Substituting \eqref{eq:B11final}, \eqref{eq:B12final}, and
		\eqref{eq:B13final} into \eqref{eq:B1threeintegrals} therefore gives
		\begin{equation}\label{eq:B1bound}
			{\begin{aligned}
					\B_1\le&\alpha^*\left[1+C\sqrt{\delta_S}+C\delta_S^{1/4}\right]
					\delta_S\int_0^1\tu ^2\dd y
					+C\delta_S^{1/4}(\G_1+\G_2).
			\end{aligned}}
		\end{equation}
		\subsubsection{\texorpdfstring{Estimates on $\D$}{Estimates on D}}
		Set
		\begin{equation}\label{eq:Dtwodef}
			\begin{aligned}
				\D&=\intR a\left[\frac\mu v \tu _\xi^2
				+\frac\kappa{v\theta}\tz_\xi^2\right]\dd\xi=\D_1+\D_2,\\
				\D_1&:=\intR\frac{a\mu}{v}\tu _\xi^2\dd\xi,\qquad
				\D_2:=\intR\frac{a\kappa}{v\theta}\tz_\xi^2\dd\xi.
			\end{aligned}
		\end{equation}
		First, $a\ge1$ and \eqref{eq:B2ymeasures} imply
		\begin{align}
			\tu _\xi=\tu _y\frac{\dd y}{\dd\xi},\qquad
			\D_1\ge\intR\frac\mu v \tu _\xi^2\dd\xi
			=\int_0^1\frac\mu v|\tu _y|^2\frac{\dd y}{\dd\xi}\dd y.
			\label{eq:D1change}
		\end{align}
		To estimate the coefficient in \eqref{eq:D1change}, 
		by \eqref{eq:sharpdiffusion} in Appendix \ref{app:sharp}, we have
		\begin{align}
			\left|\frac1{y(1-y)}\frac\mu{\bar v}\frac{\dd y}{\dd\xi}
			-\alpha^*\frac{\mu(\gamma\theta_-+1)}
			{\mu(\gamma\theta_-+1)+\kappa(\gamma-1)^2\theta_-}\delta_S\right|
			\le C\delta_S^2.
			\label{eq:Djacobiansharp}
		\end{align}
		By \eqref{eq:smallness}, we have 
		\begin{equation}\label{eq:Dvolumeerror}
			\frac{\bar v}{v}-1=-\frac{\tv}{v},\qquad
			\left|\frac{\bar v}{v}-1\right|\le C\varepsilon_1.
		\end{equation}
		Consequently, \eqref{eq:D1change}, \eqref{eq:Djacobiansharp}, and
		\eqref{eq:Dvolumeerror} yield
		\begin{align}
			\D_1\ge&\int_0^1\frac\mu{\bar v}\frac{\dd y}{\dd\xi}|\tu _y|^2\dd y
			+\int_0^1\left(\frac{\bar v}{v}-1\right)
			\frac\mu{\bar v}\frac{\dd y}{\dd\xi}|\tu _y|^2\dd y
			\ge(1-C\varepsilon_1)
			\int_0^1\frac\mu{\bar v}\frac{\dd y}{\dd\xi}|\tu _y|^2\dd y\notag
			\\
			\ge&(1-C\varepsilon_1)\left[
			\alpha^*\frac{\mu(\gamma\theta_-+1)}
			{\mu(\gamma\theta_-+1)+\kappa(\gamma-1)^2\theta_-}\delta_S-C\delta_S^2\right]
			\int_0^1y(1-y)|\tu _y|^2\dd y\notag
			\\
			\ge&\alpha^*\frac{\mu(\gamma\theta_-+1)}
			{\mu(\gamma\theta_-+1)+\kappa(\gamma-1)^2\theta_-}
			(1-C(\delta_0+\varepsilon_1))\delta_S
			\int_0^1y(1-y)|\tu _y|^2\dd y.
			\label{eq:D1final}
		\end{align}
		Likewise, for the temperature integral in \eqref{eq:Dtwodef},
		\begin{align}
			\D_2
			\ge&\intR\frac\kappa{v\theta}\tz_\xi^2\dd\xi
			=\int_0^1\frac\kappa{v\theta}|\tz_y|^2\frac{\dd y}{\dd\xi}\dd y
			=\int_0^1\frac{\kappa\bar v}{\mu v\theta}|\tz_y|^2
			\frac\mu{\bar v}\frac{\dd y}{\dd\xi}\dd y\notag
			\\
			=&\int_0^1\frac\kappa{\mu\theta_-}|\tz_y|^2
			\frac\mu{\bar v}\frac{\dd y}{\dd\xi}\dd y
			+\int_0^1\frac\kappa\mu
			\left(\frac{\bar v}{v\theta}-\frac1{\theta_-}\right)|\tz_y|^2
			\frac\mu{\bar v}\frac{\dd y}{\dd\xi}\dd y.
			\label{eq:D2change}
		\end{align}
		The coefficient error is explicit:
		\begin{align}
			\frac{\bar v}{v\theta}-\frac1{\theta_-}
			&=\frac{\bar v-v}{v\theta}+\frac1\theta-\frac1{\theta_-}
			=-\frac{\tv}{v\theta}
			-\frac{\tz+(\bar\theta-\theta_-)}{\theta\theta_-},\notag
			\\
			\left|\frac{\bar v}{v\theta}-\frac1{\theta_-}\right|
			&\le C\big(|\tv|+|\tz|+|\bar\theta-\theta_-|\big)
			\le C(\varepsilon_1+\delta_S).
			\label{eq:D2coefficient}
		\end{align}
		Using \eqref{eq:D2coefficient} and \eqref{eq:Djacobiansharp} in
		\eqref{eq:D2change}, we obtain
		\begin{align}
			\D_2
			\ge&\frac\kappa{\mu\theta_-}(1-C(\delta_S+\varepsilon_1))
			\int_0^1\frac\mu{\bar v}\frac{\dd y}{\dd\xi}|\tz_y|^2\dd y\notag
			\\
			\ge&\alpha^*\frac{\mu(\gamma\theta_-+1)}
			{\mu(\gamma\theta_-+1)+\kappa(\gamma-1)^2\theta_-}
			\frac\kappa{\mu\theta_-}(1-C(\delta_0+\varepsilon_1))\delta_S
			\int_0^1y(1-y)|\tz_y|^2\dd y.
			\label{eq:D2final}
		\end{align}
		We now use the weighted Poincar\'e inequality
		\cite[Lemma 2.9]{KV-JEMS}: for $f\in L^2(0,1)$ with a weak derivative and
		finite right-hand side,
		\begin{equation}\label{eq:Poincare}
			\int_0^1\left|f-\int_0^1 f\dd y\right|^2\dd y
			\le\frac12\int_0^1y(1-y)|f_y|^2\dd y.
		\end{equation}
		The bound supplied by \eqref{eq:D1final} alone does not absorb
		\eqref{eq:B1bound}. 
		Indeed, its
		leading coefficient after multiplication by $3/4$, relative to $\alpha^*$, is
		\[
		\frac32\frac{\mu(\gamma\theta_-+1)}
		{\mu(\gamma\theta_-+1)+\kappa(\gamma-1)^2\theta_-}
		\longrightarrow0\qquad(\gamma\longrightarrow\infty).
		\]
		We therefore extract an additional $\tu $ term from $\D_2$.
		First, by the definition of $\bar w$ in \eqref{eq:B2yGs},
		\begin{equation}\label{eq:Poincareproof}
			\begin{aligned}
				\int_0^1|\tu -\bar w|^2\dd y
				&=\int_0^1\tu ^2\dd y-\bar w^2,\\
				\int_0^1\left|\tz-\int_0^1\tz\dd y\right|^2\dd y
				&=\int_0^1\tz^2\dd y-\left(\int_0^1\tz\dd y\right)^2.
			\end{aligned}
		\end{equation}
		Combining \eqref{eq:D1final}, \eqref{eq:D2final},
		\eqref{eq:Poincare}, and \eqref{eq:Poincareproof}, we have
		\begin{align}
			\D\ge&2\alpha^*\frac{\mu(\gamma\theta_-+1)}
			{\mu(\gamma\theta_-+1)+\kappa(\gamma-1)^2\theta_-}
			(1-C(\delta_0+\varepsilon_1))\delta_S\notag
			\\
			&\quad\times\left[\int_0^1\tu ^2\dd y-\bar w^2
			+\frac\kappa{\mu\theta_-}
			\left(\int_0^1\tz^2\dd y-\left(\int_0^1\tz\dd y\right)^2\right)\right].
			\label{eq:Dvariance}
		\end{align}
		For the temperature square in \eqref{eq:Dvariance}, expand
		along the square in $\G_2$:
		\begin{align*}
			\int_0^1\tz^2\dd y
			=&\int_0^1\left|\left(\tz-\frac{(\gamma-1)\theta_-}{v_-\sigma^*}\tu \right)
			+\frac{(\gamma-1)\theta_-}{v_-\sigma^*}\tu \right|^2\dd y\\
			=&\left(\frac{(\gamma-1)\theta_-}{v_-\sigma^*}\right)^2\int_0^1\tu ^2\dd y
			+\frac{2(\gamma-1)\theta_-}{v_-\sigma^*}\int_0^1\tu 
			\left(\tz-\frac{(\gamma-1)\theta_-}{v_-\sigma^*}\tu \right)\dd y\\
			&+\int_0^1\left(\tz-\frac{(\gamma-1)\theta_-}{v_-\sigma^*}\tu \right)^2\dd y.
		\end{align*}
		Young's inequality applied to the middle term gives
		\begin{align*}
			&\frac{2(\gamma-1)\theta_-}{v_-\sigma^*}\int_0^1\tu 
			\left(\tz-\frac{(\gamma-1)\theta_-}{v_-\sigma^*}\tu \right)\dd y\notag
			\\
			\ge&-\delta_S^{1/4}
			\left(\frac{(\gamma-1)\theta_-}{v_-\sigma^*}\right)^2\int_0^1\tu ^2\dd y
			-\delta_S^{-1/4}
			\int_0^1\left(\tz-\frac{(\gamma-1)\theta_-}{v_-\sigma^*}\tu \right)^2\dd y.
		\end{align*}
		Thus, we have
		\begin{align}
			\int_0^1\tz^2\dd y
			\ge\left(\frac{(\gamma-1)\theta_-}{v_-\sigma^*}\right)^2
			\left[1-\delta_S^{1/4}\right]\int_0^1\tu ^2\dd y
			-\delta_S^{-1/4}
			\int_0^1\left(\tz-\frac{(\gamma-1)\theta_-}{v_-\sigma^*}\tu \right)^2\dd y.
			\label{eq:temperaturesquarelower}
		\end{align}
		For the temperature mean, the same splitting and Cauchy--Schwarz yield
		\begin{align}
			\left(\int_0^1\tz\dd y\right)^2
			&=\left[\frac{(\gamma-1)\theta_-}{v_-\sigma^*}\bar w
			+\int_0^1\left(\tz-\frac{(\gamma-1)\theta_-}{v_-\sigma^*}\tu \right)\dd y\right]^2\notag
			\\
			&\le2\left(\frac{(\gamma-1)\theta_-}{v_-\sigma^*}\right)^2\bar w^2
			+2\int_0^1\left(\tz-\frac{(\gamma-1)\theta_-}{v_-\sigma^*}\tu \right)^2\dd y.
			\label{eq:temperaturemeanupper}
		\end{align}
		Subtract \eqref{eq:temperaturemeanupper} from
		\eqref{eq:temperaturesquarelower}:
		\begin{align}
			\int_0^1\tz^2\dd y-\left(\int_0^1\tz\dd y\right)^2&\ge\left(\frac{(\gamma-1)\theta_-}{v_-\sigma^*}\right)^2
			\left[1-\delta_S^{1/4}\right]\int_0^1\tu ^2\dd y
			-2\left(\frac{(\gamma-1)\theta_-}{v_-\sigma^*}\right)^2\bar w^2\notag
			\\
			&\qquad-\left[\delta_S^{-1/4}+2\right]
			\int_0^1\left(\tz-\frac{(\gamma-1)\theta_-}{v_-\sigma^*}\tu \right)^2\dd y.
			\label{eq:temperaturevariancelower}
		\end{align}
		The last integral is converted to $\G_2$ by
		\eqref{eq:B2ymeasures} and \eqref{eq:goodsquares}:
		\begin{equation}\label{eq:temperatureG2error}
			\begin{aligned}
				&\delta_S\left[
				\delta_S^{-1/4}+2
				\right]
				\int_0^1
				\left(
				\tz-\frac{(\gamma-1)\theta_-}{v_-\sigma^*}\tu
				\right)^2\dd y
				\\
				&\quad=
				\sqrt{\delta_S}
				\left[
				\delta_S^{-1/4}+2
				\right]
				\intR a_\xi
				\left(
				\tz-\frac{(\gamma-1)\theta_-}{v_-\sigma^*}\tu
				\right)^2\dd\xi
				\\
				&\quad=
				\frac{2(\gamma-1)\theta_-}{\sigma^*}
				\left[
				\delta_S^{1/4}
				+2\sqrt{\delta_S}
				\right]\G_2
				\le
				C\delta_S^{1/4}\G_2.
			\end{aligned}
		\end{equation}
		Before substituting into \eqref{eq:Dvariance}, use
		$(v_-)^2(\sigma^*)^2=\gamma\theta_-+1$ to compute the coefficients:
		\begin{align}
			\frac\kappa{\mu\theta_-}
			\left(\frac{(\gamma-1)\theta_-}{v_-\sigma^*}\right)^2
			=\frac{\kappa(\gamma-1)^2\theta_-}{\mu(\gamma\theta_-+1)},\quad
			\frac{\mu(\gamma\theta_-+1)}
			{\mu(\gamma\theta_-+1)+\kappa(\gamma-1)^2\theta_-}
			\left[1+\frac{\kappa(\gamma-1)^2\theta_-}{\mu(\gamma\theta_-+1)}\right]
			&=1,
			\label{eq:Dcoefficientidentity}
		\end{align}
		and
		\begin{align}
			&\frac{\mu(\gamma\theta_-+1)}
			{\mu(\gamma\theta_-+1)+\kappa(\gamma-1)^2\theta_-}
			\left[1+\frac{\kappa(\gamma-1)^2\theta_-}{\mu(\gamma\theta_-+1)}
			\left(1-\delta_S^{1/4}\right)\right]\notag
			\\
			=&1-\frac{\kappa(\gamma-1)^2\theta_-}
			{\mu(\gamma\theta_-+1)+\kappa(\gamma-1)^2\theta_-}
			\delta_S^{1/4}
			\ge1-\delta_S^{1/4},\notag
			\\
			&\frac{\mu(\gamma\theta_-+1)}
			{\mu(\gamma\theta_-+1)+\kappa(\gamma-1)^2\theta_-}
			\left[1+\frac{2\kappa(\gamma-1)^2\theta_-}{\mu(\gamma\theta_-+1)}\right]=\frac{\mu(\gamma\theta_-+1)+2\kappa(\gamma-1)^2\theta_-}
			{\mu(\gamma\theta_-+1)+\kappa(\gamma-1)^2\theta_-}.
			\label{eq:Dpositiveandmeancoefficients}
		\end{align}
		Substituting \eqref{eq:temperaturevariancelower}--\eqref{eq:temperatureG2error}
		into \eqref{eq:Dvariance}, and using \eqref{eq:Dcoefficientidentity}--\eqref{eq:Dpositiveandmeancoefficients}, gives
		\begin{align}
			\D\ge&2\alpha^*(1-C(\delta_0+\varepsilon_1))\delta_S
			\frac{\mu(\gamma\theta_-+1)}
			{\mu(\gamma\theta_-+1)+\kappa(\gamma-1)^2\theta_-}\notag\\
			&\quad\times\left[1+\frac{\kappa(\gamma-1)^2\theta_-}{\mu(\gamma\theta_-+1)}
			\left(1-\delta_S^{1/4}\right)\right]
			\int_0^1\tu ^2\dd y-C\delta_S^{1/4}\G_2\notag\\
			&-2\alpha^*(1-C(\delta_0+\varepsilon_1))\delta_S
			\frac{\mu(\gamma\theta_-+1)}
			{\mu(\gamma\theta_-+1)+\kappa(\gamma-1)^2\theta_-}
			\left[1+\frac{2\kappa(\gamma-1)^2\theta_-}{\mu(\gamma\theta_-+1)}\right]\bar w^2\notag\\
			\ge&2\alpha^*(1-C(\delta_0+\varepsilon_1))
			\left[1-\delta_S^{1/4}\right]
			\delta_S\int_0^1\tu ^2\dd y\notag\\
			&-2\alpha^*\frac{\mu(\gamma\theta_-+1)+2\kappa(\gamma-1)^2\theta_-}
			{\mu(\gamma\theta_-+1)+\kappa(\gamma-1)^2\theta_-}\delta_S\bar w^2
			-C\delta_S^{1/4}\G_2\notag\\
            \ge&2\alpha^*\left[1-C(\delta_0+\varepsilon_1)
					-C\delta_S^{1/4}\right]
					\delta_S\int_0^1\tu ^2\dd y\notag\\
					&-2\alpha^*\frac{\mu(\gamma\theta_-+1)+2\kappa(\gamma-1)^2\theta_-}
					{\mu(\gamma\theta_-+1)+\kappa(\gamma-1)^2\theta_-}\delta_S\bar w^2
					-C\delta_S^{1/4}\G_2.\label{eq:Dleading}
		\end{align}
		
		\subsubsection{\texorpdfstring{Combination of the leading terms}{Combination of the leading terms}}
		Combine \eqref{eq:B2bound}, \eqref{eq:shiftleading}, 
		\eqref{eq:B1bound} and \eqref{eq:Dleading}, we have
		\begin{align}
			&-\frac{\delta_S}{2M}|\dot{X}|^2+\B_1+\B_2-\G-\frac34\D\notag
			\\
			&\le C(\sqrt{\delta_S}+\varepsilon_1)\Gs
			+\left\{\alpha^*\left[1+C\sqrt{\delta_S}+C\delta_S^{1/4}\right]
			+C(\sqrt{\delta_S}+\delta_0)^2\right.\notag
			\\
			&\qquad\left.-2\alpha^*\left(\frac34-C\varepsilon_1\right)
			\left[1-C(\delta_0+\varepsilon_1)-C\delta_S^{1/4}\right]
			\right\}\delta_S\int_0^1\tu ^2\dd y\notag
			\\
			&\qquad+\left[-1+C\varepsilon_1+C\delta_S^{1/4}
			+C\sqrt{\delta_S}\right](\G_1+\G_2)
			-(\sigma^*)^2M\delta_S\bar w^2\notag
			\\
			&\qquad+2\alpha^*\left(\frac34-C\varepsilon_1\right)
			\frac{\mu(\gamma\theta_-+1)+2\kappa(\gamma-1)^2\theta_-}
			{\mu(\gamma\theta_-+1)+\kappa(\gamma-1)^2\theta_-}\delta_S\bar w^2.
			\label{eq:leadingcoefficients}
		\end{align}
		Under \eqref{eq:smallness}, the coefficients in
		\eqref{eq:leadingcoefficients} satisfy
		\begin{align*}
			&\begin{aligned}
				&\alpha^*\left[1+C\sqrt{\delta_S}+C\delta_S^{1/4}\right]
				+C(\sqrt{\delta_S}+\delta_0)^2\\
				&\quad-2\alpha^*\left(\frac34-C\varepsilon_1\right)
				\left[1-C(\delta_0+\varepsilon_1)-C\delta_S^{1/4}\right]
				\le-\frac{\alpha^*}{4},
			\end{aligned}\\
			&-1+C\varepsilon_1+C\delta_S^{1/4}
			+C\sqrt{\delta_S}\le-\frac12,\qquad
			0<\frac34-C\varepsilon_1\le\frac34.
		\end{align*}
		Hence \eqref{eq:leadingcoefficients} becomes
		\begin{align*}
			&-\frac{\delta_S}{2M}|\dot{X}|^2+\B_1+\B_2-\G-\frac34\D\\
			&\le C(\sqrt{\delta_S}+\varepsilon_1)\Gs
			-\frac{\alpha^*}{4}\delta_S\int_0^1\tu ^2\dd y-\frac12(\G_1+\G_2)
			-(\sigma^*)^2M\delta_S\bar w^2\\
			&\qquad+\frac32\alpha^*
			\frac{\mu(\gamma\theta_-+1)+2\kappa(\gamma-1)^2\theta_-}
			{\mu(\gamma\theta_-+1)+\kappa(\gamma-1)^2\theta_-}\delta_S\bar w^2.
		\end{align*}
		The choice \eqref{eq:M} cancels the last two terms, thus we have
		\begin{align}
			&-\frac{\delta_S}{2M}|\dot{X}|^2+\B_1+\B_2-\G-\frac34\D
			\le C(\sqrt{\delta_S}+\varepsilon_1)\Gs
			-\frac{\alpha^*}{4}\delta_S\int_0^1\tu ^2\dd y-\frac12(\G_1+\G_2).
			\label{eq:leadingaftermean}
		\end{align}
		Finally, recalling the first identity in \eqref{eq:B11measures},
	 using \eqref{eq:B2ymeasures} and \eqref{eq:gooddef}--\eqref{eq:goodsquares}, we obtain
		\begin{align}
			\intR\bar v_\xi\tv^2\dd\xi
			&=\intR\bar v_\xi\left[\left(\tv+\frac \tu {\sigma^*}\right)
			-\frac \tu {\sigma^*}\right]^2\dd\xi\notag
			\\
			&\le2\intR\bar v_\xi\left(\tv+\frac \tu {\sigma^*}\right)^2\dd\xi
			+\frac2{(\sigma^*)^2}\intR\bar v_\xi \tu ^2\dd\xi
			\le C\sqrt{\delta_S}\G_1
			+C\delta_S\int_0^1\tu ^2\dd y,\label{eq:leadingvolumeweight}\\
			\intR\bar v_\xi\tz^2\dd\xi
			&=\intR\bar v_\xi\left[\left(\tz-\frac{(\gamma-1)\theta_-}{v_-\sigma^*}\tu \right)
			+\frac{(\gamma-1)\theta_-}{v_-\sigma^*}\tu \right]^2\dd\xi\notag
			\\
			&\le2\intR\bar v_\xi\left(\tz-\frac{(\gamma-1)\theta_-}{v_-\sigma^*}\tu \right)^2\dd\xi
			+2\left(\frac{(\gamma-1)\theta_-}{v_-\sigma^*}\right)^2\intR\bar v_\xi \tu ^2\dd\xi\notag
			\\
			&\le C\sqrt{\delta_S}\G_2
			+C\delta_S\int_0^1\tu ^2\dd y.\label{eq:leadingtemperatureweight}
		\end{align}
		Combining \eqref{eq:B11measures},
		\eqref{eq:leadingvolumeweight}, and
		\eqref{eq:leadingtemperatureweight},
		and using \eqref{eq:gooddef}, we obtain
		\begin{equation}\label{eq:leadingGscomparison}
			\Gs\le C\left[\delta_S\int_0^1\tu ^2\dd y
			+\sqrt{\delta_S}(\G_1+\G_2)\right].
		\end{equation}
		Since $\sqrt{\delta_S}\le1$ by \eqref{eq:smallness},
		\eqref{eq:leadingGscomparison} gives, for a fixed $0<c\le1/4$,
		\begin{align}
			-\frac{\alpha^*}{4}\delta_S\int_0^1\tu ^2\dd y-\frac12(\G_1+\G_2)
			\le-2c\Gs-\frac14(\G_1+\G_2).
			\label{eq:leadingnegativecontrol}
		\end{align}
		Finally, insert \eqref{eq:leadingnegativecontrol} into
		\eqref{eq:leadingaftermean} and take $C(\sqrt{\delta_S}+\varepsilon_1)\le c$:
		\begin{align*}
			&-\frac{\delta_S}{2M}|\dot{X}|^2+\B_1+\B_2-\G-\frac34\D\\
			&\le-\big[2c-C(\sqrt{\delta_S}+\varepsilon_1)\big]\Gs-\frac14(\G_1+\G_2)
			\le-c\Gs-c(\G_1+\G_2).
		\end{align*}
		This proves \eqref{eq:leading} in Lemma \ref{lem:leading}.\end{proof}
	
	\subsection{Remaining terms and the zeroth order estimate}\label{sec:zeroth}
	Lemma~\ref{lem:leading} absorbs the leading hyperbolic, shift and viscous-thermal contributions. The energy identity still contains the remainder terms $\Y_i$, $\B_i$ and the electric fluxes $\Pp_j$. In this subsection we estimate these terms, including a weighted Poisson bound, and close the zeroth-order inequality.
	First, \eqref{eq:decomposition} and \eqref{eq:shiftidentity} give
	\begin{align*}
		\frac{\dd}{\dd t}\E
		=&-\frac{\delta_S}{2M}|\dot{X}|^2+\B_1+\B_2-\G-\frac34\D\\
		&-\frac{\delta_S}{2M}|\dot{X}|^2
		+\dot{X}\sum_{i=4}^6\Y_i+\sum_{i=3}^5\B_i-\frac14\D
		+\dot{X}\sum_{i=7}^9\Y_i+\B_6.
	\end{align*}
	By Lemma \ref{lem:leading} and Young's inequality,
	\begin{align}
		&\left|\dot{X}\sum_{i=4}^6\Y_i\right|
		\le\frac{\delta_S}{8M}|\dot{X}|^2
		+\frac{2M}{\delta_S}\left|\sum_{i=4}^6\Y_i\right|^2
		\le\frac{\delta_S}{8M}|\dot{X}|^2
		+\frac{6M}{\delta_S}\sum_{i=4}^6|\Y_i|^2,\label{eq:remainingYoung}\\
		&\begin{aligned}
			\frac{\dd}{\dd t}\E
			\le&-c\Gs-c(\G_1+\G_2)-\frac{3\delta_S}{8M}|\dot{X}|^2\\
			&+\frac{6M}{\delta_S}\sum_{i=4}^6|\Y_i|^2
			+\sum_{i=3}^5\B_i-\frac14\D
			+\dot{X}\sum_{i=7}^9\Y_i+\B_6.
		\end{aligned}\label{eq:remainingenergy}
	\end{align}
	We estimate the remaining terms in \eqref{eq:remainingenergy}.
	
	\subsubsection{\texorpdfstring{Estimates on $\Y_i$, $i=4,5,6$}{Estimates on Yi, i=4,5,6}}
	For $\Y_4,\Y_5$ defined in \eqref{eq:Y456}, the quadratic bound
	on $\Phi$, \eqref{eq:profilederivatives}, and \eqref{eq:gooddef} yield
	\begin{align*}
		|\Y_4|+|\Y_5|
		&\le\intR a|\bar\theta_\xi|
		\left[\Phi\left(\frac v{\bar v}\right)
		+\frac1{\gamma-1}\Phi\left(\frac\theta{\bar\theta}\right)\right]\dd\xi\le C\intR\bar v_\xi(\tv^2+\tz^2)\dd\xi\le C\Gs.
	\end{align*}
	In addition, \eqref{eq:profilebounds} and \eqref{eq:smallness} give
	\begin{align}
		|\Y_4|+|\Y_5|
		&\le C\norm{\bar v_\xi}_\infty\norm{(\tv,\tz)}_2^2
		\le C\delta_S^2\varepsilon_1^2,\notag
		\\
		\frac1{\delta_S}(|\Y_4|^2+|\Y_5|^2)
		&\le\frac1{\delta_S}(|\Y_4|+|\Y_5|)^2
		\le C\delta_S\varepsilon_1^2\Gs.
		\label{eq:Y45squared}
	\end{align}
	Similarly, for $\Y_6=-\intR a_\xi H_0\dd\xi$ in
	\eqref{eq:Y456}, use \eqref{eq:smallness}, \eqref{eq:H0}, \eqref{eq:weightderivative},
	and $\norm{\bar v_\xi}_\infty\le C\delta_S^2$:
	\begin{align}
		\frac1{\delta_S}|\Y_6|^2
		&\le\frac C{\delta_S}
		\left(\intR a_\xi(\tv^2+\tu ^2+\tz^2)\dd\xi\right)^2=\frac{C}{\delta_S^2}
		\left(\intR\bar v_\xi(\tv^2+\tu ^2+\tz^2)\dd\xi\right)^2\notag
		\\
		&\le\frac{C}{\delta_S^2}\norm{\bar v_\xi}_\infty
		\norm{(\tv,\tu ,\tz)}_2^2
		\intR\bar v_\xi(\tv^2+\tu ^2+\tz^2)\dd\xi\notag
		\\
		&\le C\norm{(\tv,\tu ,\tz)}_2^2
		\intR\bar v_\xi(\tv^2+\tu ^2+\tz^2)\dd\xi\notag
		\\
		&\le C\varepsilon_1^2
		\intR\bar v_\xi(\tv^2+\tu ^2+\tz^2)\dd\xi
		=C\varepsilon_1^2\Gs.
		\label{eq:Y6bound}
	\end{align}
	Thus, \eqref{eq:Y45squared} and \eqref{eq:Y6bound}
	together with \eqref{eq:remainingYoung} imply
	\begin{equation}\label{eq:Y456bound}
		\begin{gathered}
			\frac{6M}{\delta_S}\sum_{i=4}^6|\Y_i|^2\le C\varepsilon_1^2\Gs,\qquad
			\left|\dot{X}\sum_{i=4}^6\Y_i\right|
			\le\frac{\delta_S}{8M}|\dot{X}|^2+C\varepsilon_1^2\Gs.
		\end{gathered}
	\end{equation}
	\subsubsection{\texorpdfstring{Estimates on $\B_i$, $i=3,4,5$}{Estimates on Bi, i=3,4,5}}
	
	For $\B_3$ in \eqref{eq:B3}, first observe that
	\begin{gather}
		\frac{u_\xi}{v}-\frac{\bar u_\xi}{\bar v}
		=\frac{\tu _\xi}{v}-\frac{\bar u_\xi\tv}{v\bar v},\qquad
		\frac{\theta_\xi}{v}-\frac{\bar\theta_\xi}{\bar v}
		=\frac{\tz_\xi}{v}-\frac{\bar\theta_\xi\tv}{v\bar v}.
		\label{eq:B3fluxes}
	\end{gather}
	Young's inequality and \eqref{eq:Ddef} give
	\begin{align}
		|\B_3|
		\le&C\intR a_\xi|\tu |\big(|\tu _\xi|+|\bar u_\xi||\tv|\big)\dd\xi+C\intR a_\xi|\tz|\big(|\tz_\xi|+|\bar\theta_\xi||\tv|\big)\dd\xi\notag
		\\
		\le&\frac1{80}\D+C\intR a_\xi^2(\tu ^2+\tz^2)\dd\xi
		+C\intR(\bar u_\xi^2+\bar\theta_\xi^2)\tv^2\dd\xi.
		\label{eq:B3Young}
	\end{align}
	By \eqref{eq:weightderivative}, \eqref{eq:profilebounds},
	\eqref{eq:profilederivatives}, and \eqref{eq:gooddef},
	\begin{align}
		\intR a_\xi^2(\tu ^2+\tz^2)\dd\xi
		&=\frac{1}{\delta_S}\intR\bar v_\xi^2(\tu ^2+\tz^2)\dd\xi\le\frac{1}{\delta_S}\norm{\bar v_\xi}_\infty
		\intR\bar v_\xi(\tu ^2+\tz^2)\dd\xi
		\le C\delta_S\Gs,\label{eq:B3weightproduct}\\
		\intR(\bar u_\xi^2+\bar\theta_\xi^2)\tv^2\dd\xi
		&\le C\intR\bar v_\xi^2\tv^2\dd\xi
		\le C\delta_S^2\intR\bar v_\xi\tv^2\dd\xi
		\le C\delta_S^2\Gs.\label{eq:B3profileproduct}
	\end{align}
	Using \eqref{eq:B3weightproduct}--\eqref{eq:B3profileproduct}
	in \eqref{eq:B3Young} gives
	\begin{equation}\label{eq:B3final}
		|\B_3|\le\frac1{80}\D+C(\delta_S+\delta_S^2)\Gs
		\le\frac1{80}\D+C\delta_S\Gs.
	\end{equation}
	
	For $\B_4$ in \eqref{eq:B4}, we have
	\begin{align}
		\frac{\kappa\tz\theta_\xi}{\theta^2}
		\left(\frac{\theta_\xi}{v}-\frac{\bar\theta_\xi}{\bar v}\right)
		=&\frac{\kappa\tz}{v\theta^2}\tz_\xi^2
		+\frac{\kappa\tz\bar\theta_\xi}{v\theta^2}\tz_\xi-\frac{\kappa\tz\bar\theta_\xi\tv}{v\bar v\theta^2}\tz_\xi
		-\frac{\kappa\tz\bar\theta_\xi^2\tv}{v\bar v\theta^2},\label{eq:B4heatexpand}\\
		\frac{\mu\tz}{\theta}\left(\frac{u_\xi^2}{v}-\frac{\bar u_\xi^2}{\bar v}\right)
		=&\frac{\mu\tz}{v\theta}\tu _\xi^2
		+\frac{2\mu\tz\bar u_\xi}{v\theta}\tu _\xi
		-\frac{\mu\tz\bar u_\xi^2\tv}{v\bar v\theta},\label{eq:B4viscexpand}\\
		\left|\kappa\left(\frac{\bar\theta_\xi}{\bar v}\right)_\xi
		+\mu\frac{\bar u_\xi^2}{\bar v}\right|
		\le&C\big(|\bar\theta_{\xi\xi}|+\bar v_\xi^2\big)
		\le C\delta_S\bar v_\xi.\label{eq:B4profile}
	\end{align}
	Consequently, \eqref{eq:B3fluxes} and
	\eqref{eq:B4heatexpand}--\eqref{eq:B4profile} yield
	\begin{align*}
		|\B_4|\le&C\intR\Big[|\bar u_\xi||\tv||\tu _\xi|
		+|\tz|\big(|\tz_\xi|+|\bar\theta_\xi|\big)
		\big(|\tz_\xi|+|\bar\theta_\xi||\tv|\big)\Big]\dd\xi\\
		&+C\intR\Big[|\bar\theta_\xi||\tv||\tz_\xi|
		+|\tz|\big(\tu _\xi^2+|\bar u_\xi||\tu _\xi|+\bar u_\xi^2|\tv|\big)\Big]\dd\xi+C\intR\tz^2\big(|\bar\theta_{\xi\xi}|+\bar v_\xi^2
		+\bar\theta_\xi^2+\bar u_\xi^2\big)\dd\xi.
	\end{align*}
	By Young's inequality, \eqref{eq:smallness}, and
	\eqref{eq:profilederivatives},
	\begin{align}
		|\B_4|
		\le&\left(C\varepsilon_1+\frac1{80}\right)\D
		+C\intR\big(|\bar\theta_{\xi\xi}|+\bar v_\xi^2
		+\bar\theta_\xi^2+\bar u_\xi^2\big)(\tv^2+\tz^2)\dd\xi\notag
		\\
		\le&\left(C\varepsilon_1+\frac1{80}\right)\D
		+C(\delta_S+\delta_S^2)\intR\bar v_\xi(\tv^2+\tz^2)\dd\xi
		\le\frac1{40}\D+C\delta_S\Gs,
		\label{eq:B4final}
	\end{align}
	where the last inequality holds for $C\varepsilon_1\le1/80$.

	$\B_5$ is estimated in \eqref{eq:B5final}. Thus,
	combining \eqref{eq:B3final}, \eqref{eq:B4final}, and
	\eqref{eq:B5final}, we have
	\begin{equation}\label{eq:B345bound}
		\sum_{i=3}^5\B_i\le\sum_{i=3}^5|\B_i|
		\le\frac1{20}\D+C(\delta_S+\varepsilon_1)\Gs.
	\end{equation}
	Insert \eqref{eq:Y456bound} and \eqref{eq:B345bound} into
	\eqref{eq:remainingenergy}. Taking $\delta_0,\varepsilon_1$ small absorbs
	$C(\delta_S+\varepsilon_1+\varepsilon_1^2)\Gs$;  we obtain
	\begin{equation}\label{eq:beforeelectric}
		\frac{\dd}{\dd t}\E
		\le-\frac{3\delta_S}{8M}|\dot{X}|^2-\frac c 2\Gs-c(\G_1+\G_2)-\frac3{16}\D
		+\dot{X}\sum_{i=7}^9\Y_i+\B_6.
	\end{equation}
	\subsubsection{\texorpdfstring{A weighted Poisson estimate}{A weighted Poisson estimate}}
	Multiplication of \eqref{eq:perturbPDE}$_4$ by $e^{\bar\phi}$ yields
	\begin{equation}\label{eq:poissondifference}
		\la^2e^{\bar\phi}\left[\frac{\tp_\xi}{v}
		+\bar\phi_\xi\left(\frac1v-\frac1{\bar v}\right)\right]_\xi
		=-\tv+v(1-e^{-\tp}).
	\end{equation}
	Rewriting \eqref{eq:poissondifference}, we obtain
	\begin{equation}\label{eq:positivepoisson}
		-\la^2\left(\frac{\tp_\xi}{v}\right)_\xi
		+ve^{-\bar\phi}\left(\int_0^1e^{-s\tp}\dd s\right)\tp
		=e^{-\bar\phi}\tv
		+\la^2\left[\bar\phi_\xi\left(\frac1v-\frac1{\bar v}\right)\right]_\xi.
	\end{equation}
	Multiply \eqref{eq:positivepoisson} by $\bar v_\xi\tp$ and integrate.
	Using \eqref{eq:profilederivatives}, we obtain
	\begin{align}
		&\la^2\intR\frac{\bar v_\xi}{v}\tp_\xi^2\dd\xi
		+\intR\bar v_\xi ve^{-\bar\phi}
		\left(\int_0^1e^{-s\tp}\dd s\right)\tp^2\dd\xi\notag
		\\
		&=\intR\bar v_\xi e^{-\bar\phi}\tv\tp\dd\xi
		-\la^2\intR\frac{\bar v_{\xi\xi}}v\tp\tp_\xi\dd\xi-\la^2\intR(\bar v_{\xi\xi}\tp+\bar v_\xi\tp_\xi)
		\bar\phi_\xi\left(\frac1v-\frac1{\bar v}\right)\dd\xi\notag
		\\
		&\le C\intR\bar v_\xi|\tv\tp|\dd\xi
		+C\delta_S\intR\bar v_\xi|\tp\tp_\xi|\dd\xi
		+C\delta_S^2\intR\bar v_\xi|\tv|(|\tp|+|\tp_\xi|)\dd\xi.
		\label{eq:weightedpoissonIBP}
	\end{align}
	Under \eqref{eq:smallness}, for a fixed $c>0$,
	\begin{equation*}
		\frac{\la^2}{v}\ge c,\qquad
		ve^{-\bar\phi}\int_0^1e^{-s\tp}\dd s\ge c.
	\end{equation*}
	Applying Young's inequality to \eqref{eq:weightedpoissonIBP},
	we  obtain
	\begin{align}
		c\intR\bar v_\xi(\tp^2+\tp_\xi^2)\dd\xi
		&\le C\intR\bar v_\xi\tv^2\dd\xi
		+\left(\frac c2+C\delta_S\right)
		\intR\bar v_\xi(\tp^2+\tp_\xi^2)\dd\xi.
		\label{eq:weightedpoissonabsorb}
	\end{align}
	By \eqref{eq:gooddef} and \eqref{eq:weightedpoissonabsorb}, it holds
	\begin{equation}\label{eq:weightedpoisson}
		{\quad\intR\bar v_\xi(\tp^2+\tp_\xi^2)\dd\xi
			\le C\intR\bar v_\xi\tv^2\dd\xi\le C\Gs.\quad}
	\end{equation}
	For the estimates below, set
	\begin{equation}\label{eq:Rp}
		\Rp=\intR\left(
		\tv_\xi^2+\tp_\xi^2+\tp_{\xi\xi}^2
		+\chi_\xi^2+\chi_{\xi\xi}^2
		\right)\dd\xi,
	\end{equation}
	where $\chi$ is defined in \eqref{eq:chi}.
	We will also use the following consequences 
	for any functions $f,g$
	\begin{align}
		\intR a_\xi|f g|\dd\xi
		&\le\frac{1}{\sqrt{\delta_S}}\left(\intR\bar v_\xi f^2\dd\xi\right)^{1/2}
		\left(\intR\bar v_\xi g^2\dd\xi\right)^{1/2}\notag
		\\
		&\le\frac{1}{\sqrt{\delta_S}}\norm{\bar v_\xi}_\infty^{1/2}
		\left(\intR\bar v_\xi f^2\dd\xi\right)^{1/2}\norm g_2
		\le C\sqrt{\delta_S}\left(\intR\bar v_\xi f^2\dd\xi\right)^{1/2}\norm g_2,
		\label{eq:aweightedproduct}\\
		\intR\bar v_\xi|f g|\dd\xi
		&\le\norm{\bar v_\xi}_\infty^{1/2}
		\left(\intR\bar v_\xi f^2\dd\xi\right)^{1/2}\norm g_2
		\le C\delta_S\left(\intR\bar v_\xi f^2\dd\xi\right)^{1/2}\norm g_2,
		\label{eq:vweightedproduct}\\
		\norm{\bar v_\xi}_2^2
		&\le\norm{\bar v_\xi}_\infty\intR\bar v_\xi\dd\xi
		=\delta_S\norm{\bar v_\xi}_\infty\le C\delta_S^3.
		\label{eq:profileLtwo}
	\end{align}
	\subsubsection{\texorpdfstring{Estimates on $\Pp_1$ and $\Pp_2$}{Estimates on P1 and P2}} 
	The terms $\Pp_1,\Pp_2$ are defined in \eqref{eq:P12},
	and $R_F$ in  \eqref{eq:RFexpand}. First, we calculate
	\begin{align}
		&\frac1{v^2}-\frac1{\bar v^2}
		=-\frac{v+\bar v}{v^2\bar v^2}\tv,\notag
		\\
		&\frac{v_\xi\phi_\xi}{v^3}-\frac{\bar v_\xi\bar\phi_\xi}{\bar v^3}
		=\frac{\tv_\xi\tp_\xi}{v^3}
		+\frac{\bar v_\xi\tp_\xi+\bar\phi_\xi\tv_\xi}{v^3}
		+\bar v_\xi\bar\phi_\xi\left(\frac1{v^3}-\frac1{\bar v^3}\right),\notag
		\\
		&\frac{\phi_\xi^2}{v^2}-\frac{\bar\phi_\xi^2}{\bar v^2}
		=\frac{\tp_\xi^2+2\bar\phi_\xi\tp_\xi}{v^2}
		+\bar\phi_\xi^2\left(\frac1{v^2}-\frac1{\bar v^2}\right).
		\label{eq:RFproducts}
	\end{align}
	Inserting \eqref{eq:RFproducts} into \eqref{eq:RFexpand},
	then using \eqref{eq:profilederivatives} and \eqref{eq:smallness}, yields
	\begin{align}
		|R_F|\le&C\big[|\tv||\tp_{\xi\xi}|+|\bar\phi_{\xi\xi}||\tv|
		+|\tv_\xi\tp_\xi|+\bar v_\xi(|\tv_\xi|+|\tp_\xi|)
		+\bar v_\xi^2|\tv|+\tp_\xi^2\big]\notag
		\\
		\le&C\varepsilon_1(|\tp_{\xi\xi}|+|\tv_\xi|+|\tp_\xi|)
		+C\bar v_\xi(|\tv_\xi|+|\tp_\xi|)+C\delta_S\bar v_\xi|\tv|.
		\label{eq:RFbound}
	\end{align}
	By \eqref{eq:N}, \eqref{eq:RFbound},
	\eqref{eq:aweightedproduct}--\eqref{eq:profileLtwo}, and \eqref{eq:Rp}, one has
	\begin{align}
		|\Pp_1|
		&=\left|\intR a_\xi \tu \left(\frac{\la^2}{\bar v^2}\tp_{\xi\xi}+R_F\right)\dd\xi\right|\notag
		\\
		&\le C\intR a_\xi|\tu |\big[|\tp_{\xi\xi}|+|\tv_\xi|+|\tp_\xi|
		+\delta_S\bar v_\xi|\tv|\big]\dd\xi\notag
		\\
		&\le C\sqrt{\delta_S}\left(\intR\bar v_\xi \tu ^2\dd\xi\right)^{1/2}
		\norm{(\tp_{\xi\xi},\tv_\xi,\tp_\xi)}_2
		+C\delta_S^{5/2}\intR\bar v_\xi|\tu \tv|\dd\xi\notag
		\\
		&\le C\sqrt{\delta_S}\sqrt{\Gs}\sqrt{\Rp}+C\delta_S^{5/2}\Gs
		\le C\sqrt{\delta_S}(\Gs+\Rp).
		\label{eq:P1bound}
	\end{align}
	Likewise, using \eqref{eq:Ddef} and \eqref{eq:RFbound}, we have
	\begin{align}
		|\Pp_2|
		&=\left|\intR a\tu _\xi R_F\dd\xi\right|\le C\varepsilon_1\intR|\tu _\xi|(|\tp_{\xi\xi}|+|\tv_\xi|+|\tp_\xi|)\dd\xi\notag
		\\
		&\quad\qquad\qquad\qquad\qquad+C\intR\bar v_\xi|\tu _\xi|(|\tv_\xi|+|\tp_\xi|)\dd\xi
		+C\delta_S\intR\bar v_\xi|\tu _\xi\tv|\dd\xi\notag
		\\
		&\le C(\varepsilon_1+\delta_S^2)\norm{\tu _\xi}_2
		\norm{(\tp_{\xi\xi},\tv_\xi,\tp_\xi)}_2
		+C\delta_S^2\left(\intR\bar v_\xi\tv^2\dd\xi\right)^{1/2}\norm{\tu _\xi}_2\notag
		\\
		&\le C(\varepsilon_1+\delta_S^2)(\D+\Rp)
		+C\delta_S^2\sqrt{\D}\sqrt{\Gs}
		\le C(\varepsilon_1+\delta_S^2)(\D+\Gs+\Rp).
		\label{eq:P2bound}
	\end{align}
	\subsubsection{\texorpdfstring{Estimates on $\Pp_3$}{Estimates on P3}}
	Substituting \eqref{eq:poissondifference} into the definition of $\mathcal N_\phi$
	in \eqref{eq:N}, we have
	\begin{align}
		\mathcal N_\phi=&(v-e^{\bar\phi})\tp+v(1-e^{-\tp}-\tp)\notag
		\\
		&+\la^2e^{\bar\phi}\left[\left(\frac1{\bar v}-\frac1v\right)\tp_{\xi\xi}
		+\frac{v_\xi}{v^2}\tp_\xi
		-\left\{\bar\phi_\xi\left(\frac1v-\frac1{\bar v}\right)\right\}_\xi\right].
		\label{eq:Nexpand}
	\end{align}
	For $\Pp_3$ defined in \eqref{eq:P34}, insert
	\eqref{eq:Nexpand} and integrate the first two terms by parts:
	\begin{align}
		\Pp_3=&-\intR\left[\frac{a\la^2}{\bar v^2}
		\big\{(v-e^{\bar\phi})\tp+v(1-e^{-\tp}-\tp)\big\}\right]_\xi\chi_\xi\dd\xi\notag
		\\
		&+\intR\frac{a\la^4e^{\bar\phi}}{\bar v^2}
		\left[\left(\frac1{\bar v}-\frac1v\right)\tp_{\xi\xi}
		+\frac{v_\xi}{v^2}\tp_\xi
		-\left\{\bar\phi_\xi\left(\frac1v-\frac1{\bar v}\right)\right\}_\xi\right]
		\chi_{\xi\xi}\dd\xi.
		\label{eq:P3IBP}
	\end{align}
	For the profile \eqref{eq:shock}, Poisson's equation
	\eqref{eq:NSP} and \eqref{eq:profilederivatives} give
	\begin{gather}
		\bar v-e^{\bar\phi}
		=-\la^2e^{\bar\phi}\left(\frac{\bar\phi_\xi}{\bar v}\right)_\xi,\qquad
		|\bar v-e^{\bar\phi}|\le C\delta_S\bar v_\xi,
		\qquad |(\bar v-e^{\bar\phi})_\xi|\le C\delta_S^2\bar v_\xi,
		\label{eq:P3profiledefect}
	\end{gather}
	and
	\begin{align}
		&\big[(v-e^{\bar\phi})\tp+v(1-e^{-\tp}-\tp)\big]_\xi\notag
		\\
		&=\big[\tv_\xi+(\bar v-e^{\bar\phi})_\xi\big]\tp
		+(v-e^{\bar\phi})\tp_\xi
		+v_\xi(1-e^{-\tp}-\tp)+v(e^{-\tp}-1)\tp_\xi,
		\label{eq:P3productderivative}\\
		&|1-e^{-\tp}-\tp|\le C\tp^2,
		\qquad |e^{-\tp}-1|\le C|\tp|.\notag
	\end{align}
	For the derivative terms in \eqref{eq:P3IBP},
	\eqref{eq:profilederivatives} and \eqref{eq:smallness} yield
	\begin{align}
		&\left|\left(\frac1{\bar v}-\frac1v\right)\tp_{\xi\xi}
		+\frac{v_\xi}{v^2}\tp_\xi
		-\left\{\bar\phi_\xi\left(\frac1v-\frac1{\bar v}\right)\right\}_\xi\right|\notag
		\\
		&\le C\varepsilon_1(|\tp_{\xi\xi}|+|\tv_\xi|)
		+C\bar v_\xi(|\tp_\xi|+|\tv_\xi|)
		+C\delta_S\bar v_\xi|\tv|.
		\label{eq:P3derivativebound}
	\end{align}
	Therefore, combining \eqref{eq:P3IBP}--\eqref{eq:P3derivativebound},
	and then applying \eqref{eq:weightedpoisson},
	\eqref{eq:aweightedproduct}--\eqref{eq:profileLtwo}, together with
	the definition of $\Rp$ in \eqref{eq:Rp}, we obtain
	\begin{align}
		|\Pp_3|\le&C\varepsilon_1\intR
		(|\tv_\xi|+|\tp_\xi|)|\chi_\xi|\dd\xi+C\intR(a_\xi+\bar v_\xi)
		\big[\varepsilon_1(|\tv|+|\tp|)+\delta_S\bar v_\xi|\tp|\big]|\chi_\xi|\dd\xi\notag
		\\
		&+C\intR\big[(\delta_S^2+\varepsilon_1)\bar v_\xi|\tp|
		+\delta_S\bar v_\xi|\tp_\xi|\big]|\chi_\xi|\dd\xi\notag
		\\
		&+C\intR\big[\varepsilon_1(|\tp_{\xi\xi}|+|\tv_\xi|)
		+\bar v_\xi(|\tp_\xi|+|\tv_\xi|)
		+\delta_S\bar v_\xi|\tv|\big]|\chi_{\xi\xi}|\dd\xi\notag
		\\
		\le&C\varepsilon_1\Rp
		+C\big[(\sqrt{\delta_S}+\delta_S)\varepsilon_1+\delta_S^2\big]\sqrt{\Gs}\sqrt{\Rp}
		+C\delta_S^2\Rp\notag
		\\
		\le&C(\varepsilon_1+\delta_S)(\Gs+\Rp).
		\label{eq:P3bound}
	\end{align}
	\subsubsection{\texorpdfstring{Estimates on $\Pp_4$, $\Pp_5$, and $\G_\phi$}{Estimates on P4, P5, and G phi}}
	Recall $\Pp_4$ from \eqref{eq:P34}, $\Pp_5$ from
	\eqref{eq:P5}, and $\G_\phi$ from \eqref{eq:Jgoodassembled}, with
	$\eta_\phi$ given in \eqref{eq:entropy}. The coefficient derivatives are
	\begin{align}
		&\left(\frac{\la^2}{\bar v^2}\right)_\xi
		=-\frac{2\la^2}{\bar v^3}\bar v_\xi,\qquad
		\left(\frac{\la^4e^{\bar\phi}}{\bar v^3}\right)_\xi
		=\frac{\la^4e^{\bar\phi}}{\bar v^3}
		\left(\bar\phi_\xi-3\frac{\bar v_\xi}{\bar v}\right),\notag
		\\
		&\left(\frac{\la^2e^{\bar\phi}}{\bar v^2}\right)_\xi
		=\frac{\la^2e^{\bar\phi}}{\bar v^2}
		\left(\bar\phi_\xi-2\frac{\bar v_\xi}{\bar v}\right).
		\label{eq:electriccoefficientderivatives}
	\end{align}
	Using \eqref{eq:profilederivatives},
	\eqref{eq:weightderivative},
	\eqref{eq:weightedpoisson}, \eqref{eq:aweightedproduct},
	\eqref{eq:vweightedproduct}, and the definition of $\Rp$ in
	\eqref{eq:Rp}, together with
	$\sqrt{\delta_S}+\delta_S\le C\sqrt{\delta_S}$, we obtain
	\begin{align}
		&\left|\left(\frac{a\la^2e^{\bar\phi}}{\bar v^2}\right)_\xi\right|
		=\frac{\la^2e^{\bar\phi}}{\bar v^2}
		\left|a_\xi+a\left(\bar\phi_\xi-2\frac{\bar v_\xi}{\bar v}\right)\right|
		\le C(a_\xi+\bar v_\xi),\notag
		\\
		&|\Pp_4|
		=\left|\intR\left(\frac{a\la^2e^{\bar\phi}}{\bar v^2}\right)_\xi
		\tp\chi_\xi\dd\xi\right|
		\le C\intR(a_\xi+\bar v_\xi)|\tp\chi_\xi|\dd\xi\notag
		\\
		&\qquad\le C(\sqrt{\delta_S}+\delta_S)\sqrt{\Gs}\sqrt{\Rp}
		\le C\sqrt{\delta_S}(\Gs+\Rp).
		\label{eq:P4final}
	\end{align}
	Likewise,
	\eqref{eq:P5} and \eqref{eq:electriccoefficientderivatives} give
	\begin{align}
		|\Pp_5|
		&\le C\intR\bar v_\xi
		\big(|\tv\tp_{\xi\xi}|+\tp_{\xi\xi}^2+\tp_\xi^2\big)\dd\xi\notag
		\\
		&\le C\delta_S\left(\intR\bar v_\xi\tv^2\dd\xi\right)^{1/2}
		\norm{\tp_{\xi\xi}}_2
		+C\norm{\bar v_\xi}_\infty\norm{(\tp_\xi,\tp_{\xi\xi})}_2^2\notag
		\\
		&\le C\delta_S\sqrt{\Gs}\norm{\tp_{\xi\xi}}_2
		+C\delta_S^2\norm{(\tp_\xi,\tp_{\xi\xi})}_2^2
		\le C\delta_S(\Gs+\Rp).
		\label{eq:P5bound}
	\end{align}
	Finally, the expression \eqref{eq:entropy} for $\eta_\phi$
	and \eqref{eq:Jgoodassembled} imply
	\begin{align}
		|\G_\phi|
		&=\sigma\left|\intR a_\xi\left[
		\frac{\la^2}{\bar v^2}\tv\tp_{\xi\xi}
		+\frac{\la^4e^{\bar\phi}}{2\bar v^3}\tp_{\xi\xi}^2
		+\frac{\la^2e^{\bar\phi}}{2\bar v^2}\tp_\xi^2\right]\dd\xi\right|\notag
		\\
		&\le C\intR a_\xi
		\big(|\tv\tp_{\xi\xi}|+\tp_{\xi\xi}^2+\tp_\xi^2\big)\dd\xi\notag
		\\
		&\le C\sqrt{\delta_S}\sqrt{\Gs}\norm{\tp_{\xi\xi}}_2
		+C\norm{a_\xi}_\infty\norm{(\tp_\xi,\tp_{\xi\xi})}_2^2\notag
		\\
		&\le C\sqrt{\delta_S}\sqrt{\Gs}\sqrt{\Rp}+C\delta_S^{3/2}\Rp
		\le C\sqrt{\delta_S}(\Gs+\Rp).
		\label{eq:Gphibound}
	\end{align}
	\subsubsection{\texorpdfstring{Estimates on $\dot{X}\Y_i$, $i=7,8,9$}{Estimates on XdotYi, i=7,8,9}}
	For $\Y_7$ in
	\eqref{eq:Y7}, by \eqref{eq:P5},
	\eqref{eq:Y7}, \eqref{eq:Jgoodassembled}, one has
	\begin{align*}
		\Y_7
		=&-\intR a_\xi\eta_\phi\dd\xi
		-\intR a\left[\left(\frac{\la^2}{\bar v^2}\right)_\xi\tv\tp_{\xi\xi}
		+\frac12\left(\frac{\la^4e^{\bar\phi}}{\bar v^3}\right)_\xi\tp_{\xi\xi}^2
		+\frac12\left(\frac{\la^2e^{\bar\phi}}{\bar v^2}\right)_\xi\tp_\xi^2\right]\dd\xi\\
		=&\frac1\sigma(\Pp_5-\G_\phi).
	\end{align*}
	Using $|\dot{X}|\le C\varepsilon_1$ from
	\eqref{eq:shiftbound}, together with \eqref{eq:P5bound}--\eqref{eq:Gphibound},
	we obtain
	\begin{align}
		|\dot{X}\Y_7|
		&\le\frac{|\dot{X}|}{\sigma}(|\Pp_5|+|\G_\phi|)\le C\varepsilon_1\intR(a_\xi+\bar v_\xi)
		\big(|\tv\tp_{\xi\xi}|+\tp_{\xi\xi}^2+\tp_\xi^2\big)\dd\xi\notag
		\\
		&\le C\varepsilon_1\left[(\sqrt{\delta_S}+\delta_S)\sqrt{\Gs}\sqrt{\Rp}
		+(\delta_S^{3/2}+\delta_S^2)\Rp\right]
		\le C\varepsilon_1(\Gs+\Rp).
		\label{eq:Y7final}
	\end{align}
	For $\Y_8$ in \eqref{eq:Y89}, first use
	\eqref{eq:profilederivatives}, \eqref{eq:gooddef}, and
	\eqref{eq:profileLtwo} to obtain
	\begin{gather}
		|\bar\phi_{\xi\xi}|\le C\delta_S\bar v_\xi,
		\qquad |\bar\phi_{\xi\xi\xi}|\le C\delta_S^2\bar v_\xi,
		\label{eq:Y8profilederivatives}\\
		\begin{aligned}
			\intR\bar v_\xi|\tv|\dd\xi
			&\le\left(\intR\bar v_\xi\dd\xi\right)^{1/2}
			\left(\intR\bar v_\xi\tv^2\dd\xi\right)^{1/2}
			\le\delta_S^{1/2}\sqrt{\Gs},\\
			\intR\bar v_\xi|\tp_\xi|\dd\xi
			&\le\norm{\bar v_\xi}_2\norm{\tp_\xi}_2
			\le C\delta_S^{3/2}\norm{\tp_\xi}_2,\\
			\intR\bar v_\xi|\tp_{\xi\xi}|\dd\xi
			&\le\norm{\bar v_\xi}_2\norm{\tp_{\xi\xi}}_2
			\le C\delta_S^{3/2}\norm{\tp_{\xi\xi}}_2.
		\end{aligned}\label{eq:Y8weightedLone}
	\end{gather}
	Substitution into \eqref{eq:Y89} and Young's inequality give
	\begin{align}
		|\dot{X}\Y_8|
		=&|\dot{X}|\left|\intR a\left[
		\left(\frac{\la^2}{\bar v^2}\tv
		+\frac{\la^4e^{\bar\phi}}{\bar v^3}\tp_{\xi\xi}\right)\bar\phi_{\xi\xi\xi}
		+\frac{\la^2e^{\bar\phi}}{\bar v^2}\tp_\xi\bar\phi_{\xi\xi}
		\right]\dd\xi\right|\notag
		\\
		\le&C|\dot{X}|\left[\delta_S^2\intR\bar v_\xi
		\big(|\tv|+|\tp_{\xi\xi}|\big)\dd\xi
		+\delta_S\intR\bar v_\xi|\tp_\xi|\dd\xi\right]\notag
		\\
		\le&C|\dot{X}|\left[\delta_S^{5/2}\sqrt{\Gs}
		+\delta_S^{7/2}\norm{\tp_{\xi\xi}}_2
		+\delta_S^{5/2}\norm{\tp_\xi}_2\right]\notag
		\\
		\le&C\delta_S^2|\dot{X}|^2
		+C\delta_S^3\Gs+C\delta_S^5\norm{\tp_{\xi\xi}}_2^2
		+C\delta_S^3\norm{\tp_\xi}_2^2\notag
		\\
		\le&C\delta_S^2|\dot{X}|^2+C\delta_S^3(\Gs+\Rp).
		\label{eq:Y8final}
	\end{align}
	For $\Y_9$ in \eqref{eq:Y89}, the same argument using
	\eqref{eq:profileLtwo} yields
	\begin{align}
		|\dot{X}\Y_9|
		&=|\dot{X}|\left|\intR\frac{a\la^2}{\bar v^2}
		\bar v_\xi\tp_{\xi\xi}\dd\xi\right|
		\le C|\dot{X}|\norm{\bar v_\xi}_2\norm{\tp_{\xi\xi}}_2\notag
		\\
		&\le C\big(\delta_S|\dot{X}|\big)
		\big(\delta_S^{1/2}\norm{\tp_{\xi\xi}}_2\big)
		\le C\delta_S^2|\dot{X}|^2+C\delta_S\norm{\tp_{\xi\xi}}_2^2.
		\label{eq:Y9bound}
	\end{align}
	Finally, use $\B_6=\sum_{j=1}^5\Pp_j-\G_\phi$ from
	\eqref{eq:B6}. Adding \eqref{eq:P1bound}--\eqref{eq:P3bound},
	\eqref{eq:P4final}--\eqref{eq:Gphibound}, and
	\eqref{eq:Y7final}--\eqref{eq:Y9bound}, we obtain
	\begin{align}
		\left|\dot{X}\sum_{i=7}^9\Y_i+\B_6\right|
		&\le\sum_{i=7}^9|\dot{X}\Y_i|+\sum_{j=1}^5|\Pp_j|+|\G_\phi|\notag
		\\
		&\le C\delta_S^2|\dot{X}|^2
		+C(\sqrt{\delta_S}+\varepsilon_1)(\Gs+\D+\Rp).
		\label{eq:electricassembled}
	\end{align}
	Since $\delta_S^2\le\sqrt{\delta_S}\,\delta_S$, this implies
	\begin{equation}\label{eq:electricfinal}
		{\quad
			\left|\dot{X}\sum_{i=7}^9\Y_i+\B_6\right|
			\le C(\sqrt{\delta_S}+\varepsilon_1)
			\left[\delta_S|\dot{X}|^2+\Gs+\D+\Rp\right].\quad}
	\end{equation}
	\subsubsection{\texorpdfstring{The zeroth order inequality}{The zeroth order inequality}}
	
	\begin{proposition}\label{prop:final}
		Under the hypotheses of Proposition~\ref{thm:closure},	it holds
		\begin{equation}\label{eq:final}
			{\begin{aligned}
					\frac{\dd}{\dd t}\E
					\le&-\frac{\delta_S}{4M}|\dot{X}|^2-c\Gs-c(\G_1+\G_2)-\frac18\D\\
					&+C(\sqrt{\delta_S}+\varepsilon_1)
					\intR\left[\tv_\xi^2+\tp_\xi^2+\tp_{\xi\xi}^2
					+\chi_\xi^2+\chi_{\xi\xi}^2\right]\dd\xi.
			\end{aligned}}
		\end{equation}
	\end{proposition}\begin{proof}
		Insert \eqref{eq:electricassembled}--\eqref{eq:electricfinal} into \eqref{eq:beforeelectric}.
		The terms $C(\sqrt{\delta_S}+\varepsilon_1)
		[\delta_S|\dot{X}|^2+\Gs+\D]$ are absorbed. This gives \eqref{eq:final} and completes the proof of Proposition \ref{prop:final}. 
		\end{proof}
	Moreover, 
	integrating \eqref{eq:final} over \([0,t]\), we obtain
	\begin{align}
		\E(t)+\int_0^t\left[\frac{\delta_S}{4M}|\dot{X}|^2
		+c\Gs+c(\G_1+\G_2)+\frac18\D\right]\dd s 
		\le\E(0)+C(\sqrt{\delta_S}+\varepsilon_1)\int_0^t\Rp(s)\dd s.
		\label{eq:integratedfinal}
	\end{align}
	
	We next estimate the spatial and time derivatives in $\Rp(s)$ \eqref{eq:Rp} and
	combine them with the differentiated fluid equations. This is the step
	that closes \eqref{eq:integratedfinal}.
	
	\subsection{Spatial and time-differentiated Poisson estimates}
	\label{sec:poisson}
	The zeroth-order inequality \eqref{eq:integratedfinal} still contains the remainder $\Rp$, which includes spatial and time derivatives of the potential perturbation. In this subsection we use the coercive Poisson equation in \eqref{eq:NSP} to control these electric terms: spatial estimates give the sign of the electric energy, and time-differentiated estimates absorb the corresponding remainder.
	For this section and the next one, we set
	\begin{align}
		&\mathfrak q=\frac1v-\frac1{\bar v}=-\frac{\tv}{v\bar v},\qquad
		\mathfrak q_\xi=-\frac{\tv_\xi}{v^2}
		+\bar v_\xi\left(\frac1{\bar v^2}-\frac1{v^2}\right),\notag
		\\
		&\mathfrak q_{\xi\xi}=-\frac{\tv_{\xi\xi}}{v^2}
		+\frac{2v_\xi\tv_\xi}{v^3}
		+\bar v_{\xi\xi}\left(\frac1{\bar v^2}-\frac1{v^2}\right)
		+2\bar v_\xi\left(\frac{v_\xi}{v^3}-\frac{\bar v_\xi}{\bar v^3}\right),\notag
		\\
		&|\mathfrak q|\le C|\tv|,\quad
		|\mathfrak q_\xi|\le C(|\tv_\xi|+\bar v_\xi|\tv|),\quad
		|\mathfrak q_{\xi\xi}|\le C\left[|\tv_{\xi\xi}|
		+(\varepsilon_1+\delta_S^2)|\tv_\xi|
		+\delta_S\bar v_\xi|\tv|\right].\label{eq:Qderivatives}
	\end{align}
	\subsubsection{\texorpdfstring{Spatial estimates and the sign of the electric term}{Spatial estimates and the sign of the electric term}}
	\begin{lemma}\label{lem:spatial}
		Under the hypotheses of Proposition~\ref{thm:closure}, for $k=0,1,2$ and $j=1,2$, it holds
		\begin{align}
			&\norm{\tp}_{H^{k+1}}^2\le C\norm{\tv}_{H^k}^2,\label{eq:ellipticHk}\\
			&\norm{\tp_\xi}_{H^1}^2
			\le C\norm{\tv_\xi}_2^2+C(\delta_0+\varepsilon_1)\Gs,
			\label{eq:ellipticgradient}\\
			&\intR(\partial_\xi^j\tv)(\partial_\xi^j\tp)\dd\xi
			\ge c\norm{\partial_\xi^j\tp}_{H^1}^2
			-C(\delta_0+\varepsilon_1)
			\left[\norm{\tv_\xi}_{H^{j-1}}^2+\intR\bar v_\xi\tv^2\dd\xi\right].
			\label{eq:electricpair}
		\end{align}
	\end{lemma}\begin{proof}
		First multiply \eqref{eq:positivepoisson} by $\tp$. Positivity of
		$v,\bar v$ and smallness of $\tp$ give
		\begin{align}
			c\norm{\tp}_{H^1}^2
			&\le\intR e^{-\bar\phi}\tv\tp\dd\xi
			-\la^2\intR\bar\phi_\xi \mathfrak q\tp_\xi\dd\xi\le\frac c2\norm{\tp}_{H^1}^2+C\norm{\tv}_2^2.
			\label{eq:spatialzero}
		\end{align}
		\medskip\noindent\textbf{1. One spatial derivative.}
		Multiplying $\partial_\xi$\eqref{eq:poissondifference} by
		$\tp_\xi$, one has
		\begin{align*}
			\intR\tv_\xi\tp_\xi\dd\xi
			=&\intR\left(ve^{-\tp}\tp_\xi^2
			+\frac{\la^2e^{\bar\phi}}v\tp_{\xi\xi}^2\right)\dd\xi+\intR v_\xi(1-e^{-\tp})\tp_\xi\dd\xi\\
			&+\la^2\intR e^{\bar\phi}\tp_{\xi\xi}
			\left[-\frac{v_\xi}{v^2}\tp_\xi+(\bar\phi_\xi \mathfrak q)_\xi\right]\dd\xi.
		\end{align*}
		Using \eqref{eq:weightedpoisson},
		\eqref{eq:aweightedproduct}--\eqref{eq:profileLtwo}, and \eqref{eq:Qderivatives}, we obtain
		\begin{align*}
			&\left|\intR v_\xi(1-e^{-\tp})\tp_\xi\dd\xi\right|
			\le C\varepsilon_1\norm{\tv_\xi}_2\norm{\tp_\xi}_2
			+C\delta_S\left(\intR\bar v_\xi\tv^2\dd\xi\right)^{1/2}\norm{\tp_\xi}_2,
			\\
			&\norm{(\bar\phi_\xi \mathfrak q)_\xi}_2
			\le C\delta_S^2\norm{\tv_\xi}_2
			+C\delta_S^2\left(\intR\bar v_\xi\tv^2\dd\xi\right)^{1/2},\\
			&\left|\intR\tv_\xi\tp_\xi\dd\xi
			-\intR\left(ve^{-\tp}\tp_\xi^2
			+\frac{\la^2e^{\bar\phi}}v\tp_{\xi\xi}^2\right)\dd\xi\right|
			\le C(\delta_0+\varepsilon_1)
			\left[\norm{\tv_\xi}_2^2+\norm{\tp_\xi}_{H^1}^2
			+\intR\bar v_\xi\tv^2\dd\xi\right].
		\end{align*}
		Absorption proves \eqref{eq:electricpair} for $j=1$.
		Applying Young's inequality to its left side gives
		\begin{equation}\label{eq:spatialone}
			\norm{\tp_\xi}_{H^1}^2
			\le C\norm{\tv_\xi}_2^2
			+C(\delta_0+\varepsilon_1)\intR\bar v_\xi\tv^2\dd\xi.
		\end{equation}
		\medskip\noindent\textbf{2. Two spatial derivatives.}
		By the same argument after two differentiations, it holds that
		\begin{align*}
			&\intR\tv_{\xi\xi}\tp_{\xi\xi}\dd\xi
			\\=&\intR\left(ve^{-\tp}\tp_{\xi\xi}^2
			+\frac{\la^2e^{\bar\phi}}v\tp_{\xi\xi\xi}^2\right)\dd\xi+\intR\left[v_{\xi\xi}(1-e^{-\tp})
			+2v_\xi e^{-\tp}\tp_\xi-ve^{-\tp}\tp_\xi^2\right]\tp_{\xi\xi}\dd\xi\\
			&+\la^2\intR e^{\bar\phi}\tp_{\xi\xi\xi}
			\left[-\frac{2v_\xi}{v^2}\tp_{\xi\xi}
			+\left(\frac{2v_\xi^2}{v^3}-\frac{v_{\xi\xi}}{v^2}\right)\tp_\xi
			+(\bar\phi_\xi \mathfrak q)_{\xi\xi}
			+\bar\phi_\xi\left(\frac{\tp_{\xi\xi}}v
			-\frac{v_\xi}{v^2}\tp_\xi+(\bar\phi_\xi \mathfrak q)_\xi\right)\right]\dd\xi.
		\end{align*}
		We take $(\bar\phi_\xi \mathfrak q)_{\xi\xi}$ as an example, the other terms can be treated in the same way. We have
		\begin{align*}
			(\bar\phi_\xi \mathfrak q)_{\xi\xi}
			=\bar\phi_{\xi\xi\xi}\mathfrak q+2\bar\phi_{\xi\xi}\mathfrak q_\xi
			+\bar\phi_\xi \mathfrak q_{\xi\xi},\qquad\norm{(\bar\phi_\xi \mathfrak q)_{\xi\xi}}_2
			\le C\delta_S^2\norm{\tv_\xi}_{H^1}
			+C\delta_S^3\left(\intR\bar v_\xi\tv^2\dd\xi\right)^{1/2}.
		\end{align*}
		Thus, we have 
		\begin{align}
			\left|\intR\tv_{\xi\xi}\tp_{\xi\xi}\dd\xi
			-\intR\left(ve^{-\tp}\tp_{\xi\xi}^2
			+\frac{\la^2e^{\bar\phi}}v\tp_{\xi\xi\xi}^2\right)\dd\xi\right|\le C(\delta_0+\varepsilon_1)
			\left[\norm{\tv_\xi}_{H^1}^2+\norm{\tp_\xi}_{H^2}^2
			+\intR\bar v_\xi\tv^2\dd\xi\right].\label{eq:pairtwoerror}
		\end{align}
		Use \eqref{eq:spatialone} for the first derivative of $\tp$ and
		absorb the remaining small terms in \eqref{eq:pairtwoerror}.
		This proves \eqref{eq:electricpair} for $j=2$. Applying Young's inequality to \eqref{eq:electricpair} gives
		$\norm{\tp_{\xi\xi}}_{H^1}^2\le C\norm{\tv}_{H^2}^2$.
		Together with \eqref{eq:spatialzero} and \eqref{eq:spatialone}, this estimate proves
		\eqref{eq:ellipticHk} and \eqref{eq:ellipticgradient}.\end{proof}
	\subsubsection{\texorpdfstring{Time derivatives and the electric remainder}{Time derivatives and the electric remainder}}
	\begin{lemma}\label{lem:time}
		For $\chi$ defined in \eqref{eq:chi}, it holds that
		\begin{align}
			&\norm\chi_{H^2}^2
			\le C\D+C(\delta_0+\varepsilon_1)^2
			\left(\norm{\tv_\xi}_2^2+\norm{\tu _{\xi\xi}}_2^2+\Gs\right),
			\label{eq:chiH2}\\
			&\Rp\le C\left(\D+\norm{\tv_\xi}_2^2\right)
			+C(\delta_0+\varepsilon_1)\left(\Gs+\norm{\tu _{\xi\xi}}_2^2\right).
			\label{eq:Rphiabsorb}
		\end{align}
	\end{lemma}\begin{proof}
		
		Applying $\partial_t-\sigma\partial_\xi$ to  \eqref{eq:NSP}$_4$, and using the first equation of the same system,
		we obtain
		\begin{equation}\label{eq:timeelliptic}
			-\la^2\left(\frac{\psi_\xi}{v}\right)_\xi+ve^{-\phi}\psi
			=e^{-\phi}u_\xi-\la^2\left(\frac{\phi_\xi u_\xi}{v^2}\right)_\xi,
			\quad \psi=(\dt)\phi,\quad\chi=\psi+\sigma\bar\phi_\xi.
		\end{equation}
		Applying $-\sigma\partial_\xi$ to 
		\eqref{eq:shiftedNSP}$_4$ and using $\bar u_\xi=-\sigma\bar v_\xi$, we obtain
		\[
		\sigma\lambda_D^2
		\left(\frac{\bar\phi_{\xi\xi}}{\bar v}\right)_\xi
		-\sigma\bar v e^{-\bar\phi}\bar\phi_\xi
		=
		e^{-\bar\phi}\bar u_\xi
		-\lambda_D^2
		\left(
		\frac{\bar\phi_\xi\bar u_\xi}{\bar v^2}
		\right)_\xi.
		\]
		Subtracting this identity from \eqref{eq:timeelliptic} and using
		\[
		\psi=\chi-\sigma\bar\phi_\xi,\qquad
		u_\xi=\widetilde u_\xi+\bar u_\xi,\qquad
		\mathfrak q=\frac1v-\frac1{\bar v},
		\]
		we have
		\begin{equation}\label{eq:chiequation}
			-\la^2\left(\frac{\chi_\xi}{v}\right)_\xi+ve^{-\phi}\chi
			=f_0+(f_1)_\xi,
		\end{equation}
		where
		\begin{align*}
			&f_0=e^{-\phi}\tu _\xi+(e^{-\phi}-e^{-\bar\phi})\bar u_\xi
			+\sigma(ve^{-\phi}-\bar v e^{-\bar\phi})\bar\phi_\xi,\\
			&f_1=-\la^2\left[\frac{\phi_\xi}{v^2}\tu _\xi
			+\left(\frac{\phi_\xi}{v^2}-\frac{\bar\phi_\xi}{\bar v^2}\right)\bar u_\xi
			+\sigma \mathfrak q\bar\phi_{\xi\xi}\right],\\
			&(f_1)_\xi=-\la^2\left\{
			\frac{\phi_\xi}{v^2}\tu _{\xi\xi}
			+\left(\frac{\phi_{\xi\xi}}{v^2}-\frac{2v_\xi\phi_\xi}{v^3}\right)\tu _\xi
			+\left(\frac{\phi_\xi}{v^2}-\frac{\bar\phi_\xi}{\bar v^2}\right)\bar u_{\xi\xi}\right.\\
			&\quad\quad\qquad\qquad\left.
			+\left(\frac{\phi_{\xi\xi}}{v^2}-\frac{2v_\xi\phi_\xi}{v^3}
			-\frac{\bar\phi_{\xi\xi}}{\bar v^2}
			+\frac{2\bar v_\xi\bar\phi_\xi}{\bar v^3}\right)\bar u_\xi
			+\sigma \mathfrak q_\xi\bar\phi_{\xi\xi}+\sigma \mathfrak q\bar\phi_{\xi\xi\xi}
			\right\}.
		\end{align*}
		Combining \eqref{eq:smallness},
		\eqref{eq:weightedpoisson}, and \eqref{eq:Qderivatives} yields
		\begin{align}
			\norm{f_0}_2^2&\le C\norm{\tu _\xi}_2^2+C\delta_S^2\Gs,\notag
			\\
			\norm{f_1}_2^2+\norm{(f_1)_\xi}_2^2
			&\le C(\delta_0+\varepsilon_1)^2
			\left[\norm{\tu _\xi}_{H^1}^2+\norm{\tv_\xi}_2^2
			+\norm{\tp_\xi}_{H^1}^2+\Gs\right].\label{eq:chiforcingbounds}
		\end{align}
		Multiplication of \eqref{eq:chiequation} by $\chi$, followed by
		solving that equation for $\chi_{\xi\xi}$, gives successively
		\begin{align}
			c\norm\chi_{H^1}^2
			&\le\intR(f_0\chi-f_1\chi_\xi)\dd\xi
			\le\frac c2\norm\chi_{H^1}^2+C(\norm{f_0}_2^2+\norm{f_1}_2^2),\notag
			\\
			\frac{\la^2}{v}\chi_{\xi\xi}
			&=\frac{\la^2v_\xi}{v^2}\chi_\xi+ve^{-\phi}\chi-f_0-(f_1)_\xi,\notag
			\\
			\norm\chi_{H^2}^2
			&\le C\big(\norm{f_0}_2^2+\norm{f_1}_2^2+\norm{(f_1)_\xi}_2^2\big).
			\label{eq:chitest}
		\end{align}
		Inserting \eqref{eq:ellipticgradient} and \eqref{eq:chiforcingbounds}
		into \eqref{eq:chitest}, this proves \eqref{eq:chiH2}. By the definition of $\mathcal{R}_\phi$
		\eqref{eq:Rp} and combining \eqref{eq:chiH2}, we prove \eqref{eq:Rphiabsorb}.\end{proof}
	Combining \eqref{eq:final} and \eqref{eq:Rphiabsorb}, and reducing
	$\delta_0,\varepsilon_1$, we have
	\begin{equation}\label{eq:zeroforclosure}
		{\quad
			\dot{\E}+c\left[\delta_S|\dot{X}|^2+\Gs+\G_1+\G_2+\D\right]
			\le C(\sqrt{\delta_0}+\varepsilon_1)
			\left(\norm{\tv_\xi}_2^2+\norm{\tu _{\xi\xi}}_2^2\right).
			\quad}
	\end{equation}
	
	\begin{proof}[Proof of Lemma \ref{lem:basicenergy}]
		Integrating \eqref{eq:zeroforclosure} on $[0,s]$ for $0\le s\le t$ gives
		\eqref{eq:basicenergy}.
	\end{proof}
	\subsection{First and second derivative estimates}
	\label{sec:derivatives}
	To close the $H^2$ hypothesis in
	\eqref{eq:smallness} and obtain \eqref{eq:closedestimate}, we carry out both derivative levels.
	In Subsections~\ref{sec:derivatives} and~\ref{sec:closure}, for the
	solution fields, the profile fields in \eqref{eq:shiftedprofile}, and
	$\mathfrak q$ defined in \eqref{eq:Qderivatives}, write
	$f_j=\partial_\xi^j f$, $f_0=f$.
	The perturbation equations in the original pressure variables are
	\begin{equation}\label{eq:originalperturbation}
		\left\{\begin{aligned}
			&(\dt)\tv=\tu _1+\dot{X}\bar v_1,\\
			&(\dt)\tu =-(p-\bar p)_1
			+\mu\left(\frac{\tu _1}{v}+\bar u_1\mathfrak q\right)_1
			+\frac{\tp_1}{v}+\bar\phi_1\mathfrak q+\dot{X}\bar u_1,\\
			&\cv(\dt)\tz=-p \tu _1-(p-\bar p)\bar u_1
			+\kappa\left(\frac{\tz_1}{v}+\bar\theta_1\mathfrak q\right)_1\\
			&\qquad\qquad\qquad\quad+\mu\left(\frac{\tu _1^2+2\bar u_1\tu _1}{v}+\bar u_1^2\mathfrak q\right)
			+\cv\dot{X}\bar\theta_1.
		\end{aligned}\right.
	\end{equation}
	The electric force in the second equation produces the pairing
	\eqref{eq:electricpair}. We therefore keep this form for the derivative
	estimates. Define, for $j=1,2$,
	\begin{align}
		\mathcal K_j&=\intR\left(\frac\mu2\tv_j^2-v\tu _{j-1}\tv_j\right)\dd\xi,
		\qquad
		\mathcal L_j=\frac12\norm{\tu _j}_2^2+\frac{\cv}{2}\norm{\tz_j}_2^2,
		\label{eq:derivativeenergies}\\
		\mathcal K_j&\ge\frac\mu4\norm{\tv_j}_2^2-C\norm{\tu _{j-1}}_2^2,
		\qquad
		|\mathcal K_j|\le C\big(\norm{\tv_j}_2^2+\norm{\tu _{j-1}}_2^2\big).
		\label{eq:Kcoercivity}
	\end{align}
	\subsubsection{\texorpdfstring{Explicit coefficient derivatives}{Explicit coefficient derivatives}}
	For use in each energy identity, we separate only the highest
	derivative. The remaining expressions are written out for both values of
	$j$:
	\begin{align}
		&\partial_\xi^j(p-\bar p)=\frac{\tz_j}{v}-\frac\theta{v^2}\tv_j+R_{p,j},\qquad\partial_\xi^j\left(\frac{\tu _1}{v}+\bar u_1\mathfrak q\right)
		=\frac{\tu _{j+1}}v+R_{u,j},\notag\\
		&\partial_\xi^j\left(\frac{\tz_1}{v}+\bar\theta_1\mathfrak q\right)
		=\frac{\tz_{j+1}}v+R_{\theta,j},\qquad\partial_\xi^{j-1}\left(\frac{\tp_1}{v}+\bar\phi_1\mathfrak q\right)
		=\frac{\tp_j}v+R_{\phi,j},\label{eq:highestderivatives}
	\end{align}
	where
	\begin{align}
		R_{p,1}=&\bar\theta_1\mathfrak q
		+\bar v_1\left(\frac{\bar\theta}{\bar v^2}-\frac\theta{v^2}\right),\notag
		\\
		R_{p,2}=&\bar\theta_2\mathfrak q
		+\bar v_2\left(\frac{\bar\theta}{\bar v^2}-\frac\theta{v^2}\right)
		-2\left(\frac{\theta_1v_1}{v^2}-\frac{\bar\theta_1\bar v_1}{\bar v^2}\right)
		+2\left(\frac{\theta v_1^2}{v^3}-\frac{\bar\theta\bar v_1^2}{\bar v^3}\right),
		\label{eq:pressurecommutators}\\
		R_{u,1}=&-\frac{v_1}{v^2}\tu _1+\bar u_2\mathfrak q+\bar u_1\mathfrak q_1,\notag
		\\
		R_{u,2}=&-\frac{2v_1}{v^2}\tu _2
		+\left(\frac{2v_1^2}{v^3}-\frac{v_2}{v^2}\right)\tu _1
		+\bar u_3\mathfrak q+2\bar u_2\mathfrak q_1+\bar u_1\mathfrak q_2,\notag
		\\
		R_{\theta,1}=&-\frac{v_1}{v^2}\tz_1+\bar\theta_2\mathfrak q+\bar\theta_1\mathfrak q_1,\notag
		\\
		R_{\theta,2}=&-\frac{2v_1}{v^2}\tz_2
		+\left(\frac{2v_1^2}{v^3}-\frac{v_2}{v^2}\right)\tz_1
		+\bar\theta_3\mathfrak q+2\bar\theta_2\mathfrak q_1+\bar\theta_1\mathfrak q_2,\notag
		\\
		R_{\phi,1}=&\bar\phi_1\mathfrak q,\qquad
		R_{\phi,2}=-\frac{v_1}{v^2}\tp_1+\bar\phi_2\mathfrak q+\bar\phi_1\mathfrak q_1.\notag
	\end{align}
	We take the product
	$\theta_1v_1/v^2-\bar\theta_1\bar v_1/\bar v^2$ in
	\eqref{eq:pressurecommutators} as a representative term.
	Expanding $\theta_1=\tz_1+\bar\theta_1$ and
	$v_1=\tv_1+\bar v_1$, we obtain
	\begin{align*}
		\frac{\theta_1v_1}{v^2}-\frac{\bar\theta_1\bar v_1}{\bar v^2}
		&=\frac{\tz_1\tv_1+\bar\theta_1\tv_1+\bar v_1\tz_1}{v^2}
		+\bar\theta_1\bar v_1\left(\frac1{v^2}-\frac1{\bar v^2}\right).
	\end{align*}
	By \eqref{eq:smallness}, \eqref{eq:profilebounds}--\eqref{eq:profilederivatives},
	\eqref{eq:Qderivatives}, and the definition of $\Gs$ in \eqref{eq:gooddef},
	H\"older's inequality and the embedding $H^2\hookrightarrow W^{1,\infty}$ give
	\begin{align*}
		\left\|\frac{\theta_1v_1}{v^2}
		-\frac{\bar\theta_1\bar v_1}{\bar v^2}\right\|_2
		&\le C\norm{\tz_1}_\infty\norm{\tv_1}_2
		+C\delta_S^2\norm{(\tv_1,\tz_1)}_2
		+C\norm{\bar v_1^2\tv}_2\\
		&\le C(\varepsilon_1+\delta_S^2)\norm{(\tv_1,\tz_1)}_2
		+C\delta_S^3(\Gs)^{1/2}\\
		&\le C(\delta_0+\varepsilon_1)
		\left[\norm{(\tv_1,\tz_1)}_2+(\Gs)^{1/2}\right].
	\end{align*}
	The remaining terms are estimated in the same way. 
	Consequently, \eqref{eq:profilederivatives},
	\eqref{eq:Qderivatives}, and the $H^2$ embedding into $W^{1,\infty}$ give
	\begin{align}
		&\norm{R_{p,1}}_2^2+\norm{R_{\phi,1}}_2^2
		\le C\delta_S^2\Gs,\qquad
		\norm{R_{p,2}}_2^2
		\le C(\delta_0+\varepsilon_1)^2
		\left(\norm{(\tv_1,\tz_1)}_2^2+\Gs\right),\notag
		\\
		&\norm{R_{u,j}}_2^2+\norm{R_{\theta,j}}_2^2
		\le C(\delta_0+\varepsilon_1)^2
		\left(\norm{(\tv_1,\tu _1,\tz_1)}_{H^{j-1}}^2+\Gs\right),\quad j=1,2,\notag
		\\
		&\norm{R_{\phi,2}}_2^2
		\le C(\delta_0+\varepsilon_1)^2
		\left(\norm{(\tv_1,\tp_1)}_2^2+\Gs\right).
		\label{eq:commutatorbounds}
	\end{align}
	
	\subsubsection{\texorpdfstring{The density derivative energies}{The density derivative energies}}
	\begin{lemma}\label{lem:density}
		For the energies ${\mathcal K}_1,{\mathcal K}_2$ defined in \eqref{eq:derivativeenergies}, it holds
		\begin{align}
			\dot{\mathcal K}_1+c\left(\norm{\tv_1}_2^2+\norm{\tp_1}_{H^1}^2\right)
			&\le C\D+C(\delta_0+\varepsilon_1)\Gs
			+C\delta_S^2|\dot{X}|^2,\label{eq:Kone}\\
			\dot{\mathcal K}_2+c\left(\norm{\tv_2}_2^2+\norm{\tp_2}_{H^1}^2\right)
			&\le C\left(\D+\norm{(\tu _2,\tz_2)}_2^2\right)\notag
			\\
			&\quad+C(\delta_0+\varepsilon_1)\left(\norm{\tv_1}_2^2+\Gs\right)
			+C\delta_S^2|\dot{X}|^2.\label{eq:Ktwo}
		\end{align}
	\end{lemma}\begin{proof}
		Multiply the $\partial_\xi^j$\eqref{eq:originalperturbation}$_1$ by $\mu\tv_j$ and
		the $\partial_\xi^{j-1}$\eqref{eq:originalperturbation}$_2$ by $-v\tv_j$.
		When the mixed energy is differentiated, use $(\dt)v=u_1$ and integrate
		$-v\tu _{j-1}\tu _{j+1}$ by parts. The highest viscous terms cancel, and we obtain
		\begin{align}
			&\dot{\mathcal K}_j
			+\intR\frac\theta v\tv_j^2\dd\xi+\intR\tv_j\tp_j\dd\xi\notag\\
			=&\intR\tv_j\tz_j\dd\xi+\intR v \tu _j^2\dd\xi+\intR v\tv_j(R_{p,j}-\mu R_{u,j}-R_{\phi,j})\dd\xi\notag
			\\
			&+\intR\left(v_1\tu _{j-1}\tu _j-u_1\tu _{j-1}\tv_j\right)\dd\xi+\dot{X}\intR\left[
			\mu\bar v_{j+1}\tv_j-v\bar u_j\tv_j-v\tu _{j-1}\bar v_{j+1}
			\right]\dd\xi.\label{eq:Kidentity}
		\end{align}
		\medskip\noindent\textbf{1. Pressure, viscosity, and the electric pairing.} 
		The temperature cross term is controlled by its own derivative
		energy. For any fixed sufficiently small ${\tilde\eta}>0$,
		\begin{align}
			&\left|\intR\tv_j\tz_j\dd\xi\right|
			\le{\tilde\eta}\norm{\tv_j}_2^2+C_{\tilde\eta}\norm{\tz_j}_2^2,\qquad\qquad
			\intR v \tu _j^2\dd\xi\le C\norm{\tu _j}_2^2,\notag
			\\
			&\left|\intR v\tv_1(R_{p,1}-\mu R_{u,1}-R_{\phi,1})\dd\xi\right|
			\le\big({\tilde\eta}+C(\delta_0+\varepsilon_1)\big)\norm{\tv_1}_2^2
			+C\D+C_{\tilde\eta}(\delta_0+\varepsilon_1)\Gs,\notag
			\\
			&\left|\intR v\tv_2(R_{p,2}-\mu R_{u,2}-R_{\phi,2})\dd\xi\right|
			\le C(\delta_0+\varepsilon_1)
			\left[\norm{\tv_1}_{H^1}^2+\norm{\tu _1,\tz_1}_{H^1}^2+\Gs\right].
			\label{eq:Kcoefficientestimates}
		\end{align}
		For the term $-v_1\tp_1/v^2$ in $R_{\phi,2}$, use \eqref{eq:ellipticgradient} to bound $\|\tp_1\|_2^2$ in terms of $\|\tv_1\|_2^2$ and $\Gs$. The left side of \eqref{eq:Kidentity} has the sign
		\begin{align}
			\intR\frac\theta v\tv_j^2\dd\xi+\intR\tv_j\tp_j\dd\xi
			\ge&c\left(\norm{\tv_j}_2^2+\norm{\tp_j}_{H^1}^2\right)-C(\delta_0+\varepsilon_1)
			\left(\norm{\tv_1}_{H^{j-1}}^2+\Gs\right).
			\label{eq:Kpositive}
		\end{align}
		\medskip\noindent\textbf{2. The transport products.}
		At the first derivative level, the perturbation cubic terms cancel
		exactly. At the second level, all remaining factors of first derivatives
		are small in $L^\infty$:
		\begin{align}
			&v_1\tu \tu _1-u_1\tu \tv_1
			=\bar v_1\tu \tu _1-\bar u_1\tu \tv_1,\notag
			\\
			&\left|\intR(v_1\tu \tu _1-u_1\tu \tv_1)\dd\xi\right|
			\le C\delta_S\sqrt{\Gs}\left(\norm{\tu _1}_2+\norm{\tv_1}_2\right)
			\le{\tilde\eta}\norm{\tv_1}_2^2+C\D+C_{\tilde\eta}\delta_S^2\Gs,\notag
			\\
			&\left|\intR(v_1\tu _1\tu _2-u_1\tu _1\tv_2)\dd\xi\right|
			\le C(\delta_0+\varepsilon_1)
			\left(\norm{\tv_2}_2^2+\norm{\tu _1}_2^2+\norm{\tu _2}_2^2\right).
			\label{eq:Ktransport}
		\end{align}
		\medskip\noindent\textbf{3. The shift terms.}
		In the first energy, the term with the undifferentiated $\tu $ is
		integrated once more by parts. This avoids a time integral of $\norm \tu _2^2$:
		\begin{align*}
			&-\intR v\tu \bar v_2\dd\xi
			=\intR(v_1\tu +v\tu _1)\bar v_1\dd\xi,\notag
			\\
			&\left|\intR(v_1\tu +v\tu _1)\bar v_1\dd\xi\right|
			\le C\delta_S^{3/2}
			\left(\varepsilon_1\norm{\tv_1}_2+\norm{\tu _1}_2\right)
			+C\delta_S^{5/2}\sqrt{\Gs}.
		\end{align*}
		Combining these estimates with the bounds
		$\norm{\bar v_j,\bar u_j,\bar\theta_j}_2
		\le C\delta_S^{j+1/2}$
		for $j=1,2,3$, we obtain
		\begin{align}
			& \left|\dot{X}\intR
			\left(\mu\bar v_2\tv_1-v\bar u_1\tv_1-v\tu \bar v_2\right)\dd\xi\right|
			\le C\delta_S^2|\dot{X}|^2
			+C\delta_S\left(\norm{\tv_1}_2^2+\D+\Gs\right),\notag
			\\
			&\left|\dot{X}\intR
			\left(\mu\bar v_3\tv_2-v\bar u_2\tv_2-v\tu _1\bar v_3\right)\dd\xi\right|
			\le C\delta_S^2|\dot{X}|^2
			+C\delta_S\left(\norm{\tv_2}_2^2+\D\right).
			\label{eq:Kshift}
		\end{align}
		Inserting \eqref{eq:Kcoefficientestimates}--\eqref{eq:Kshift} into
		\eqref{eq:Kidentity}, we have
		\eqref{eq:Kone} and \eqref{eq:Ktwo}.\end{proof}
	\subsubsection{\texorpdfstring{Velocity and temperature derivative energies}{Velocity and temperature derivative energies}}
	\begin{lemma}\label{lem:parabolic}
		For the energies ${\mathcal L}_1,{\mathcal L}_2$ defined in \eqref{eq:derivativeenergies}, it holds
		\begin{align}
			&\dot{\mathcal L}_1+c\norm{(\tu _2,\tz_2)}_2^2
			\le C\left(\norm{\tv_1}_2^2+\D\right)
			+C(\delta_0+\varepsilon_1)\Gs+C\delta_S^2|\dot{X}|^2,
			\label{eq:Lone}\\
			&\dot{\mathcal L}_2+c\norm{\tu _3,\tz_3}_2^2
			\le C\left(\norm{\tv_1}_{H^1}^2+\norm{(\tu _2,\tz_2)}_2^2+\D\right)+C(\delta_0+\varepsilon_1)\Gs+C\delta_S^2|\dot{X}|^2.
			\label{eq:Ltwo}
		\end{align}
	\end{lemma}\begin{proof}\leavevmode\par
		\medskip\noindent\textbf{1. The differentiated momentum equation.}
		Apply $\partial_\xi^{j-1}$ to 
		\eqref{eq:originalperturbation}$_2$, multiply by $-\tu _{j+1}$, and integrate.
		Using \eqref{eq:highestderivatives}, we get
		\begin{align*}
			\frac12\frac{\dd}{\dd t}\norm{\tu _j}_2^2
			+\mu\intR\frac{\tu _{j+1}^2}{v}\dd\xi
			=&\intR\left(\frac{\tz_j}{v}-\frac\theta{v^2}\tv_j\right)\tu _{j+1}\dd\xi-\dot{X}\intR\bar u_j\tu _{j+1}\dd\xi\\
			&+\intR(R_{p,j}-\mu R_{u,j})\tu _{j+1}\dd\xi
			-\intR\left(\frac{\tp_j}{v}+R_{\phi,j}\right)\tu _{j+1}\dd\xi.
		\end{align*}
		Each integral is estimated separately. 
		\eqref{eq:ellipticgradient} and \eqref{eq:commutatorbounds} are used for
		the electric force and coefficient terms, respectively:
		\begin{align*}
			&\left|\intR\left(\frac{\tz_j}{v}-\frac\theta{v^2}\tv_j\right)\tu _{j+1}\dd\xi\right|
			\le{\tilde\eta}\norm{\tu _{j+1}}_2^2+C_{\tilde\eta}\norm{\tv_j,\tz_j}_2^2,\\
			&\left|\intR(R_{p,j}-\mu R_{u,j})\tu _{j+1}\dd\xi\right|
			\le{\tilde\eta}\norm{\tu _{j+1}}_2^2
			+C_{\tilde\eta}(\delta_0+\varepsilon_1)^2
			\left[\norm{(\tv_1,\tu _1,\tz_1)}_{H^{j-1}}^2+\Gs\right],\\
			&\left|\intR\left(\frac{\tp_j}{v}+R_{\phi,j}\right)\tu _{j+1}\dd\xi\right|
			\le{\tilde\eta}\norm{\tu _{j+1}}_2^2
			+C_{\tilde\eta}\norm{\tv_1}_2^2+C_{\tilde\eta}(\delta_0+\varepsilon_1)\Gs,\\
			&\left|\dot{X}\intR\bar u_j\tu _{j+1}\dd\xi\right|
			\le{\tilde\eta}\norm{\tu _{j+1}}_2^2+C_{\tilde\eta}\delta_S^{2j+1}|\dot{X}|^2.
		\end{align*}
		Choose ${\tilde\eta}$ so that four such contributions are absorbed by
		the viscous term. This yields, for $j=1,2$,
		\begin{align}
			\frac12\frac{\dd}{\dd t}\norm{\tu _1}_2^2+c\norm{\tu _2}_2^2
			&\le C(\norm{\tv_1}_2^2+\D)
			+C(\delta_0+\varepsilon_1)\Gs+C\delta_S^2|\dot{X}|^2,\notag
			\\
			\frac12\frac{\dd}{\dd t}\norm{\tu _2}_2^2+c\norm{\tu _3}_2^2
			&\le C\left(\norm{\tv_1}_{H^1}^2+\norm{(\tu _2,\tz_2)}_2^2+\D\right)
			+C(\delta_0+\varepsilon_1)\Gs+C\delta_S^2|\dot{X}|^2.
			\label{eq:velocityestimates}
		\end{align}
		\medskip\noindent\textbf{2. The differentiated temperature equation.}
		The corresponding multiplier for 
		\eqref{eq:originalperturbation}$_3$ is $-\tz_{j+1}$ after $j-1$
		differentiations. This gives
		\begin{align}
			&\frac{\cv}{2}\frac{\dd}{\dd t}\norm{\tz_j}_2^2
			+\kappa\intR\frac{\tz_{j+1}^2}{v}\dd\xi \notag\\
			=&\intR\partial_\xi^{j-1}\left[p\tu _1+(p-\bar p)\bar u_1\right]
			\tz_{j+1}\dd\xi\notag
			-\kappa\intR R_{\theta,j}\tz_{j+1}\dd\xi\notag
			\\
			&-\mu\intR\partial_\xi^{j-1}
			\left[\frac{\tu _1^2+2\bar u_1\tu _1}{v}+\bar u_1^2\mathfrak q\right]\tz_{j+1}\dd\xi-\cv\dot{X}\intR\bar\theta_j\tz_{j+1}\dd\xi.
			\label{eq:thermalidentity}
		\end{align}
		A direct calculation gives
		\begin{align}
			&\norm{p\tu _1+(p-\bar p)\bar u_1}_2^2
			\le C\norm{\tu _1}_2^2+C\delta_S^2\Gs,\notag
			\\
			&\norm{\partial_\xi[p\tu _1+(p-\bar p)\bar u_1]}_2^2
			\le C\norm{\tu _2}_2^2
			+C(\delta_0+\varepsilon_1)^2
			\left(\norm{(\tv_1,\tu _1,\tz_1)}_2^2+\Gs\right),\label{eq:pressureworkbounds}\\
			&\norm{\frac{\tu _1^2+2\bar u_1\tu _1}{v}+\bar u_1^2\mathfrak q}_2^2
			\le C(\delta_0+\varepsilon_1)^2\left(\norm{\tu _1}_2^2+\Gs\right),\notag
			\\
			&\norm{\partial_\xi\left[\frac{\tu _1^2+2\bar u_1\tu _1}{v}+\bar u_1^2\mathfrak q\right]}_2^2
			\le C(\delta_0+\varepsilon_1)^2
			\left(\norm{\tu _1}_{H^1}^2+\norm{\tv_1}_2^2+\Gs\right).
			\label{eq:heatingbounds}
		\end{align}
		Applying Young's inequality to the four integrals
		in \eqref{eq:thermalidentity} with \eqref{eq:pressureworkbounds} and \eqref{eq:heatingbounds}, one has
		\begin{align*}
			&\left|\intR\partial_\xi^{j-1}[p\tu _1+(p-\bar p)\bar u_1]\tz_{j+1}\dd\xi\right|
			\\
			\le&{\tilde\eta}\norm{\tz_{j+1}}_2^2+C_{\tilde\eta}\norm{\tu _j}_2^2+C_{\tilde\eta}(\delta_0+\varepsilon_1)^2
			\left(\norm{(\tv_1,\tu _1,\tz_1)}_{H^{j-1}}^2+\Gs\right),\\
			&\left|\kappa\intR R_{\theta,j}\tz_{j+1}\dd\xi\right|
			\le{\tilde\eta}\norm{\tz_{j+1}}_2^2
			+C_{\tilde\eta}(\delta_0+\varepsilon_1)^2
			\left(\norm{(\tv_1,\tu _1,\tz_1)}_{H^{j-1}}^2+\Gs\right),\\
			&\left|\mu\intR\partial_\xi^{j-1}
			\left[\frac{\tu _1^2+2\bar u_1\tu _1}{v}+\bar u_1^2\mathfrak q\right]\tz_{j+1}\dd\xi\right|
			\le{\tilde\eta}\norm{\tz_{j+1}}_2^2
			+C_{\tilde\eta}(\delta_0+\varepsilon_1)^2
			\left(\norm{(\tv_1,\tu _1,\tz_1)}_{H^{j-1}}^2+\Gs\right),\\
			&\left|\cv\dot{X}\intR\bar\theta_j\tz_{j+1}\dd\xi\right|
			\le{\tilde\eta}\norm{\tz_{j+1}}_2^2+C_{\tilde\eta}\delta_S^{2j+1}|\dot{X}|^2.
		\end{align*}
		After absorption and using \eqref{eq:smallness}, the first
		derivative estimate is
		\begin{align}
			\frac{\cv}{2}\frac{\dd}{\dd t}\norm{\tz_1}_2^2+c\norm{\tz_2}_2^2
			\le C\D+C(\delta_0+\varepsilon_1)^2\norm{\tv_1}_2^2+C(\delta_0+\varepsilon_1)\Gs+C\delta_S^2|\dot{X}|^2.
			\label{eq:thermalone}
		\end{align}
		At the second derivative level, the term $p\tu _2$ uses the velocity
		second derivative already present in \eqref{eq:velocityestimates}. We obtain
		\begin{align}
			\frac{\cv}{2}\frac{\dd}{\dd t}\norm{\tz_2}_2^2+c\norm{\tz_3}_2^2
			\le&C\norm{\tu _2}_2^2
			+C(\delta_0+\varepsilon_1)^2
			\left[\norm{\tv_1}_{H^1}^2+\norm{(\tu _2,\tz_2)}_2^2+\D\right]\notag
			\\
			&+C(\delta_0+\varepsilon_1)\Gs+C\delta_S^2|\dot{X}|^2.
			\label{eq:thermaltwo}
		\end{align}
		Adding \eqref{eq:velocityestimates} to \eqref{eq:thermalone} and
		\eqref{eq:thermaltwo}, respectively, proves \eqref{eq:Lone} and
		\eqref{eq:Ltwo}. 
	\end{proof}
	
	\subsection{Closure and stability of the single shock}
	\label{sec:closure}
	\subsubsection{\texorpdfstring{The closed a priori estimate}{The closed a priori estimate}}
	\begin{proof}[Proof of Proposition~\ref{thm:closure}]\leavevmode\par
		\medskip\noindent\textbf{1. Choice of the energy coefficients.}
		After the constants in Lemmas~\ref{lem:density} and
		\ref{lem:parabolic} have been fixed, choose a fixed $0<{\eta_{\mathrm E}}\ll1$ and set
		\begin{equation}\notag
			\E_2=\E+{\eta_{\mathrm E}}\mathcal K_1+{\eta_{\mathrm E}}^2\mathcal L_1
			+{\eta_{\mathrm E}}^3\mathcal K_2+{\eta_{\mathrm E}}^4\mathcal L_2.
		\end{equation}
		By \eqref{eq:coercive}, \eqref{eq:Kcoercivity}, and
		\eqref{eq:ellipticHk}, reducing ${\eta_{\mathrm E}}$ if necessary gives
		\begin{align}
			C_{\eta_{\mathrm E}}^{-1}\norm{\tv,\tu ,\tz}_{H^2}^2
			\le \E_2\le C_{\eta_{\mathrm E}}\norm{\tv,\tu ,\tz}_{H^2}^2.
			\label{eq:totalcoercivity}
		\end{align}
		\medskip\noindent\textbf{2. Absorption of every derivative term.}
		Add \eqref{eq:zeroforclosure}, ${\eta_{\mathrm E}}$ times \eqref{eq:Kone},
		${\eta_{\mathrm E}}^2$ times \eqref{eq:Lone}, ${\eta_{\mathrm E}}^3$ times \eqref{eq:Ktwo},
		and ${\eta_{\mathrm E}}^4$ times \eqref{eq:Ltwo}. The resulting inequality, before
		absorption, is
		\begin{align}
			\dot{\E}_2&+c\left[\delta_S|\dot{X}|^2+\Gs+\G_1+\G_2+\D\right]
			+c{\eta_{\mathrm E}}\left(\norm{\tv_1}_2^2+\norm{\tp_1}_{H^1}^2\right)\notag
			\\
			&+c{\eta_{\mathrm E}}^2\norm{(\tu _2,\tz_2)}_2^2
			+c{\eta_{\mathrm E}}^3\left(\norm{\tv_2}_2^2+\norm{\tp_2}_{H^1}^2\right)
			+c{\eta_{\mathrm E}}^4\norm{\tu _3,\tz_3}_2^2\notag
			\\
			\le&C{\eta_{\mathrm E}}\D+C{\eta_{\mathrm E}}(\delta_0+\varepsilon_1)\Gs
			+C{\eta_{\mathrm E}}\delta_S^2|\dot{X}|^2+C\left(\sqrt{\delta_0}+\varepsilon_1+{\eta_{\mathrm E}}^3\right)\norm{(\tu _2,\tz_2)}_2^2\notag
			\\
			&+C\left(\sqrt{\delta_0}+\varepsilon_1+{\eta_{\mathrm E}}^2
			+{\eta_{\mathrm E}}^3(\delta_0+\varepsilon_1)\right)\norm{\tv_1}_2^2
			+C{\eta_{\mathrm E}}^4\norm{\tv_2}_2^2.
			\label{eq:closurebeforeabsorption}
		\end{align}
		First make $C{\eta_{\mathrm E}}$ smaller than the fixed positive constants.
		Then choose $\delta_0,\varepsilon_1$ so that
		$C(\sqrt{\delta_0}+\varepsilon_1)\ll{\eta_{\mathrm E}}^4$.
		Each term on the right of \eqref{eq:closurebeforeabsorption} is thereby
		absorbed by its corresponding term on the left. Hence
		\begin{equation}\label{eq:closeddifferential}
			\begin{aligned}
				\dot{\E}_2+c_{\eta_{\mathrm E}}\big[&\delta_S|\dot{X}|^2+\Gs+\G_1+\G_2
				+\norm{\tv_1}_{H^1}^2+\norm{\tu _1,\tz_1}_{H^2}^2
				+\norm{\tp_1}_{H^2}^2\big]\le0.
			\end{aligned}
		\end{equation}
		Integrate \eqref{eq:closeddifferential}, use
		\eqref{eq:totalcoercivity} at times $0$ and $t$, and apply
		\eqref{eq:ellipticHk} and \eqref{eq:chiH2}. This proves
		\eqref{eq:closedestimate}. 
	\end{proof}
	
	\subsubsection{\texorpdfstring{Continuation and asymptotic stability}{Continuation and asymptotic stability}} By local existence and the a priori estimates in Proposition \ref{thm:closure} and \eqref{eq:shiftbound}, following the continuation and asymptotic arguments in \cite{KKS,KVW}, we have the following result; the details of its proof are omitted for simplicity of presentation. 
	 
	\begin{proposition}\label{prop:stability}
		Under the hypotheses of Theorem \ref{thm:AC},
		the local strong solution extends globally and satisfies
		\eqref{eq:closedestimate} for all $t\ge0$.
		Moreover,
		\begin{equation}\label{eq:asymptotic}
			\norm{(v,u,\theta)(t,\cdot)-(v^S,u^S,\theta^S)(\cdot-X(t))}_\infty
			+\norm{\phi(t,\cdot)-\phi^S(\cdot-X(t))}_{W^{1,\infty}}
			\longrightarrow0,
		\end{equation}
		\begin{equation}\label{eq:shiftasymptotic}
			|\dot{X}(t)|\longrightarrow0,\qquad
			\frac{|X(t)|}t\longrightarrow0\qquad(t\to\infty).
		\end{equation}
	\end{proposition}
	
	From~\eqref{eq:closedestimate}, \eqref{eq:asymptotic},
	and \eqref{eq:shiftasymptotic}, we have proved
	Theorem \ref{thm:AC}.\qed

	\appendix
	
	\section{Center manifolds and small complete orbits}\label{app:center}
	The following finite-dimensional form suffices here; see \cite[Chapter 1]{Carr}.
	\begin{theorem}\label{thm:center}
		Let $r\geq2$ be finite and consider $Z'=LZ+\mathcal N_{\mathrm{cm}}(Z)$, where
		$\mathcal N_{\mathrm{cm}}\in C^r$, $\mathcal N_{\mathrm{cm}}(0)=D\mathcal N_{\mathrm{cm}}(0)=0$. Suppose
		\begin{equation}\notag
			\R^n=E^c\oplus E^s\oplus E^u,\qquad
			L|_{E^c}=0,\qquad
			\operatorname{Re}\operatorname{spec}(L|_{E^s})<0,
			\quad \operatorname{Re}\operatorname{spec}(L|_{E^u})>0.
		\end{equation}
		There is a local $C^r$ invariant graph
		\begin{equation*}
			Z_h=\Psi(Z_c),\qquad Z_h\in E^s\oplus E^u,
			\qquad \Psi(0)=D\Psi(0)=0.
		\end{equation*}
		Every solution that stays in a sufficiently small fixed neighborhood
		for all real times lies on this graph.
	\end{theorem}
	\begin{proof}
		The local graph, its regularity and tangency follow from the
		center-manifold theorem. We verify the containment assertion.
		Straighten the graph by $z=Z_h-\Psi(Z_c)$. Invariance gives
		\begin{equation*}
			\begin{gathered}
				z'=L_hz+R(Z_c,z),\quad R(Z_c,0)=0,
				\quad |R(Z_c,z)|\leq\ell_{\mathrm{cm}}|z|,
				\quad L_h=L|_{E^s\oplus E^u},
			\end{gathered}
		\end{equation*}
		where $\ell_{\mathrm{cm}}$ can be made arbitrarily small by shrinking the
		neighborhood. Let $P_s,P_u$ be the hyperbolic spectral projections.
		Choose $M_{\mathrm{cm}},\omega>0$ such that
		\begin{equation*}
			\|e^{L_ht}P_s\|\leq M_{\mathrm{cm}}e^{-\omega t},\qquad
			\|e^{-L_ht}P_u\|\leq M_{\mathrm{cm}}e^{-\omega t},\qquad t\geq0.
		\end{equation*}
		For a bounded complete orbit, variation of constants and the limits
		at the appropriate infinite ends give
		\begin{equation}\label{eq:bounded-orbit-integrals}
			\begin{aligned}
				z_s(t)&=\int_{-\infty}^t e^{L_h(t-a)}P_sR(Z_c(a),z(a))\dd a,\\
				z_u(t)&=-\int_t^{\infty}e^{L_h(t-a)}P_uR(Z_c(a),z(a))\dd a.
			\end{aligned}
		\end{equation}
		Thus $\|z\|_\infty\leq(2M_{\mathrm{cm}}\ell_{\mathrm{cm}}/\omega)\|z\|_\infty$.
		Choosing $2M_{\mathrm{cm}}\ell_{\mathrm{cm}}<\omega$ forces $z\equiv0$.
	\end{proof}

\section{Sharp diffusion and effective-pressure estimates}\label{app:sharp}

This appendix derives the sharp diffusion estimate used in Section \ref{sec:leading} from the scalar normal form established in Section \ref{sec:existence}. The effective pressure $\pp_{\mathrm{eff}}=(\theta+1)/v$ enters the leading diffusion coefficient through the characteristic speed $\sigma^*$ and the genuine-nonlinearity coefficient $\alpha^*$ defined in \eqref{eq:constants}. We use the normalization in \eqref{eq:profile-identification}, namely $v_-=1$ and $\theta_-=\theta_0$.
	
\begin{lemma}\label{lem:sharp}
		For the small shock $(v^S,u^S,\theta^S,\phi^S)$ in
		\eqref{eq:shock}--\eqref{eq:profile-identification}
		and  shift $X(t):[0,T]\to\R$ defined in \eqref{eq:shiftODE}, set
		\[
		\delta_S=v_+-v_-,\qquad
		\bar v(t,\xi)=v^S(\xi-X(t)),\qquad
		y=\frac{\bar v-v_-}{\delta_S}\in(0,1).
		\]
		With $\alpha^*$ defined in \eqref{eq:constants}, there exists
		$C>0$, such that
		\begin{equation}\label{eq:sharpdiffusion}
			\left|
			\frac{\mu}{\bar v}\frac{y_\xi}{y(1-y)}
			-\alpha^*\frac{\mu(\gamma\theta_-+1)}
			{\mu(\gamma\theta_-+1)+\kappa(\gamma-1)^2\theta_-}
			\delta_S
			\right|
			\le C\delta_S^2.
		\end{equation}
	\end{lemma}
	
	\begin{proof}
		Recall that $\eta=\rho^S-1$ in \eqref{eq:phi-monotonicity}
		and $\rho_+=1-\varepsilon$ in Lemma~\ref{lem:Hugoniot}.
		By \eqref{eq:scalar-profile} and \eqref{eq:factorisation},
		\[
		\eta_\zeta=\eta(\eta+\varepsilon)b(\eta,\varepsilon).
		\]
		As shown in the proof of Lemma~\ref{lem:decay}, $b$ is $C^{r-2}$
		on a fixed neighborhood of the origin and $b(0,0)=\Gamma/D$, where $\Gamma$ and $D$
		are defined in \eqref{def.Gamma.ad} and \eqref{eq:small-eigenvalue}.
		Set
		\[
		\eta_X(t,\xi)
		=\eta\bigl(\zeta(\xi-X(t))\bigr).
		\]
		Since $-\varepsilon<\eta_X<0$ along the profile, it follows that
		\[
		\left|b(\eta_X,\varepsilon)-\frac{\Gamma}{D}\right|
		\le C(|\eta_X|+\varepsilon)
		\le C\varepsilon.
		\]
		By \eqref{eq:speed} and \eqref{eq:profile-identification},
		\[
		\bar v=\frac{1}{\rho^S(\zeta(\xi-X(t)))}=\frac1{1+\eta_X},\qquad
	\frac{\dd}{\dd\xi}=\bar v\frac{\dd}{\dd\zeta},\qquad
	v_+=\frac1{1-\varepsilon},\qquad
	\delta_S=\frac{\varepsilon}{1-\varepsilon}.
		\]
		Using $\bar v-v_-=-\bar v\eta_X$ and
		$v_+-\bar v=\bar v v_+(\eta_X+\varepsilon)$, we obtain
		\begin{equation}\label{eq:volumefactorisation}
			\bar v_\xi=-\bar v^3\eta_X(\eta_X+\varepsilon)b(\eta_X,\varepsilon)
			=\frac{\bar v}{v_+}b(\eta_X,\varepsilon)
			(\bar v-v_-)(v_+-\bar v).
		\end{equation}
		By the definition of $y$,
		\[
		y_\xi=\frac{\bar v_\xi}{\delta_S},\qquad
		y(1-y)=\frac{(\bar v-v_-)(v_+-\bar v)}{\delta_S^2}.
		\]
		Hence
		\[
		\frac{\mu}{\bar v}\frac{y_\xi}{y(1-y)}
		=\frac{\mu\delta_S}{v_+}b(\eta_X,\varepsilon)
		=\mu\varepsilon b(\eta_X,\varepsilon).
		\]
		Since $\varepsilon=\delta_S/(1+\delta_S)$, this yields
		\[
		\frac{\mu}{\bar v}\frac{y_\xi}{y(1-y)}
		=\frac{\mu\Gamma}{D}\delta_S+O(\delta_S^2).
		\]
		Under the normalization $v_-=1$ and $\theta_-=\theta_0$,
		\eqref{def.Gamma.ad}, \eqref{eq:small-eigenvalue}, and
		\eqref{eq:constants} give
		\[
	\alpha^*=\Gamma,\qquad
	D=\mu+\frac{\kappa(\gamma-1)^2\theta_-}{\gamma\theta_-+1},
	\qquad
	\frac{\mu\Gamma}{D}
	=\alpha^*\frac{\mu(\gamma\theta_-+1)}
	{\mu(\gamma\theta_-+1)+\kappa(\gamma-1)^2\theta_-}.
	\]
		This proves \eqref{eq:sharpdiffusion}.
	\end{proof}
	
\bigskip
\noindent {\bf Acknowledgment:}\,
The research of Renjun Duan was partially supported by the General Research Fund (Project No.~14303523) from RGC of Hong Kong and also by the grant from the National Natural Science Foundation of China (Project No.~12425109). The authors also acknowledge the use of AI tools. All mathematical statements and proofs were independently verified by the authors, who take full responsibility for the content of the manuscript.

	\medskip
	\noindent{\bf Data availability:} The manuscript contains no associated data.

	\medskip
	\noindent{\bf Conflict of Interest:} The authors declare that they have no conflict of interest.

\end{document}